\documentclass[11pt]{article}

\RequirePackage[OT1]{fontenc}
\RequirePackage[colorlinks,citecolor=blue,urlcolor=blue]{hyperref}

\usepackage{amsmath, amssymb, amsthm,amsfonts} 
\usepackage[english]{babel}
\usepackage{bbm}
\usepackage{graphicx}
\usepackage{url}
\usepackage{epstopdf}

\usepackage[margin=1.15in]{geometry}
\usepackage[toc,page]{appendix}
\usepackage[utf8]{inputenc}
\usepackage[T1]{fontenc}
\usepackage{authblk}
\usepackage{amssymb}
\usepackage{float}
\usepackage{tikz}
\usetikzlibrary{arrows,calc,shapes,decorations.pathreplacing,patterns,chains,shapes.multipart}
\usepackage{hyperref}
\hypersetup{pdftex,colorlinks=true,allcolors=blue}
\usepackage{caption}
\usepackage{pgfplots}
\usepackage{subcaption}
\usepackage{enumitem}

\newenvironment{customlegend}[1][]{%
  \begingroup
  \csname pgfplots@init@cleared@structures\endcsname
  \pgfplotsset{#1}%
}{%
  \csname pgfplots@createlegend\endcsname
  \endgroup
}%
\def\addlegendimage{\csname pgfplots@addlegendimage\endcsname}

\newtheorem{theorem}{Theorem}
\newtheorem{lemma}[theorem]{Lemma}
\newtheorem{proposition}[theorem]{Proposition}
\newtheorem{corollary}[theorem]{Corollary}

\newtheorem*{theorem*}{Theorem}

\theoremstyle{remark}

\newtheorem{remark}{Remark}
\newtheorem*{remark*}{Remark}

\def\P{\mathrm{P}}
\def\E{\mathrm{E}}
\def\Var{\mathrm{Var}}
\def\Cov{\mathrm{Cov}}

\DeclareMathOperator*{\argmax}{arg\,max}

\newcommand{\apost }{'}
\newcommand{\parrow}{\:\to_{\mathbb P}\:}
\newcommand{\darrow}{\:\to_{\rm d}\:}
\newcommand{\asarrow}{\:\to_{\rm as}\:}

\allowdisplaybreaks

\begin{document}

\title{Estimating the input of a L\'evy queue by Poisson sampling of the workload process\footnote{To appear in Bernoulli.}}

\author[1,2]{Liron Ravner}
\author[2]{Onno Boxma}
\author[1]{Michel Mandjes}
\affil[1]{\small{University of Amsterdam}}
\affil[2]{\small{Eindhoven University of Technology}}

\date{\today}
\maketitle

\begin{abstract}
This paper aims at semi-parametrically estimating the input process to a L\'evy-driven queue by sampling the workload process at Poisson times. We construct a
method-of-moments based estimator for the L\'evy process' characteristic exponent. This method exploits the known distribution of the workload sampled at an exponential time, thus taking into account the dependence between subsequent samples. Verifiable conditions for consistency and asymptotic normality are provided, along with explicit expressions for the asymptotic variance. The method requires an intermediate estimation step of estimating a constant (related to both the input distribution and the sampling rate); this constant also features in the asymptotic analysis. For subordinator L\'evy input, a partial MLE is constructed for the intermediate step and we show that it satisfies the consistency and asymptotic normality conditions. For general spectrally-positive L\'evy input a biased estimator is proposed that only uses workload observations above some threshold; the bias can be made arbitrarily small by appropriately choosing the threshold.
\end{abstract}

\section{Introduction}\label{sec:intro}
To optimally design and control queueing systems, it is of great importance to have reliable estimates of the model primitives such as the arrival rate and the service-time distribution. In many situations, however, one cannot observe the users' interarrival times and service times; instead, one only has periodic observations of the workload process. This leaves us with the challenging task of inferring the model primitives from such workload observations. A complicating factor is that subsequent observations are dependent. As a consequence, standard estimation techniques (such as maximum-likelihood estimation) are not applicable in almost all queuing systems. 

This work focuses on the problem described above in the setting of a single-server queue with work arriving according to a {\it spectrally-positive L\'evy} inputs process, i.e., a L\'evy process $X(\cdot)$ with only positive jumps. The resulting model includes the M/G/1 queue as a special case; then the input process is a Compound Poisson (CP) process. The set of spectrally-positive L\'evy input processes, however, is substantially broader than CP: it covers a broad range of other relevant processes, such as Brownian motion, the Gamma process and the Inverse Gaussian process. The L\'evy input process is uniquely defined by its {\it characteristic exponent} function $\varphi(\alpha):=\log\E e^{-\alpha X(1)}$, with $\alpha \ge 0$ \cite{book_K2006}. We assume that $\varphi(\cdot)$ is unknown and it is therefore the objective to estimate this function. For a detailed account of L\'evy-driven queues we refer to \cite{book_DM2015}.

\vspace{2mm}

\noindent {\it Approach.}
In our approach we exploit the fact that, given some initial workload, there is an explicit expression for the Laplace-Stieltjes Transform (LST) of the workload after an exponentially distributed time; see, e.g.\ \cite{KBM2006}. When the workload is sampled at random times according to an independent Poisson process (with some specified rate $\xi$), this LST enables us to deal with the dependence between the subsequent workload observations.In particular, method of moments estimation is carried out by equating the empirical and conditional expected LST, thus obtaining so called {\it Z-estimators} \cite{book_vdV1998}. Our setting can be viewed as non-parametric in the sense that we do not make specific distributional assumptions about the input process except that it is a spectrally-positive L\'evy process. However, we use parametric estimation techniques for pointwise estimation of the exponent function.

There is the complication that, due to the specific form of the above-mentioned LST, it is necessary to first estimate $\varphi^{-1}(\xi)$, i.e., the inverse of $\varphi$ at the sampling rate $\xi$. This intermediate step is also crucial for determining the asymptotic performance of the estimator for $\varphi(\alpha)$. In particular, we show that 
consistency and asymptotic normality in the intermediate step (to estimate $\varphi^{-1}(\xi)$, that is) carry over to consistency and asymptotic normality in the main estimation step (to estimate $\varphi(\alpha)$). For the case of subordinator input (i.e., a L\'evy process whose paths are non-decreasing almost surely) we construct a maximum-likelihood procedure for estimating $\varphi^{-1}(\xi)$, which we show to be consistent and asymptotically normal. This result therefore covers the special cases of the driving L\'evy process being CP, a Gamma process and an Inverse Gaussian process. For general spectrally-positive L\'evy input, i.e., also including a Brownian motion component, we provide an estimation method for $\varphi^{-1}(\xi)$ that uses workload observations only above some threshold; this procedure is inherently biased, but the bias can be controlled by choosing the threshold parameter appropriately. 

Later in this introduction we describe in detail the differences with existing estimation approaches. A major difference is that previous papers tend to aim at using workload observations to estimate quantities pertaining to the steady-state workload distribution, from which the model primitives are then inferred, whereas we explicitly use the knowledge of the dependence between two subsequent workload observations. 

\vspace{2mm}

\noindent {\it Main contributions.}
We now summarize the main contributions of our work.
\begin{itemize}
\item Thus far, the literature has mostly focused on the case of Compound Poisson input, 
whereas we address the larger class of spectrally-positive L\'evy processes. 
\item 
In our approach we do not rely on estimating quantities pertaining to the steady-state workload distribution, but rather explicitly use the system's correlation structure. 
\item We provide verifiable conditions for consistency and asymptotic normality of our estimator (including an expression for the asymptotic variance). These conditions are verified for the case of subordinator input, thus implying the asymptotic properties for the special case of CP. When $\varphi(\alpha)$ is estimated for multiple values of $\alpha$ simultaneously, we provide the corresponding asymptotic covariances.
\item The idea to use a Z-estimator in this L\'evy-driven setting is novel. As mentioned, our approach requires knowledge of $\varphi^{-1}(\xi)$, for which we present an estimation procedure.
\item We present simulation examples that explore the accuracy of the procedure. 
These show that the accuracy is generally good, but degrades for large $\alpha$ (because the asymptotic variance of the error increases with $\alpha$). We also observe that the accuracy is best for low $\xi$, which is to be expected: when $\xi$ is high, the workload process is very frequently sampled, such that new observations offer hardly any new information. As mentioned, for non-subordinator input the estimator of $\varphi^{-1}(\xi)$ is biased, but our experiments show that the accuracy is high, even for a relatively low threshold.
\end{itemize}

\vspace{2mm}

\noindent {\it Background and related literature.}
Observing a system at Poisson times is commonly known as {\it Poisson probing}. In communication networks it has been used to estimate (steady-state) packet delays in a network with unknown traffic input \cite{AJP2013, patent_SM2002}. In a practical implementation of Poisson probing, very small jobs are sent through the system and their delays are recorded. 
This method relies on the {\it Poisson Arrivals See Time Averages} (PASTA) property: the stationary distribution coincides with the distribution observed at Poisson instants; theoretical justification can be found in e.g.\ \cite{GMW1993,Y1992}. 

In the M/G/$1$ case the queue is fed by a Poisson arrival process with rate $\lambda$, whereas the service times are i.i.d.\ random variables $G_i\sim G$ with LST $G^*$. In \cite{BP1999} a non-parametric approach is proposed for estimating $\lambda$ and the LST of $G$ using observations of the lengths of the busy and idle periods. This procedure overcomes the issue of dependent observations as busy periods are independent, but it may require a very long observation period to get a large sample of busy periods, especially in heavily loaded systems.
The estimator exploits the fixed-point equation $B^*(\alpha)=G^*(\alpha+\lambda(1-B^*(\alpha)))$, where $B^*$ is the LST of the busy period, and uses an estimator $\hat{G}_n(\alpha)$ based on the empirical LST of the busy periods (with $n$ observations). The estimator is consistent and asymptotically normal: $\sqrt{n}(\hat{G}_n(\cdot)-G^*(\cdot))$ converges in distribution to a Gaussian process. Reference \cite{HP2006} provides a nonparametric procedure for estimating the service-time distribution, using an equidistant sample of the workload. It consists of two steps: (1) estimating the residual service-time distribution by inverting the convolution workload formula, $F_V=(1-\rho)\sum_{k=0}^\infty \rho^k G_e^{\star k}$, where $F_V$ is the stationary cdf of the workload and $G_e^{\star k}$ is the $k$-fold convolution of the residual service time distribution, and (2) then estimating the service-time distribution itself. As the first step relies on the stationary workload distribution, the procedure does not use any knowledge on the dependence between observations. Nevertheless, the estimation of the residual service-time distribution is shown to be consistent and asymptotically normal under some conditions on the system load. For the second step only heuristic methods are suggested. The numerical experiments in \cite{HP2006} show that the procedure performs well for moderately loaded systems but often performs poorly for systems with low or high load. 

A general framework for generalized method-of-moments (GMM) estimators for Markov processes sampled according to a Poisson process, that is possibly state and time dependent, is provided in \cite{DG2004}. As it turns out, the fact that in our setup the quantity $\varphi^{-1}(\xi)$ has to be estimated entails that we cannot use this approach. In Section \ref{sec:discussion} we further discuss the relation of our work with the GMM-based procedure of \cite{DG2004} in the context of our model.

Non-parametric estimation of the stationary waiting time distribution for a GI/G/$1$ queue was studied in \cite{P1994}. If the number of arrivals during a service time can be observed, then estimation of the service-time distribution reduces to the problem of decompounding Poisson sums, as studied in \cite{BG2003}. As is often the case, when moving from a single-server queue to its infinite-server counterpart, the analysis greatly simplifies. The object of interest in this line of research is typically the service time distribution, and we refer the reader to \cite{BP1999,BNW2013,G2016,G2018,SW2015}. A related approach appeared in \cite{NTV2006} where an estimator for the arrival rate is derived by relying on counting the number of arrivals during a busy period. In \cite{RTP2007} diffusion approximations are used. 

Estimation of the model primitives from queue observations also has important managerial applications, for example in the contexts of demand estimation \cite{AP2005} and dynamic pricing \cite{AA2013}. In \cite{AJPW2016} the traffic intensity of a non-homogeneous-in-time queue was estimated using a (non-Poisson) probing scheme that depends on the service time of the probes. This list of papers on statistical inference of queueing systems is by no means exhaustive; see \cite{ANP2017} for an extensive bibliography of this strand of research. 

In the present paper we estimate (the characteristic exponent of) a L\'evy process by observing a L\'evy-driven queue. Various papers consider the counterpart in which the L\'evy process itself is observed. We refer to e.g.\ \cite{B2011,NR2009}; \cite{G2009} studies the case of a L\'evy process without small jumps. 

\vspace{2mm}

\noindent{\it Paper organization.}
The remainder of the paper is organized as follows. Section \ref{sec:model} introduces the queueing model and sampling scheme. In Section \ref{sec:phi_z_estimator} the estimator for $\varphi(\alpha)$ is defined and conditions for consistency and asymptotic normality are provided. The following two sections deal with the intermediate estimation step of $\varphi^{-1}(\xi)$. In Section \ref{sec:MG1_MLE} an MLE is constructed for the subordinator case and it is shown to satisfy the consistency and asymptotic normality conditions. Section \ref{sec:levy} presents an approach for the intermediate step that relies only on workload observations above some threshold. Further, a bound is derived for the bias and explicitly computed in the limit when the sample size grows. In Section \ref{sec:simulation} we numerically analyze the performance of the estimators using simulated data. 
Section~\ref{sec:discussion} contains conclusions and suggestions for further research.
Finally, proofs and further technical details are provided in the appendices.

\section{Model and preliminaries}\label{sec:model}

We study a queue with a spectrally-positive L\'evy input process $J(t)$. The distribution of the process can be represented by the exponent function,
\begin{equation}
\log\E e^{-\alpha J(1)}=-c\alpha+\frac{1}{2}\sigma ^2 \alpha^2-\int_{(0,\infty)}(1-e^{-\alpha x})\nu({\rm d}x)\ ,
\end{equation}
where $c$ and $\sigma$ are non-negative constants and $\nu$ is the L\'evy jump measure. The output of the system is a unit-rate negative linear drift. This means that the net input process is $X(t)=J(t)-t$, which is also a spectrally-positive L\'evy process. For example, if $J(t)$ is CP with rate $\lambda$ and jump size LST $G^*$, then the exponent function of the net input is $\varphi(\alpha)=\lambda\left(G^*(\alpha)-1\right)+\alpha$.

Our objective is to estimate the (unknown) exponent function of the net input process, $\varphi(\alpha):=\log\E e^{-\alpha X(1)}$, with $\alpha\ge 0$. Let $(V(t))_{t\in{\mathbb R}}$ denote the workload process, which can be represented as the net input process reflected at zero: $V(t)=X(t)+\max\{V(0),L(t)\}$, where $L(t):=-\inf_{0\leq s\leq t}X(s)$. If $\E X(1)<0$, then the stationary distribution of $V:=V(\infty)$ is given by the {\it generalized Pollaczek-Khintchine} (GPK) formula \cite[p.\ 27]{book_DM2015}:
\begin{equation}\label{eq:GPK}
\E e^{-\alpha V}=\frac{\alpha \varphi\apost (0)}{\varphi(\alpha)}\ .
\end{equation} 
Furthermore, if $\sigma=0$ then the input process $J(t)$ is a subordinator (i.e., an almost surely non-decreasing L\'evy process), and the stationary probability that the workload is 0 equals 
\begin{equation}\label{eq:P_idle}
\P(V=0)=\varphi\apost (0)=-\E X(1)\ .
\end{equation}
This is the is case if, for example, the input is CP (as in the M/G/$1$ model). Importantly,  however, it also holds for spectrally-positive L\'evy input without a Brownian component, i.e., a L\'evy jump process with non-negative jumps (covering e.g.\ the Gamma process and the Inverse Gaussian process). 

All asymptotic results presented in this paper require that $\E X(1)<0$, but 
in principle the underlying estimation procedure can still be performed when this assumption is not valid. In relation to the case of an unstable queue (i.e., if $\E X(1)\geq 0$), it is noted that, while the estimators are anticipated to perform reasonably well, other estimation procedures are likely to be more appropriate due to the fact that (after an initial transient period) one effectively observes the (non-reflected) net input process $(X(t))_{t\in{\mathbb R}}$.

From now on we assume that the linear drift component $c$ of the input process $(J(t))_{t\in{\mathbb R}}$ is known. Without loss of generality we normalize this rate to $c=0$, so that the linear drift of the net input process $(X(t))_{t\in{\mathbb R}}$ is $-1$. If both of the net input process and net output rate are unknown then there are evident problems of identification, and the estimation problem is not well defined; such cases are outside the scope of this paper.

Suppose that the workload is observed at $n$ random times, sampled according to an independent Poisson process with rate $\xi$. Denote these times by $\{T_i\}_{i=1}^n$ such that $T_i-T_{i-1}$ are i.i.d.\ samples from an exponential distribution with parameter $\xi$ (putting $T_0:=0$). The sampling procedure is illustrated for the workload of an M/G/$1$ queue in Figure \ref{fig:workload}.

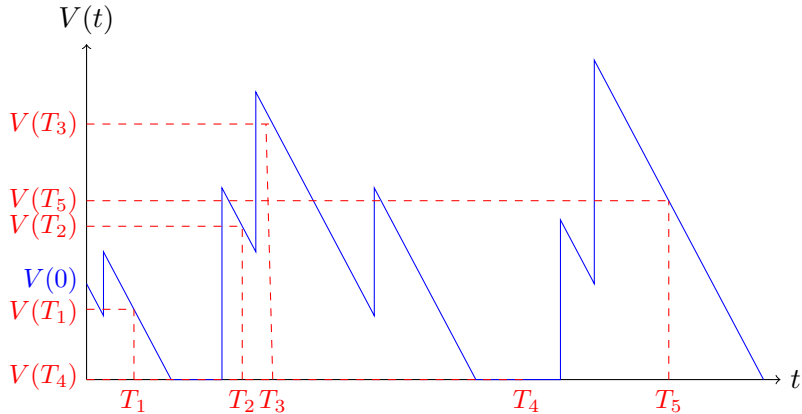
\begin{figure}
\centering
\begin{tikzpicture}[x=0.45cm,y=0.85cm]
 \def\xmin{0}
 \def\xmax{20.5}
 \def\ymin{0}
 \def\ymax{5.25}
 \draw[->] (\xmin,\ymin) -- (\xmax,\ymin) node[right] {$t$} ;
 \draw[->] (\xmin,\ymin) -- (\xmin,\ymax) node[above] {$V(t)$} ;
\foreach \x in {} {
 \node at (\x,\ymin) [below] {\x};
 \draw[-] (\x,\ymin) -- (\x,{\ymin-0.05});
 }
\foreach \y in {} {
 \node at (\xmin,\y) [left] {\y};
 \draw[-] (\xmin,\y) -- ({\xmin-0.05},\y);
 }
\draw[- ,blue] (0,1.5) -- (0.5,1) -- (0.5,2) -- (2.5,0) -- (4,0) -- (4,3) -- (5,2) -- (5,4.5) -- (8.5,1) -- (8.5,3) -- (11.5,0) -- (14,0) -- (14,2.5) -- (15,1.5) -- (15,5) -- (20,0);

 \node[draw=none,left ,color=blue] at (0.05,1.55) {\small{$V(0)$}};

 \draw[-,dashed,red] (1.4,0) -- (1.4,1.1);
 \draw[-,dashed,red] (0,1.1) -- (1.4,1.1);
 \node[draw=none,below,color=red] at (1.4,0) {\small{$T_1$}};
 \node[draw=none,left,color=red] at (0.05,1.05) {\small{$V(T_1)$}};

 \draw[-,dashed,red] (4.6,0) -- (4.6,2.4);
 \draw[-,dashed,red] (0,2.4) -- (4.6,2.4);
 \node[draw=none,below,color=red] at (4.6,0) {\small{$T_2$}};
 \node[draw=none,left,color=red] at (0.05,2.4) {\small{$V(T_2)$}};

 \draw[-,dashed,red] (5.5,0) -- (5.3,4);
 \draw[-,dashed,red] (0,4) -- (5.5,4);
 \node[draw=none,below,color=red] at (5.5,0) {\small{$T_3$}};
 \node[draw=none,left,color=red] at (0.05,4) {\small{$V(T_3)$}};

 \draw[-,dashed,red] (0,0) -- (13,0);
 \node[draw=none,below,color=red] at (13,0) {\small{$T_4$}};
 \node[draw=none,left,color=red] at (0.05,0.05) {\small{$V(T_4)$}};

  \draw[-,dashed,red] (17.2,0) -- (17.2,2.8);
 \draw[-,dashed,red] (0,2.8) -- (17.2,2.8);
 \node[draw=none,below,color=red] at (17.2,0) {\small{$T_5$}};
 \node[draw=none,left,color=red] at (0.05,2.8) {\small{$V(T_5)$}};
 
 \end{tikzpicture}
\caption{The workload $V(t)$ of an M/G/$1$ queue is sampled according to the probing process at times $T_1,T_2,\ldots$. The jump sizes and interarrival times are not observed. Inference is made based on the sample $(V(T_1),V(T_2),\dots)$.} \label{fig:workload}
\end{figure}

The distribution of the workload at $T_i$ conditional on the workload at sample $T_{i-1}$ is 
\begin{equation}\label{eq:V_LST_transient}
\E\left[e^{-\alpha V(T_i)}|V(T_{i-1})\right]=\frac{\xi}{\xi-\varphi(\alpha)}\left(e^{-\alpha V(T_{i-1})}-\frac{\alpha}{\psi(\xi)}e^{-\psi(\xi)V(T_{i-1})}\right)\ ,
\end{equation}
see, e.g.\ \cite{KBM2006} and \cite[Ch. IV]{book_DM2015};
here $\psi$ is the inverse of $\varphi$. For $\alpha=\psi(\xi)$, the LST of  \eqref{eq:V_LST_transient} can still be computed by applying L'H\^opital's rule:
\begin{equation}\label{eq:V_LST_transient_psi}
\E\left[e^{-\psi(\xi) V(T_i)}|V(T_{i-1})\right]=\frac{\xi e^{-\psi(\xi) V(T_{i-1})}\left(V(T_{i-1})+\frac{1}{\psi(\xi)}\right)}{\varphi'(\psi(\xi))}\ .
\end{equation}
Noting that $\varphi(0)=0$, it can be verified that the conditional expected workload at observation time $T_i$ is
\begin{equation}\label{eq:EV_transient}
\E\left( V(T_i)|V(T_{i-1})\right)=V(T_{i-1})+\frac{e^{-\psi(\xi)V(T_{i-1})}}{\psi(\xi)}-\frac{\varphi\apost (0)}{\xi}\ ,
\end{equation}
and the conditional second moment is
\begin{equation}\label{eq:EV2_transient}
\E\left[ V(T_i)^2|V(T_{i-1})\right]=V(T_{i-1})^2-2\frac{\varphi\apost (0)}{\xi}\left(V(T_{i-1})+\frac{e^{-\psi(\xi)V(T_{i-1})}}{\psi(\xi)}-\frac{\varphi\apost (0)}{\xi}\right)+\frac{\varphi{''}(0)}{\xi}\ .
\end{equation}
The above holds for any L\'evy-driven queue with spectrally-positive input. 
In the {subordinator} case, essentially due to the fact that the workload process spends time at zero with positive probability, by taking $\alpha\to\infty$ in \eqref{eq:V_LST_transient},
it turns out that
\begin{equation}\label{eq:P_idle_transient}
\P(V(T_i)=0|V(T_{i-1})=v)=\frac{\xi}{\psi(\xi)}e^{-\psi(\xi)v}\ .
\end{equation}
The above results are in terms of LSTs;
the joint distribution of the sample can in principle be evaluated by numerical inversion of the LSTs involved, but this is computationally demanding. Therefore, in the sequel we use a method-of-moments type estimator that requires finding the root of a function equating the empirical LST or the empirical moments to their respective expectations. 
The term $\psi(\xi)=\varphi^{-1}(\xi)$, which depends on the net input process as well as the sampling rate, appears in the above expressions. Estimating this quantity plays a crucial role in this work.

\vspace{2mm}

\noindent {\it Notation.}
Much of this work involves large sample asymptotic analysis, and as such we will frequently make use of the following notations: ${\parrow}$ for convergence in probability, ${\asarrow}$ for almost sure convergence, ${\darrow}$ for convergence in distribution, and $\approx$ to indicate that two random sequences have the same limit in probability.

\section{Semiparametric estimation of $\varphi$}\label{sec:phi_z_estimator}

This section proposes a Z-estimator for $\varphi(\alpha)$ for a given $\alpha>0$. We then study the asymptotic properties of the estimator, in that we provide conditions for consistency and asymptotic normality given we know $\psi(\xi)$; the estimation of $\psi(\xi)$ is the topic of the next section. For brevity, we throughout denote $V_i:=V(T_i)$.

Given a sample of workload observations at Poisson times, $(V_1,\ldots,V_n)$, let $\psi_n$ denote a series of estimators for $\psi(\xi)$. In particular, $\psi_n$ is a function $\mathbb{R}_+^n\to\mathbb{R}_+$ of the first $n$ observations $(V_1,\ldots,V_n)$. We do not make any assumptions on $\psi_n$ for now. Later, however, we impose conditions on this sequence that enable the evaluation of the  asymptotic performance of the estimators, and in addition a specific construction for the special case of subordinator input. 

An estimator for $\varphi(\alpha)$ is derived by equating the conditional LST \eqref{eq:V_LST_transient} and the empirical LST and solving
\begin{equation}\label{eq:phi_z_equation}
\frac{1}{n}\sum_{i=1}^n e^{-\alpha V_i}=\frac{1}{n}\sum_{i=1}^n \frac{\xi}{\xi-\varphi(\alpha)}\left(e^{-\alpha V_{i-1}}-\frac{\alpha}{\psi_n}e^{-\psi_n V_{i-1}}\right)\ .
\end{equation}
Letting
\begin{equation}\label{eq:phi_z_alpha}
Z_i(\alpha):=e^{-\alpha V_i}- \frac{\xi}{\xi-\varphi(\alpha)}\left(e^{-\alpha V_{i-1}}-\frac{\alpha}{\psi_n}e^{-\psi_n V_{i-1}}\right)\ ,
\end{equation}
the estimator is given by the root of $\frac{1}{n}\sum_{i=1}^n Z_i(\alpha)$. We rearrange \eqref{eq:phi_z_equation} to
\begin{equation}
\varphi(\alpha)=\xi\left[1-\left.{\frac{1}{n}\sum_{i=1}^n\left(e^{-\alpha V_{i-1}}-\frac{\alpha}{\psi_n}e^{-\psi_nV_{i-1}}\right)}\right/{\frac{1}{n}\sum_{i=1}^n e^{-\alpha V_i}}\right]\ ,
\end{equation}
yielding the estimator
\begin{equation}\label{eq:phi_z_estimator}
\hat{\varphi}_n(\alpha;\psi_n)=\frac{\xi\alpha}{\psi_n}\left.\left({\frac{\psi_n}{\alpha n}\left(e^{-\alpha V_n}-e^{-\alpha V_0}\right)+\frac{1}{n}\sum_{i=1}^n e^{-\psi_n V_{i-1}}}\right)\right/{\frac{1}{n}\sum_{i=1}^n e^{-\alpha V_{i}}}\ .
\end{equation}
This type of estimator is generally referred to as a Z-estimator \cite[Ch. V]{book_vdV1998}.
Note that in our case the samples are dependent, thus preventing direct use of classical results on consistency and asymptotic normality. Another justification for this estimator is that it coincides with the estimator minimizing the sum of quadratic deviations of the empirical LST from their respective conditional expectations, i.e. minimizing the conditional mean square error (MSE).


\begin{remark}\label{rem:moment_estimation}
This method can also be applied for parametric estimation of the moments. For example, an estimator for the first moment of the net input process, i.e., $\theta:=\E X(1)$, can be derived. This is done by taking the derivative in \eqref{eq:phi_z_estimator}:
\begin{equation}\label{eq:phi_d_hat}\theta_n=
-\hat{\varphi}\apost _n(0;\psi_n)=-\lim_{\alpha\downarrow 0}\frac{{\rm d}}{{\rm d}\alpha}\hat{\varphi}_n(\alpha;\psi_n)=-\frac{\xi}{n\psi_n}\sum_{i=1}^n e^{-\psi_n V_{i-1}}\ .
\end{equation}
Alternatively, $\hat{\varphi}\apost _n(0;\psi_n)$ could be estimated directly from \eqref{eq:EV_transient}; then the estimation equation would also include the term $({V_n-V_0})/{n}$. Clearly, as $n$ grows both estimators coincide. The same procedure can be followed for the (joint) estimation of higher moments. The estimators of $k$-th moment $\E (X(1)^k)$ follows from
\begin{equation}
\theta_n^{(k)}=(-1)^{k}\hat{\varphi}^{(k)}_n(0;\psi_n)\ .
\end{equation} 
\end{remark}

\begin{remark}\label{rem:alpha=psi_xi}
Recall that at $\alpha=\psi(\xi)$ the conditional LST $\E\left[e^{-\alpha V(T_i)}|V(T_{i-1})\right]$ is not given directly by \eqref{eq:V_LST_transient} but can be computed using L'H\^opital's rule. 
In practice this case will not play a role, though, as $\psi_n=\psi(\xi)$ will not occur in our setting in which various continuous random variables play a role. Specifically, while $\hat{\varphi}_n(\psi_n;\psi_n)=\xi$, by construction of the Z-estimator, the event $\hat{\varphi}_n(\psi(\xi);\psi_n)=\xi$ has zero probability to occur.
\end{remark}

\subsection{Consistency}\label{sec:phi_consistent}

The main result presented in this subsection is Theorem \ref{thm:phi_consistency}, claiming that consistency of the estimator $\psi_n$ implies consistency of the estimator $\hat{\varphi}_n(\alpha;\psi_n)$ for all $\alpha>0$. The proof relies on first establishing the LLN for the empirical LST using the GPK formula \eqref{eq:GPK} for the stationary distribution of the workload observations at Poisson sampling times. The result is then established by applying Slutsky's lemma and the continuous mapping theorem to the estimator given in \eqref{eq:phi_z_estimator}.

\begin{theorem}\label{thm:phi_consistency}
If $\psi_n{\parrow}\psi(\xi)$ as $n\to\infty$ then for every $\alpha>0$ the pointwise estimator $\hat{\varphi}_n(\alpha;\psi_n)$ is consistent: $\hat{\varphi}_n(\alpha;\psi_n){\parrow}\varphi(\alpha)$ as $n\to\infty$.
\end{theorem}

\begin{proof}
If $\E X(1)<0$, then $(V(t))_{t\in{\mathbb R}}$ has a stationary distribution. By PASTA the limit distribution at Poisson sampling instants equals the stationary distribution of $(V(t))_{t\in{\mathbb R}}$, which we denote by $V$.
Denote the Poisson sampling process by $N(t)$, with $t_n:=\inf\{t:\ N(t)=n-1\}$. The choice of $n-1$ in the definition of the counting process is made for the sake of brevity in the following analysis due to the observations always appearing as $V_{i-1}$ in the estimation equation \eqref{eq:phi_z_equation}. Then, as $t_n\to\infty$ when $n\to\infty$, Equation \eqref{eq:GPK} implies that for any $\beta\geq 0$,
\begin{equation}\label{eq:PK_Vi}
\frac{1}{n}\sum_{i=1}^n e^{- \beta V_{i-1}}=\frac{1}{N(t_n)}\sum_{i=1}^{N(t_n)} e^{- \beta V(t_i)}{\parrow} \E e^{-\beta V}=\frac{\beta\varphi\apost (0)}{\varphi(\beta)}\ ,
\end{equation} 
as $n\to\infty$.
We define the right-continuous process 
\begin{equation}
H(t):=\sum_{n=1}^\infty \mathbf{1}(t_{n}\leq t<t_{n+1})\psi_n, \ t\geq 0\ ,
\end{equation}
and observe that $H(t_n)=\psi_n$ for every $n\geq 1$. By Slutsky's lemma \cite[p.\ 11]{book_vdV1998}, $\psi_n V(t_n)=H(t_n)V(t_n)\parrow\psi(\xi) V$ as $n\to\infty$. Hence, applying PASTA once more,
\begin{equation}\label{eq:PK_Vi}
\frac{1}{N(t_n)}\sum_{i=1}^{N(t_n)} e^{- H(t_n) V(t_i)}{\parrow} \frac{\psi(\xi)\varphi\apost (0)}{\varphi(\psi(\xi))}\ .
\end{equation} 
Furthermore, the term $\frac{\psi_n}{\alpha n}\left(e^{-\alpha V_n}-e^{-\alpha V_0}\right)$ converges to zero almost surely if the stability condition $\E X(1)<0$ holds. Next we can use the fact that $\varphi(\psi(\xi))=\varphi(\varphi^{-1}(\xi))=\xi$ and apply the continuous mapping theorem (see \cite{book_vdV1998}, p. 7) to \eqref{eq:phi_z_estimator}, thus obtaining
\begin{equation}
\hat{\varphi}_n(\alpha;\psi_n){\parrow}\left.\left({\alpha\xi\frac{\varphi^{-1}(\xi)\varphi\apost (0)}{\varphi(\varphi^{-1}(\xi))}}\right)\right/\left({\varphi^{-1}(\xi)\frac{\alpha\varphi\apost (0)}{\varphi(\alpha)}}\right)=\varphi(\alpha)\ .
\end{equation}
Thus, the pointwise estimator of $\varphi(\alpha)$ is consistent for every $\alpha>0$.
\end{proof}

\begin{remark}
Consider the case of parametric estimation of the (first and higher) moments, as described in Remark~\ref{rem:moment_estimation}. Using the same arguments it is straightforward to show that the estimators are consistent if $\psi_n$ is consistent. 
\end{remark}

\subsection{Asymptotic normality of $\hat{\varphi}_n$}\label{sec:phi_normal}
In this section we establish, under specific conditions, asymptotic normality of our estimator $\hat{\varphi}_n(\alpha;\psi_n)$ as $n\to\infty$: we show that $\sqrt{n}(\hat{\varphi}_n(\alpha;\psi_n)-\varphi(\alpha))$ converges to a zero-mean normal random variable, for any fixed $\alpha>0$. 
We find an expression for the corresponding asymptotic variance, so that one can assess the estimation error for large samples. 

We then extend this result to a multivariate setting: we establish the joint normality of the vector 
\begin{equation}
\sqrt{n}\left(\hat{\boldsymbol \Phi}_n({\boldsymbol\alpha})-\Phi({\boldsymbol\alpha})\right)\ ,
\end{equation}
where $\boldsymbol \alpha:=(\alpha_1,\ldots,\alpha_p)$, $\Phi({\boldsymbol\alpha}):=(\varphi(\alpha_1),\ldots,\varphi(\alpha_p))$ for some $p\in\mathbb{N}$, and $\hat{\boldsymbol \Phi}_n({\boldsymbol\alpha}):=(\hat{\varphi}_n(\alpha_1;\psi_n),\ldots,\hat{\varphi}_n(\alpha_p;\psi_n))$. Each of the estimation equations \eqref{eq:phi_z_equation} is solved independently for $\alpha_i$, but they rely on the same sample and are therefore dependent. 

We start the section by stating a Central Limit Theorem (CLT) for martingale difference sums. Lemma \ref{lemma:MCLT} below consists of two parts; both are essentially known results, that we customized to our needs. The first part is the univariate result of \cite[Thm. 3.2]{book_HH1980}. The second part is a special case of the multidimensional extension of \cite[Thm. 12.6]{book_H1997}; here we remark that the result  in \cite{book_H1997} relates to continuous time, but as we are only interested in the discrete-time case we present a more concise statement without the quadratic variation terms.

All convergence statements in the following are as $n\to\infty$; we therefore omit this specification from now on. 
\begin{lemma}\label{lemma:MCLT}\
\begin{enumerate}
\item[{\rm (a)}] 
Let $M_n$ be a discrete-time martingale with respect to a filtration $\{\mathcal{F}_n\}_{n\geq 0}$, and denote
\begin{equation}
Z_{ni}:=\frac{M_i-M_{i-1}}{\sqrt{\Var(M_n)}},\quad i=1,\ldots,n .
\end{equation}
If
\begin{equation}\label{eq:M_mclt1_finite}
\E\max_{1\leq i\leq n}Z_{ni}^2<\infty,\quad n=1,2,\ldots,
\end{equation}
and
\begin{equation}\label{eq:M_mclt2_maxtozero}
\max_{1\leq i\leq n}|Z_{ni}|{\parrow} 0  ,
\end{equation}
then $
\sum_{i=1}^n Z_{ni} {\darrow} \mathrm{N}(0,1).$
\item[{\rm (b)}] Let \[{\boldsymbol M}_n^\top:=\big(M_{n}^{(1)},\ldots,M_{n}^{(p)}\big)\] be a vector of discrete-time martingales,
and let \[{\boldsymbol Z}_n^\top:=\Big(\sum_{i=1}^n Z_{ni}^{(1)},\ldots,\sum_{i=1}^n Z_{ni}^{(p)}\Big),\]
with $Z_{ni}^{(k)}$, for $k=1,\ldots,p$, the respective normalized differences. If these objects satisfy all conditions of {\rm (a)}, and 
\begin{equation}\label{eq:M_mclt3_M2}
\lim_{n\to\infty} \E\left[{\boldsymbol Z}_n{\boldsymbol Z}_n^\top\right] = S \ ,
\end{equation}
where $S\in\mathbb{R}^{p\times p}$ is a positive-definite matrix with finite elements, then ${\boldsymbol Z}_n{\darrow} \mathrm{N}(0,S)$.
\end{enumerate}
\end{lemma}

\subsubsection{Univariate asymptotic normality}
Recall that the estimation procedure also relies on an external estimator for $\psi(\xi)$. As the CLT below shows, the asymptotic distribution of $\sqrt{n}(\psi_n-\psi(\xi))$ and its joint distribution with the estimator for $\varphi(\alpha)$ play a crucial role in the analysis. 

\begin{theorem}\label{thm:phiZ_asymp_norm}
Suppose that $\E X(1)<0$ and $\psi_n-\psi(\xi) \approx \frac{1}{n}M_n+R_n$ such that:
\begin{enumerate}
\item[{\rm (i)}] $\sqrt{n}R_n{\parrow} 0$ and $M_n$ is a martingale with respect to $(V_0,\ldots,V_n)$ that satisfies the conditions of Lemma \ref{lemma:MCLT}a,
\item[{\rm (ii)}] $\lim_{n\to\infty}{\Var(M_n)}/{n}= \sigma_\xi^2<\infty$, 
\item[{\rm (iii)}] $\lim_{n\to\infty}\frac{1}{n}\sum_{i=1}^n\E\left[Z_i(\alpha) M_n\right]<\infty$,
\end{enumerate}
where $Z_i(\alpha)$ is given by \eqref{eq:phi_z_alpha}. Then $\sqrt{n}(\psi_n-\psi(\xi)){\darrow} \mathrm{N}(0,\sigma_\xi^2)$, and for every $\alpha>0$,
\begin{equation}\label{eq:phiZ_asymp_norm}
\sqrt{n}(\hat{\varphi}_n(\alpha;\psi_n)-\varphi(\alpha)) {\darrow}\mathrm{N}\left(0,\sigma_{\alpha,\xi}^2\right) \ ,
\end{equation}
where $0<\sigma_{\alpha,\xi}^2<\infty$.
\end{theorem}
\begin{remark}
The normality assumptions involving the asymptotic estimation error of $\psi(\xi)$ have to be verified for the specific $\psi_n$ used. Condition (i) entails that the estimation error satisfies a martingale CLT, condition (ii) specifies the convergence rate of the variance and can in principle be generalized to other rates, and condition (iii) further demands that the asymptotic covariance between estimation errors is bounded with the appropriate scaling. As it turns out, under some regularity conditions these assumptions are valid for various estimation procedures (see, e.g., \cite{book_H1997}) . In the next section we present an MLE for $\psi(\xi)$ for the subordinator case, and show that it satisfies the conditions of Theorem \ref{thm:phiZ_asymp_norm}. 
\end{remark}

\begin{proof}[Proof of Theorem \ref{thm:phiZ_asymp_norm}] In this proof a number of lemmas are needed. These are stated after the end of the proof, and are proven in Appendix \ref{sec:appA}.

{
We focus on the marginal asymptotic distribution of the estimation error for a fixed $\alpha>0$. We denote the estimator by $\varphi_n:=\hat{\varphi}_n(\alpha;\psi_n)$ and recall that the true value at $\alpha$ is $\varphi(\alpha)$.

The starting point is the estimation equation \eqref{eq:phi_z_equation}, which can be written as
\begin{equation}\label{eq:Jn}
J_n(\psi,\varphi):=\frac{1}{n}\sum_{i=1}^n \left[e^{-\alpha V_i}-\frac{\xi}{\xi-\varphi}\left(e^{-\alpha V_{i-1}}-\frac{\alpha}{\psi}e^{-\psi V_{i-1}}\right)\right]=0\ .
\end{equation}
The standard {\it delta method} can be applied for deriving the asymptotic distribution of $\varphi_n$; see, e.g., \cite[p.\ 51]{book_vdV1998}. Concretely, taking a Taylor expansion of $J_n(\psi_n,\varphi_n)$, where $\psi_n$ and $\varphi_n$ are the estimators, around the true parameters $\psi(\xi)$ and $\varphi(\alpha)$, yields
\begin{equation}\label{eq:Jn_Taylor}
J_n(\psi_n,\varphi_n)\approx J_n({\psi(\xi)},\varphi(\alpha))+(\psi_n-{\psi(\xi)})\frac{\partial}{\partial \psi}J_n({\psi(\xi)},\varphi(\alpha))+(\varphi_n-\varphi(\alpha))\frac{\partial}{\partial \varphi}J_n({\psi(\xi)},\varphi(\alpha))\ ,
\end{equation} 
where the remainder will be argued to vanish, in a convergence-in-probability sense, in Lemma \ref{lemma:Jn_remainder}. As $J_n(\psi_n,\varphi_n)=0$, we thus obtain
\begin{equation}\label{eq:Jn_approx}
\sqrt{n}(\varphi_n-\varphi(\alpha))\approx \frac{\sqrt{n}J_n({\psi(\xi)},\varphi(\alpha))+\sqrt{n}(\psi_n-{\psi(\xi)})\frac{\partial}{\partial \psi}J_n({\psi(\xi)},\varphi(\alpha))}{-\frac{\partial}{\partial \varphi}J_n({\psi(\xi)},\varphi(\alpha))}\ .
\end{equation}
Lemma \ref{lemma:Jn_limits} states that partial derivatives featuring in the 
right-hand side of \eqref{eq:Jn_approx}
converge almost surely to constants, which we call $\partial J_{\psi}$ and $\partial J_{\varphi}$. Thus, by applying Slutsky's lemma to \eqref{eq:Jn_approx}, $\sqrt{n}(\varphi_n-\varphi(\alpha))$ converges in distribution to the same limit as
\begin{equation}\label{eq:Jn_approx2}
-\frac{1}{\partial J_{\varphi}} \left(\sqrt{n}\,J_n({\psi(\xi)},\varphi(\alpha))+\sqrt{n}(\psi_n-{\psi(\xi)})\,\partial J_{\psi}\right)\ .
\end{equation}
By condition (i), \[\sqrt{n}(\psi_n-{\psi(\xi)})\approx \frac{1}{\sqrt{n}}M_n=\sqrt{\frac{\Var(M_n)}{n}}\sum_{i=1}^n Z_{ni},\] where $\sum_{i=1}^n Z_{ni}$ is a normalized sum of martingale differences that, due to Lemma \ref{lemma:MCLT}a, converges to a standard normal random variable. We conclude, by condition (ii), that the limiting distribution of $\sqrt{n}(\psi_n-{\psi(\xi)})$ is normal. In Lemma \ref{lemma:Jn_normal}a we further show that $\sqrt{n}J_n({\psi(\xi)},\varphi(\alpha)){\darrow}\mathrm{N}\left(0,\sigma_{\alpha}^2 \right)$, where $\sigma_{\alpha}^2$ is given by \eqref{eq:phi_sigma_alpha}.
In order to complete the proof of Theorem \ref{thm:phiZ_asymp_norm} it is now sufficient to prove that the limiting joint distribution of $\sqrt{n}J_n({\psi(\xi)},\varphi(\alpha))$ and $\sqrt{n}(\psi_n-{\psi(\xi)})$ is (bivariate) normal, or equivalently that any linear combination of them is asymptotically normal. This is done in Lemma \ref{lemma:Jn_normal}b by applying Lemma \ref{lemma:MCLT}b, which is facilitated by the assumption of condition (iii). Finally, we conclude that the asymptotic variance of $\sqrt{n}(\varphi_n-\varphi(\alpha))$ is
\begin{equation}\label{eq:phi_Z_asymp_var}
\sigma_{\alpha,\xi}^2:=\frac{\sigma_\alpha^2+2\partial J_{\psi}\sigma_{\alpha,\psi}^2+(\partial J_{\psi})^2\sigma_\xi^2}{\left(\partial J_{\varphi} \right)^2}\ ,
\end{equation} 
where
\begin{equation}\label{eq:cov_phi_psi}
\sigma_{\alpha,\psi}^2:=\lim_{n\to\infty}\Cov\left[\sqrt{n}J_n({\psi(\xi)},\varphi(\alpha)),\sqrt{n}(\psi_n-{\psi(\xi)})\right]\ .
\end{equation}This proves the claimed asymptotic normality.
}
\end{proof}

\begin{remark}
The explicit form of the term $\sigma_{\alpha,\psi}^2$ depends on the specific estimator $\psi_n$. In the sequel we derive an explicit expression for the asymptotic variance in the subordinator case when $\psi_n$ is an MLE.
\end{remark}

\begin{lemma}\label{lemma:Jn_limits}
As $n\to \infty$,
\begin{equation}\label{eq:Jn_dpsi}
\frac{\partial}{\partial \psi}J_n({\psi(\xi)},\varphi(\alpha)){\asarrow} -\frac{\alpha\varphi'(0)}{(\xi-\varphi(\alpha))\psi'(\xi)}=:\partial J_\psi \ ,
\end{equation}
and
\begin{equation}\label{eq:Jn_dphi}
\frac{\partial}{\partial \varphi}J_n({\psi(\xi)},\varphi(\alpha)){\asarrow} -\frac{\alpha\varphi\apost (0)}{\varphi(\alpha)(\xi-\varphi(\alpha))}=:
\partial J_\varphi\ .
\end{equation}
\end{lemma}

\begin{lemma}\label{lemma:Jn_normal}
\begin{itemize}
\item[{\rm (a)}] For any $\alpha>0$,
\begin{equation}\label{eq:Jn0_normal}
\sqrt{n}J_n({\psi(\xi)},\varphi(\alpha)){\darrow}\mathrm{N}\left(0,\sigma_{\alpha}^2 \right)\ ,
\end{equation}
where
\begin{equation}\label{eq:phi_sigma_alpha}
\sigma_{\alpha}^2: = \frac{2\xi^2\alpha\varphi\apost (0)}{(\xi-\varphi(\alpha))^2}\Bigg(\frac{\left(\frac{\varphi(\alpha)}{\xi}\right)^2-\frac{2\varphi(\alpha)}{\xi}}{\varphi(2\alpha)} -\frac{\alpha}{{\psi(\xi)}\varphi(2{\psi(\xi)})}+\frac{\alpha+{\psi(\xi)}}{{\psi(\xi)}\varphi(\alpha+{\psi(\xi)})} \Bigg) \ .
\end{equation}
\item[{\rm (b)}] Furthermore, under the conditions of Theorem \ref{thm:phiZ_asymp_norm},
\begin{equation}\label{eq:Jpsi_asymp_norm}
\sqrt{n}\begin{pmatrix}
J_n({\psi(\xi)},\varphi(\alpha)) \\
\psi_n-{\psi(\xi)}
\end{pmatrix}{\darrow}\mathrm{N}\left({\boldsymbol 0},\begin{pmatrix}
\sigma_\alpha^2 & \sigma_{\alpha,\psi}^2 \\
\sigma_{\alpha,\psi}^2 & \sigma_{\xi}^2
\end{pmatrix}\right) \ .
\end{equation}
Consequently,
\begin{equation}\label{eq:Jpsi_sum_norm}
\sqrt{n}\left(J_n({\psi(\xi)},\varphi(\alpha))+\partial J_{\psi}(\psi_n-{\psi(\xi)})\right){\darrow}\mathrm{N}\left(0,\sigma_\alpha^2+2\partial J_{\psi}\,\sigma_{\alpha,\psi}^2+\left(\partial J_{\psi}\right)^2\sigma_{\xi}^2\right) \ .
\end{equation}
\end{itemize}
\end{lemma}

\begin{lemma}\label{lemma:Jn_remainder}
The remainder term of the Taylor expansion in \eqref{eq:Jn_Taylor} converges in probability to zero.
\end{lemma}

\subsubsection{Multivariate asymptotic normality}\label{sec:phiZ_multi_normality}
We now consider the situation that $\varphi(\alpha)$ is estimated for multiple values of $\alpha$ simultaneously. The main result is that the estimation error of $\hat{\boldsymbol \Phi}_n({\boldsymbol\alpha})$ converges to a multivariate normal random variable at rate $\sqrt{n}$, under the conditions of Theorem \ref{thm:phiZ_asymp_norm}. The proof combines the lines of reasoning developed for the univariate asymptotic normality with the use of the Cram\'er-Wold device. 

We first introduce some notation. 
Denote for any $x>0$, 
\begin{equation}
\partial J_{\psi,x}:=\lim_{n\to\infty}\frac{\partial}{\partial \psi}J_n({\psi(\xi)},\varphi(x)),\:\:\:\:\partial J_{\varphi,x}:=\lim_{n\to\infty}\frac{\partial}{\partial \varphi}J_n({\psi(\xi)},\varphi(x))\ ,
\end{equation} 
as defined in \eqref{eq:Jn_dpsi} and \eqref{eq:Jn_dphi}. In addition, let
\begin{align}
\nonumber
\sigma_{J,\alpha,\beta}^2:=&\:\frac{(\alpha+\beta)\varphi\apost (0)}{\varphi(\alpha+\beta)}-\frac{\xi^2\varphi\apost (0)}{(\xi-\varphi(\alpha))(\xi-\varphi(\beta))}\times\\
&\:\:\:\:\left(\frac{\alpha+\beta}{\varphi(\alpha+\beta)}+\frac{2\alpha\beta}{{\psi(\xi)}\varphi(2{\psi(\xi)})}-\frac{\beta(\alpha+{\psi(\xi)})}{{\psi(\xi)}\varphi(\alpha+{\psi(\xi)})}-\frac{\alpha(\beta+{\psi(\xi)})}{{\psi(\xi)}\varphi(\beta+{\psi(\xi)})}\right),
\end{align}
and, for every $\alpha\neq\beta$,
\begin{equation}
\sigma_{\alpha,\beta}^2:=\frac{\sigma_{J,\alpha,\beta}^2+\partial J_{\psi,\beta}\sigma_{\alpha,\xi}^2+\partial J_{\psi,\alpha}\sigma_{\beta,\xi}^2+\partial J_{\psi,\alpha}\partial J_{\psi,\beta}\sigma_{\xi}^2}{\partial J_{\varphi,\alpha}\partial J_{\varphi,\beta}} \ .
\end{equation}

{
\begin{theorem}\label{thm:phiZ_asymp_norm_multi}
Suppose that $\E X(1)<0$. Let $\Phi({\boldsymbol\alpha}):=\left(\varphi(\alpha_1),\ldots,\varphi(\alpha_p)\right)$ and $\hat{\boldsymbol \Phi}_n({\boldsymbol\alpha}):=\left(\hat{\varphi}_n(\alpha_1;\psi_n),\ldots,\hat{\varphi}_n(\alpha_p;\psi_n)\right)$, where $(\alpha_1,\ldots,\alpha_p)\in(0,\infty)^p$. If $\psi_n$ satisfies the conditions of Theorem \ref{thm:phiZ_asymp_norm}, and in particular (iii) is satisfied for every $\alpha_i$, $i=1,\ldots,p$, then
\begin{equation}\label{eq:phiZ_asymp_norm_multi}
\sqrt{n}\left(\hat{\boldsymbol \Phi}_n({\boldsymbol\alpha})-\Phi({\boldsymbol\alpha})\right) {\darrow}\mathrm{N}\left(0,\Sigma\right) \ ,
\end{equation}
where $\Sigma_{ij}=\sigma_{\alpha_i,\xi}^2$ if $i=j$ and 
$
\sigma_{\alpha_i,\alpha_j}^2$ otherwise.
\end{theorem}
}

\section{Estimating $\psi(\xi)$ for subordinator input}\label{sec:MG1_MLE}

Theorems \ref{thm:phi_consistency} and \ref{thm:phiZ_asymp_norm} have established that, under specific conditions, the Z-estimator for $\varphi(\alpha)$ is consistent (asymptotically normal, respectively) if we have a consistent (asymptotically normal) estimator for $\psi(\xi)$. In this section we propose such an estimator for $\psi(\xi)$ using a likelihood approach for the subordinator input case, i.e., the input process has almost surely non-decreasing paths. An important special case is the CP model: nonnegative jumps (with LST $G^*(\alpha)$) arrive according to a Poisson process with rate $\lambda$. We thus have that for this example,
\begin{equation}\label{eq:MG1_exponent}
\varphi(\alpha)=\lambda\left(G^*(\alpha)-1\right)+\alpha\ .
\end{equation}
In the following an ML-type estimator is constructed for $\psi(\xi)$. Suppose we have a sample of workload observations ${\boldsymbol V}=(V_0,V_1,\ldots,V_n)$. Let $Y_i:=\mathbf{1}(V_i=0)$, then by \eqref{eq:P_idle_transient},
\begin{equation}
\P(Y_i=1\,|\,{\boldsymbol V})=\frac{\xi}{\psi(\xi)}e^{-\psi(\xi)V_{i-1}},\ i=1,\ldots,n\ .
\end{equation}
Conditional on $\mathbf{V}$, the sample ${\boldsymbol Y}=(Y_1,\ldots,Y_n)$ is then distributed as a sample of independent, but non-identically distributed, Bernoulli trials, with likelihood
\begin{equation}\label{eq:Y_likelihood}
\P_{\psi(\xi)}(Y_1,\ldots,Y_n|\mathbf{V})=\prod_{i=1}^n\left[\frac{\xi}{\psi(\xi)}e^{-\psi(\xi)V_{i-1}}\right]^{Y_i}\left[1-\frac{\xi}{\psi(\xi)}e^{-\psi(\xi)V_{i-1}}\right]^{1-Y_i}\ .
\end{equation}
We can thus define the MLE for $\psi(\xi)$: $\hat{\psi}_n=\argmax_{\psi\in\Psi}\{\P_{\psi}(Y_1,\ldots,Y_n|\mathbf{V})\}$.

We assume that the parameter space $\Psi$ is compact and that the true value lies within the interior. This assumption is reasonable and can be achieved by setting $\Psi=[\xi,C]$, where $C$ is an arbitrarily large constant. Note that $\psi(\xi)\geq \xi$ by \eqref{eq:P_idle_transient}. For special cases the parameter space can be specified in a more detailed manner. For example, the exponent function for the CP case is given in \eqref{eq:MG1_exponent}, and as $G^*(\alpha)\in[0,1]$, it follows that $\alpha-\lambda \leq \varphi(\alpha) \leq \alpha$.
This implies
\begin{equation}
\psi(\xi)-\lambda \leq \varphi(\psi(\xi)) \leq \psi(\xi)\ .
\end{equation}
As $\psi(\xi)=\varphi^{-1}(\xi)$, we conclude that the parameter space is $\Psi=[\xi,\xi+\lambda]$. In this definition there is an implicit assumption that $\lambda$ is known. If it is unknown, one could again follow the pragmatic approach of replacing $\xi+\lambda$ by an arbitrarily large constant.

For any $\psi\in\Psi$, the log-likelihood function, $L_n(\psi)=\log\P_\psi(Y_1,\ldots,Y_n\,|\,{\boldsymbol V})$, is 
\begin{equation}\label{eq:loglikelihood_psi}
L_n(\psi)=\sum_{i=1}^n\left[Y_i(\log\xi-\log\psi-\psi V_{i-1})+(1-Y_i)\log\left(1-\frac{\xi}{\psi}e^{-\psi V_{i-1}}\right)\right]\ ,
\end{equation}
and taking derivatives yields
\begin{equation}\label{eq:loglikelihood_psi_deriv}
L_n\apost (\psi) = \sum_{i=1}^n \frac{\left(\frac{1}{\psi}+V_{i-1}\right)\left(\frac{\xi}{\psi}e^{-\psi V_{i-1}}-Y_i\right)}{1-\frac{\xi}{\psi}e^{-\psi V_{i-1}}}\ .
\end{equation}

The MLE is the solution of the nonlinear optimization problem,
\begin{equation}
\hat{\psi}_n:=\argmax_{\psi\in\Psi}L_n (\psi)\ .
\end{equation}
The function $L_n(\psi)$ is not always concave with respect to $\psi$ but as the parameter space is compact a maximizer always exists. If $Y_i=1$ for all $i=1,\ldots,n$ (i.e., all observations are at idle periods), then by \eqref{eq:Y_likelihood} the likelihood $L_n(\psi)$ decreases with $\psi$, and therefore the MLE is the lower boundary of the parameter space. On the other hand if $Y_i=0$ for all $i=1,\ldots,n$ (i.e., no idle periods are observed), then by \eqref{eq:Y_likelihood} the likelihood $L_n(\psi)$ increases with $\psi$, and then the MLE is the upper boundary of the support. Otherwise, the solution may be in the interior or boundary of the parameter space. In the following subsection we will show that as $n$ increases the asymptotic log-likelihood has a unique maximizer, that solves the first order equation $L_n\apost (\hat{\psi}_n)=0$.

\subsection{Consistency}\label{sec:MLE_consistency}

In this subsection we establish strong consistency of the MLE $\hat{\psi}_n$. As $\hat{\psi}_n$ is consistent (actually even strongly consistent), Theorem \ref{thm:phi_consistency} thus yields the consistency of $\hat{\varphi}_n(\alpha;\hat{\psi}_n)$. The key for determining consistency and asymptotic normality is the asymptotic analysis of the log-likelihood function given in \eqref{eq:loglikelihood_psi} and its first two derivatives. We define $\ell_n(\psi) := {n}^{-1}\, L_n(\psi)$, and $\ell_n\apost(\psi)$ and $\ell_n''(\psi)$ as the respective first two derivatives. There are several issues that prevent us from directly applying standard MLE consistency results. Firstly, the observations are dependent. Additionally, the function $|\ell_n(\psi)|$ may be unbounded at the boundaries of the parameter space: e.g., if there is some $i$ such that $Y_i=1$ and $V_{i-1}=0$ then $\ell_n(\psi)\rightarrow-\infty$ as $\psi \downarrow \xi$. Therefore, we will need to verify that as $n$ grows the MLE moves away from the boundary almost surely.
 
In the next theorem, we assume that $\Psi$ is compact and that the true parameter belongs to the interior of the parameter space, $\Psi^{\mathrm{o}}$. In practice, setting $\Psi=[\xi,C]$ for a very large $C$ should ensure this condition.

\begin{theorem}\label{thm:psi_MLE_consistency}
If $\E X(1)<0$, and if the parameter space $\Psi$ is compact and $\psi(\xi)\in\Psi^{\mathrm{o}}$, then the MLE is strongly consistent: $\hat{\psi}_n{\asarrow}\psi(\xi)$.
\end{theorem}

Upon combining Theorems \ref{thm:phi_consistency} and \ref{thm:psi_MLE_consistency}, we thus have the following result.
\begin{corollary}\label{corr:phi_MLE_consistency}
If $\E X(1)<0$, and if the parameter space $\Psi$ is compact and $\psi(\xi)\in\Psi^{\mathrm{o}}$, then the Z-estimator for $\varphi(\alpha)$ defined in \eqref{eq:phi_z_estimator} using the MLE for $\psi(\xi)$ is pointwise consistent: $\hat{\varphi}_n(\alpha;\hat{\psi}_n){\parrow} \varphi(\alpha)$, for all $\alpha>0$.
\end{corollary}

The remainder of this subsection focuses on proving Theorem \ref{thm:psi_MLE_consistency}. We commence by providing two lemmas, which are proven in Appendix \ref{sec:appB}, featuring properties of the log-likelihood function that will be utilized in the asymptotic analysis. 
In particular, we assert smoothness and the construction of a strong law of large numbers for the first two derivatives of the log-likelihood.
\begin{lemma}\label{lemma:smooth_log_likelihood}
If the parameter space $\Psi\subset[\xi,\infty)$ is a compact interval, then the functions $\ell_n(\psi)$ and $\ell_n\apost (\psi)$ are bounded and smooth on $\Psi^{\mathrm{o}}$.
\end{lemma}

\begin{lemma}\label{lemma:dlog_likelihood_limits}
If $\E X(1)<0$, then for all $\psi\in\Psi^{\mathrm{o}}$,
\begin{equation}\label{eq:d_loglikelihood_limit}
\ell_n\apost (\psi) {\asarrow} \E(\E(\ell_1\apost (\psi)|V))=\E\left[\frac{\frac{1}{\psi}+V}{1-\frac{\xi}{\psi}e^{-\psi V}}\left(\frac{\xi}{\psi}e^{-\psi V}-\frac{\xi}{{\psi(\xi)}}e^{-{\psi(\xi)} V}\right)\right]\ ,
\end{equation} 
and
\begin{equation}\label{eq:d2_loglikelihood_limit}
\ell_n''(\psi) {\asarrow} \E(\E(\ell_1''(\psi)|V))\ .
\end{equation} 
\end{lemma}

\begin{proof}[Proof of Theorem \ref{thm:psi_MLE_consistency}]{
The following condition is sufficient for the consistency of an MLE with dependent observations \cite{HM1986}: for every $\psi\neq \psi(\xi)$ (in the interior of $\Psi$) there exists a $\delta>0$ such that 
\begin{equation}\label{eq:MLE_consistency_condition}
\lim_{n\to\infty}\sup_{\zeta\in \mathcal{B}_\delta(\psi)}\{\ell_n(\zeta)-\ell_n(\psi(\xi))\}<0\ ,
\end{equation}
almost surely, where $\mathcal{B}_\delta(\psi):=(\psi-\delta,\psi+\delta)$ is an open ball around $\psi$. 
Firstly, applying \eqref{eq:d_loglikelihood_limit} yields
\begin{equation}\label{eq:E_dll_limit}
\ell_n\apost (\psi) {\asarrow} \E(\E(\ell_1\apost (\psi)|V)) \left\lbrace\begin{array}{cc}
>0, & \psi<\psi(\xi), \\
=0, & \psi=\psi(\xi) ,\\
<0, & \psi>\psi(\xi) .
\end{array}\right. 
\end{equation} 
As $\psi(\xi)\in\Psi^{\mathrm{o}}$, for $\psi<\psi(\xi)$ we can consider a ball such that $\psi+\delta<\psi(\xi)$, and then for any $\zeta<\psi+\delta<\psi(\xi)$, by Lemma \ref{lemma:smooth_log_likelihood} the difference can be written as
\begin{equation}
\ell_n(\zeta)-\ell_n(\psi(\xi))=-\left(\ell_n(\psi(\xi))-\ell_n(\zeta)\right)=-\int_{\zeta}^{\psi(\xi)}\ell_n\apost (x) \ {\rm d}x\ .
\end{equation}
By \eqref{eq:E_dll_limit}, we have that as $n\to\infty$, $\ell_n\apost (\psi)>0$ almost surely for any $\psi\in[\psi+\delta,\psi(\xi)]$, which is a non-empty interval, and thus 
\begin{equation}
\sup_{\zeta\in(\psi-\delta,\psi+\delta)}\{\ell_n(\zeta)-\ell_n(\psi(\xi))\}\leq-\int_{\psi+\delta}^{\psi(\xi)}\ell_n\apost (x){\rm d}x<0
\end{equation}
almost surely. Similarly, for any $\psi>\psi(\xi)$ we can consider a ball such that $\psi-\delta>\psi(\xi)$ and conclude that
\begin{equation}
\sup_{\zeta\in(\psi-\delta,\psi+\delta)}\{\ell_n(\zeta)-\ell_n(\psi(\xi))\}\leq \int_{\psi(\xi)}^{\psi-\delta}\ell_n\apost (x){\rm d}x<0
\end{equation}
almost surely. Therefore, for any $\psi\neq\psi(\xi)$ there exists a $\delta<|\psi-\psi(\xi)|$ such that \eqref{eq:MLE_consistency_condition} holds, almost surely. As a consequence, \eqref{eq:MLE_consistency_condition} is satisfied.} 
\end{proof}

\subsection{Asymptotic normality}\label{sec:MLE_normal}

We now discuss the asymptotic distribution of the MLE $\hat{\psi}_n$. In this subsection we show that the estimation term is asymptotically normal with rate $\sqrt{n}$, and that it further satisfies the conditions of Theorem \ref{thm:phiZ_asymp_norm}. Useful conditions for asymptotic normality of the MLE when the observations are dependent are given in \cite{HM1986b}. Again the delta method is relied on; we also make use of arguments used in the proof of Cram\'{e}r's Theorem \cite[p. 121]{book_F1996}. Define
\begin{equation}\label{eq:MLE_I_theta_0}
I_\xi:=\E\left[\frac{\frac{\xi}{\psi(\xi)}e^{-\psi(\xi) V}\left(\frac{1}{\psi(\xi)}+V\right)^2}{1-\frac{\xi}{\psi(\xi)}e^{-\psi(\xi) V}}\right] \ .
\end{equation}

\begin{theorem}\label{thm:psi_MLE_asymp_norm}
If $\E X(1)<0$, and if the parameter space $\Psi$ is compact and $\psi(\xi)\in\Psi^{\mathrm{o}}$, then
\begin{equation}\label{eq:psi_MLE_asymp_norm}
\sqrt{n}(\hat{\psi}_n-\psi(\xi)) {\darrow}\mathrm{N}\left(0,\frac{1}{I_\xi}\right) \ .
\end{equation}
\end{theorem}

We first outline the two main steps required to establish the asymptotic distribution. This will be followed by a separate lemma verifying each of these steps. 

\begin{proof}
The smoothness of $\ell_n\apost (\cdot)$ established in Lemma \ref{lemma:smooth_log_likelihood} implies
\begin{equation}
\ell_n\apost (\hat{\psi}_n)=\ell_n\apost (\psi(\xi))+\int_{\psi(\xi)}^{\hat{\psi}_n}\ell_n''(x) \ {\rm d}x\ ,
\end{equation}
and by applying a change of variables
\begin{equation}
\begin{split}
\ell_n\apost (\hat{\psi}_n) &= \ell_n\apost (\psi(\xi))+(\hat{\psi}_n-\psi(\xi))\int_{0}^{1}\ell_n''(\psi(\xi)+y(\hat{\psi}_n-\psi(\xi)))\ {\rm d} y \\
&= \ell_n\apost (\psi(\xi))+(\hat{\psi}_n-\psi(\xi))\int_{0}^{1}\ell_n''(y\hat{\psi}_n+(1-y)\psi(\xi))\ {\rm d} y\ .
\end{split}
\end{equation}
If $\psi(\xi)$ is in the interior of $\Psi$, then by Lemma \ref{lemma:dlog_likelihood_limits}, as $n\to\infty$, the MLE is given by the root of the first order condition $L_n\apost (\hat{\psi}_n)=0$, almost surely. To verify this conclusion one only has to consider the limits of $\ell_n\apost (\psi)$ at the lower and upper boundary points of the compact interval $\Psi$, which are almost surely positive and negative, respectively. Therefore, $\sqrt{n}\ell_n\apost (\hat{\psi}_n)=\frac{1}{\sqrt{n}}L_n\apost (\hat{\psi}_n){\asarrow} 0$. Hence, denoting 
\begin{equation}
D_n:=\int_{0}^{1}\ell_n''(y\hat{\psi}_n+(1-y)\psi(\xi)))\ {\rm d} y,
\end{equation} 
we have
\begin{equation}
\sqrt{n}\ell_n\apost (\psi(\xi))+\sqrt{n}(\hat{\psi}_n-\psi(\xi))D_n \approx 0\ ,
\end{equation}
and therefore 
\begin{equation}\label{eq:score_approx}
\sqrt{n}(\hat{\psi}_n-\psi(\xi))\approx \frac{\sqrt{n}\ell_n\apost (\psi(\xi))}{-D_n}\ .
\end{equation}
We show that $\sqrt{n}\ell_n\apost (\psi(\xi)){\darrow}N\left(0,I_\xi\right)$ and $D_n{\asarrow}-I_\xi$ in Lemmas \ref{lemma:ell1_normal} and \ref{lemma:Dn_limit}, respectively. Then, by Slutsky's lemma, \eqref{eq:score_approx} implies \eqref{eq:psi_MLE_asymp_norm}. 
\end{proof}

\begin{lemma}\label{lemma:ell1_normal}
Under the conditions of Theorem \ref{thm:psi_MLE_consistency}, $\sqrt{n}\ell_n\apost (\psi(\xi)){\darrow}\mathrm{N}\left(0,I_\xi\right)$.
\end{lemma}

\begin{lemma}\label{lemma:Dn_limit}
Under the conditions of Theorem \ref{thm:psi_MLE_consistency}, $D_n{\asarrow} -I_\xi$.
\end{lemma}

For the proofs of these two lemmas we refer to Appendix \ref{sec:appB}.

\begin{remark} The asymptotic variance \eqref{eq:MLE_I_theta_0} in Theorem \ref {thm:psi_MLE_asymp_norm} is $\sigma_{\xi}^2:=1/{I_\xi}$. In the CP case $\varphi(\alpha)=\lambda\left(G^*(\alpha)-1\right)+\alpha$, and therefore $\psi(\xi)$ is an increasing function of $\xi$ such that $\psi(0)=0$ and
\begin{equation}
\frac{\xi}{\psi(\xi)}=\frac{\lambda\left(G^*(\psi(\xi))-1\right)+\psi(\xi)}{\psi(\xi)}=1+\lambda\frac{G^*(\psi(\xi))-1}{\psi(\xi)}\xrightarrow{\xi\downarrow 0} \varphi'(0).
\end{equation}
This implies that the asymptotic variance $\sigma_{\xi}^2$ increases with the sampling rate $\xi$, and that for a decreasing rate the variance disappears: $\sigma_{\xi}^2\to 0$ as $\xi\downarrow 0$.
\end{remark}

\subsection{Asymptotic error of $\hat{\varphi}_n(\alpha;\hat{\psi}_n)$}\label{sec:MG1_phi_normal}{
Theorem \ref{thm:psi_MLE_asymp_norm} has implications also for the asymptotic estimation error of $\varphi(\alpha)$. We verified the martingale CLT conditions of Lemma \ref{lemma:MCLT}a for $\sqrt{n}(\psi_n-{\psi(\xi)})$ in Theorem \ref{thm:psi_MLE_asymp_norm}, and thus if we further show that the covariance of the estimation errors has a finite limit we can apply Lemma \ref{lemma:MCLT}b and Theorem \ref{thm:phiZ_asymp_norm} to conclude that $\sqrt{n}(\varphi_n(\alpha;\hat{\psi}_n)-\varphi(\alpha))$ is asymptotically normal. This is done in the proof of Theorem \ref{thm:psi_MLE_Z_asymp_norm} in Appendix \ref{sec:appB}.}

\begin{theorem}\label{thm:psi_MLE_Z_asymp_norm}
If $\E X(1)<0$, $\hat{\psi}_n$ is the MLE for $\psi(\xi)$, then
\begin{equation}\label{eq:psi_MLE_Z_asymp_norm}
\sqrt{n}(\varphi_n(\alpha;\hat{\psi}_n)-\varphi(\alpha)) {\darrow}\mathrm{N}\left(0,\sigma_{\alpha,\xi}^2\right) \ ,
\end{equation}
where $\sigma_{\alpha,\xi}^2$ is given by \eqref{eq:phi_Z_asymp_var} with
\begin{equation}\sigma_{\alpha,\psi}^2:= \frac{\xi^2}{{\psi(\xi)}(\xi-\varphi(\alpha))I_\xi}\E\left[\frac{\left(\frac{1}{{\psi(\xi)}}+V\right)e^{-{\psi(\xi)} V}}{1-\frac{\xi}{{\psi(\xi)}}e^{-{\psi(\xi)} V}}\left(e^{-\alpha V}-\frac{\alpha}{{\psi(\xi)}}e^{-{\psi(\xi)} V}-\frac{\xi-\varphi(\alpha)}{\xi}\right) \right].
\end{equation}
\end{theorem}

\section{Estimating $\psi(\xi)$ for spectrally-positive L\'evy input}\label{sec:levy}

The Z-estimator $\hat{\varphi}_n(\alpha;\psi_n)$ requires an estimator for $\psi(\xi)$. In Section \ref{sec:MG1_MLE} we 
considered the subordinator case, utilizing that there is a non-zero probability that the queue is empty. For general spectrally-positive L\'evy input with a Brownian motion component this property does not hold. In this section we present an alternative procedure.

A moment based Z-estimator $\psi_n$ can be derived using \eqref{eq:EV_transient} by solving
\begin{equation}
\sum_{i=1}^n(V_i-V_{i-1})-\sum_{i=1}^n\left(\frac{e^{-\psi_n V_{i-1}}}{\psi_n}-\frac{\varphi\apost (0)}{\xi}\right)=0\ .
\end{equation}
After some rearranging of the terms, we define the estimator $\psi_n$ as the solution of
\begin{equation}\label{eq:levy_psi_root}
V_n-V_0+\frac{n\varphi\apost (0)}{\xi}=\sum_{i=1}^n\frac{e^{-\psi_n V_{i-1}}}{\psi_n}\ ,
\end{equation}
when it exists. There is at most one solution as the left-hand side is constant and the right-hand side is decreasing from infinity to zero on $\psi_n\in\mathbb{R}_+$. However, if the constant on the left is negative there is no solution and then we can set $\psi_n=\bar{\psi}$, where $\bar{\psi}>0$ is some arbitrary large constant. Observe that under the stability condition $\varphi\apost (0)=-\E X(1)>0$, the probability of having a negative term on the left vanishes as $n\to\infty$.

The problem, however, is that $\varphi\apost (0)$ in \eqref{eq:levy_psi_root} is also unknown and needs to be somehow estimated. Note that if we na\"{\i}vely plug $\hat{\varphi}\apost _n(0;\psi_n)$ into \eqref{eq:levy_psi_root} then all terms including $\psi_n$ cancel out because the estimator of the first derivative relies on the very same equation, i.e., \eqref{eq:EV_transient}.
In the next subsections we will present an estimation scheme for $\varphi\apost (0)$ that relies on considering a sub-sample with high initial workloads above some given threshold. We also show how the estimator's inherent bias can be made arbitrarily small. 

\subsection{Estimating $\varphi\apost (0)$}\label{sec:phi_prime}

We construct an estimator for $\varphi\apost (0)$ relying on the fact that $\xi\,\E X(T)=- {\varphi\apost (0)}$ where $T\sim\exp(\xi)$, and estimating the net input only using observations that are far from zero. If a high workload was observed then the probability of reflection (i.e., the queue hitting zero) before the next observation is low. In particular, we can quantify this probability so as to provide an upper bound for the bias of the estimator. 

For some (large) threshold $\tau>0$ the workload is sampled until there are $m\in{\mathbb N}$ observations $V_i$ such that $V_{i-1}\geq \tau$. Let $M(m,\tau)$ denote the total number of observations, which is now a random variable, and denote the $j$-th such observation by $i(j)\in\{1,\ldots, M(m,\tau)\}$, for $j=1,\ldots ,m$. The estimator for $\vartheta= -\varphi\apost (0)=\xi\E X(T)$ is 
\begin{equation}\label{eq:phi_prime_hat}
\hat{\vartheta}_m(\tau):=\frac{\xi}{m}\sum_{j=1}^m (V_{i(j)}-V_{i(j)-1})\ .
\end{equation}
Observe that the statistic $\hat{\vartheta}_m(\tau)$ only depends on the increments of the workload process which makes it convenient for analysis due to the independent increments property of L\'evy processes. We first derive an upper bound for the bias of $\hat{\vartheta}_m(\tau)$ and in the sequel provide an accurate asymptotic analysis of the bias. The proofs for this section are detailed in Appendix~\ref{sec:appC}.

\begin{proposition}\label{prop:tau_bias_bound}
Let $b_m(\hat{\vartheta};\tau):=\E[\hat{\vartheta}_m(\tau)-\vartheta]$. Then, for any $m\in{\mathbb N}$ and $\tau>0$,
\begin{equation}\label{eq:bias_varphi_prime}
0<b_m(\hat{\vartheta};\tau)\leq \xi\, \E[V(T)|V(0)=0]\,e^{-\psi(\xi)\tau} \ . 
\end{equation}
\end{proposition}

The bound \eqref{eq:bias_varphi_prime} for the bias decays to zero as $\tau$ grows, but at the same time the probability of observing workloads above $\tau$ also diminishes if $\E X(1)<0$, so that the required number of samples $M(m,\tau)$ also grows. This entails that there is a tradeoff between $\tau$ and $m$ in terms of accuracy, information and required sample size. The expectation in the bound, i.e., $\E[V(T)|V(0)=0]$, can be approximately estimated by using observations $V_{i}-V_{i-1}$ such that $V_{i-1}\le \epsilon$ for some low threshold $\epsilon>0$. The bound \eqref{eq:bias_varphi_prime} also depends on $\psi(\xi)$ which can be estimated using \eqref{eq:levy_psi_root} by plugging in the estimator $\hat{\vartheta}_m(\tau)$ instead of $-\varphi\apost (0)$. There is a (seemingly) circular argument here: the estimated bias bound is also biased because it also relies on the biased estimator $\hat{\vartheta}_m$. In the next proposition, however, we argue that this bias is bounded as well and diminishes with $\tau$.

\begin{proposition}\label{prop:tau_psi_bias_bound}
Let $\tilde{\psi}_m$ be the Z-estimator given by \eqref{eq:levy_psi_root} with $\varphi\apost (0)$ replaced with the estimator $\hat{\vartheta}_m(\tau)$, then the bias $\E[\tilde{\psi}_m-\psi(\xi)]$ diminishes to zero as $\tau\to\infty$.
\end{proposition}

\subsection{Asymptotic analysis}\label{sec:asymptotic_levy}

In this subsection we consider asymptotic properties of $\hat{\vartheta}_m(\tau)$ as $m\to\infty$. For any $\tau<\infty$ there is positive bias due to the possibility of reflection, even though it may be very small, and therefore the estimator cannot be consistent. A possible solution to this is defining an estimator with a dynamic threshold $\tau_m$ that grows with the sample size. In this section we analyze the asymptotic properties of the biased estimator based on the static threshold $\tau$. We first establish that the estimator $\hat{\vartheta}_m(\tau)$ converges in probability (as $m\to\infty$, that is) to a constant that depends on $\tau$. This further enables the computation of the exact asymptotic bias.

\begin{proposition}\label{prop:tau_asymp}
If $\E X(1)<0$ then for any $\tau>0$, as $m\to\infty$, $\hat{\vartheta}_m(\tau){\parrow} \vartheta(\tau)$, where
\begin{equation}\label{eq:bias_tau_asymp}
\vartheta(\tau) := \frac{\xi\E\left[e^{-\psi(\xi)V}\mathbf{1}(V\geq \tau)\right]}{\psi(\xi)\P(V
\ge\tau)}-\varphi\apost (0)>0\ ,
\end{equation}
and $V$ is the steady-state workload.\end{proposition}
An immediate result of Proposition \ref{prop:tau_asymp} is that from \eqref{eq:bias_tau_asymp} we can derive a bound on the asymptotic bias that does not depend on $\tau$.

\begin{corollary}
For any $\tau\geq 0$,
\begin{equation}
b\left(\hat{\vartheta};\tau\right):=\lim_{m\to\infty}b_m\left(\hat{\vartheta};\tau\right)=\frac{\xi\E\left[e^{-\psi(\xi)V}\mathbf{1}(V\geq \tau)\right]}{\psi(\xi)\P(V\ge\tau)} \leq \frac{\xi}{\psi(\xi)}\ .
\end{equation}
\end{corollary}
As $\varphi(\alpha)$ is a convex and increasing function, the term ${\xi}/{\psi(\xi)}$ is an increasing function of the sampling rate $\xi$. In line with earlier observations, it shows that sampling slowly is to be preferred when there is no, or little, cost associated with the sampling duration.

\begin{remark}
A rough approximation for the expected sample size $\E M(m,\tau)$ required to collect $m$ observations above a threshold $\tau$ can be obtained by considering the corresponding stationary tail probability. If the input process is light tailed, then (see \cite[Ch. IV]{book_DM2015}),
\begin{equation}
\P(V>\tau)\approx e^{-\omega \tau}\ ,
\end{equation}
where $\E e^{\omega X(1)}=1$. If $M$ stationary observations are made then the expected number above $\tau$ is $M\cdot \P(V>\tau)$, leading to the approximation $\E M(m,\tau)\cdot \P(V>\tau)\approx m$. We thus obtain
\begin{equation}
\E M(m,\tau)\approx {m}\,{e^{\omega \tau}} \ .
\end{equation} 
\end{remark}

\section{Simulation analysis}\label{sec:simulation}

In order to numerically test the performance of the estimation procedure we carried out simulations of a queue with L\'evy input that consists of three processes: $X(t)=C(t)+W(t)+U(t)$, where
\begin{itemize}
\item[$\circ$] $C(t)$ is Compound Poisson with rate $\lambda$ and Gamma jumps-size distribution, $\Gamma(\eta,\mu)$,
\item[$\circ$] $W(t)$ is Brownian motion with drift $d$ and variance $\sigma^2$,
\item[$\circ$] $U(t)$ is a Gamma process (subordinator) with parameters $(\beta,\gamma)$.
\end{itemize}
The exponent function is then
\begin{equation}
\varphi(\alpha)=\lambda\left[\left(\frac{\mu}{\mu+\alpha}\right)^\eta-1\right]-\alpha d+\frac{1}{2}\alpha^2\sigma^2+\beta\log\left(\frac{\gamma}{\gamma+\alpha}\right) \ ,
\end{equation}
and the expected input per unit of time is
$
\E X(1)=-\varphi{'}(0)={\lambda\eta}/{\mu}+d+ {\beta}/{\gamma}.$ We first examine the performance of the Z-estimator for $\varphi$ that uses the MLE estimator for $\psi(\xi)$ in the M/M/$1$ case $(\lambda,\eta,\mu,d,\sigma^2,\beta,\gamma)=(0.8,1,1,-1,0,0,0)$. In Figure \ref{fig:phi_alpha_rho08_xi1} the real and estimated exponent functions are compared along with the confidence interval based on the asymptotic normal approximation. Even for a moderate sample size of $n=30$ the estimated function is quite close to the real one. Observe, however, that the estimate is less accurate as $\alpha$ increases. In Figure \ref{fig:phi_alpha_rho08_xi1_n} we illustrate the increase in accuracy when the sample size $n$ increases. Theorem \ref{thm:phi_consistency} established that the estimator is consistent and here we can see that the bias diminishes very quickly. Increasing the sample size is more important when one wishes to estimate $\varphi(\alpha)$ for high values of $\alpha$. 

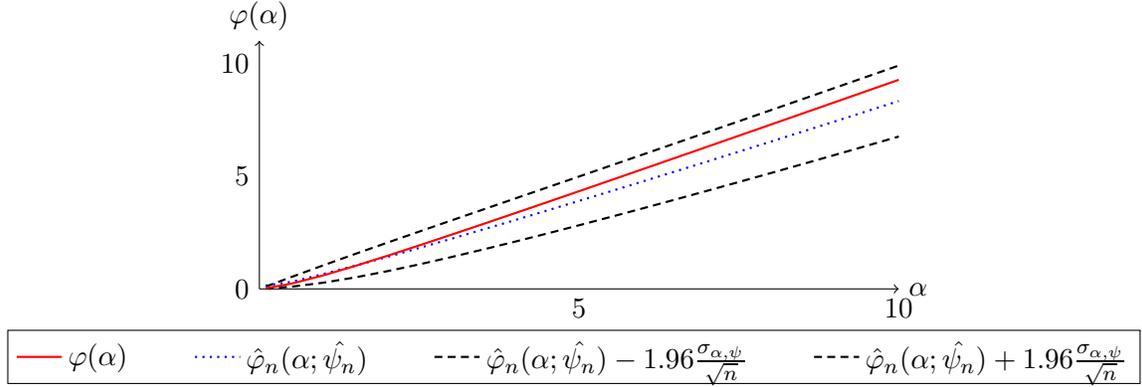
\begin{figure}[H]
\centering
\begin{tikzpicture}[xscale=0.85,yscale=0.3]
 \def\xmin{0}
 \def\xmax{10}
 \def\ymin{0}
 \def\ymax{11}
  \draw[->] (\xmin,\ymin) -- (\xmax,\ymin) node[right] {$\alpha$} ;
  \draw[->] (\xmin,\ymin) -- (\xmin,\ymax) node[above] {$\varphi(\alpha)$} ;
  \foreach \x in {5,10}
  \node at (\x,\ymin) [below] {\x};
  \foreach \y in {0,5,10}
  \node at (\xmin,\y) [left] {\y};

  \draw[densely dashed,thick, black] ( 0.1 , 0.01334 )-- ( 0.5 , 0.09848 )-- ( 1 , 0.27491 )-- ( 1.5 , 0.51023 )-- ( 2 , 0.78525 )-- ( 2.5 , 1.08781 )-- ( 3 , 1.41029 )-- ( 3.5 , 1.74781 )-- ( 4 , 2.09709 )-- ( 4.5 , 2.45589 )-- ( 5 , 2.8226 )-- ( 5.5 , 3.19603 )-- ( 6 , 3.5753 )-- ( 6.5 , 3.95969 )-- ( 7 , 4.34864 )-- ( 7.5 , 4.74169 )-- ( 8 , 5.13846 )-- ( 8.5 , 5.53861 )-- ( 9 , 5.94186 )-- ( 9.5 , 6.34797 )-- ( 10 , 6.75673 ) ;
  
   \draw[densely dashed,thick, black] ( 0.1 , 0.10398 )-- ( 0.5 , 0.52606 )-- ( 1 , 1.04564 )-- ( 1.5 , 1.55198 )-- ( 2 , 2.04985 )-- ( 2.5 , 2.54279 )-- ( 3 , 3.03293 )-- ( 3.5 , 3.52154 )-- ( 4 , 4.00945 )-- ( 4.5 , 4.49719 )-- ( 5 , 4.98513 )-- ( 5.5 , 5.47354 )-- ( 6 , 5.9626 )-- ( 6.5 , 6.45242 )-- ( 7 , 6.94308 )-- ( 7.5 , 7.43464 )-- ( 8 , 7.92711 )-- ( 8.5 , 8.42051 )-- ( 9 , 8.91482 )-- ( 9.5 , 9.41005 )-- ( 10 , 9.90617 );

  \draw[dotted,thick,blue] ( 0.1 , 0.15 )-- ( 0.5 , 0.31227 )-- ( 1 , 0.66028 )-- ( 1.5 , 1.0311 )-- ( 2 , 1.41755 )-- ( 2.5 , 1.8153 )-- ( 3 , 2.22161 )-- ( 3.5 , 2.63468 )-- ( 4 , 3.05327 )-- ( 4.5 , 3.47654 )-- ( 5 , 3.90387 )-- ( 5.5 , 4.33479 )-- ( 6 , 4.76895 )-- ( 6.5 , 5.20605 )-- ( 7 , 5.64586 )-- ( 7.5 , 6.08817 )-- ( 8 , 6.53278 )-- ( 8.5 , 6.97956 )-- ( 9 , 7.42834 )-- ( 9.5 , 7.87901 )-- ( 10 , 8.33145 );
  
   \draw[red,thick,smooth] ( 0.1 , 0.02727 )-- ( 0.5 , 0.23333 )-- ( 1 , 0.6 )-- ( 1.5 , 1.02 )-- ( 2 , 1.46667 )-- ( 2.5 , 1.92857 )-- ( 3 , 2.4 )-- ( 3.5 , 2.87778 )-- ( 4 , 3.36 )-- ( 4.5 , 3.84545 )-- ( 5 , 4.33333 )-- ( 5.5 , 4.82308 )-- ( 6 , 5.31429 )-- ( 6.5 , 5.80667 )-- ( 7 , 6.3 )-- ( 7.5 , 6.79412 )-- ( 8 , 7.28889 )-- ( 8.5 , 7.78421 )-- ( 9 , 8.28 )-- ( 9.5 , 8.77619 )-- ( 10 , 9.27273 );
   
\end{tikzpicture}

\begin{tikzpicture}
  \begin{customlegend}
  [legend entries={ $\varphi(\alpha)$,$\hat{\varphi}_n(\alpha;\hat{\psi_n})$,$\hat{\varphi}_n(\alpha;\hat{\psi_n})-1.96\frac{\sigma_{\alpha,\psi}}{\sqrt{n}}$,$\hat{\varphi}_n(\alpha;\hat{\psi_n})+1.96\frac{\sigma_{\alpha,\psi}}{\sqrt{n}}$},legend columns=-1,legend style={/tikz/every even column/.append style={column sep=0.8cm}}]  
  \addlegendimage{red,thick,smooth}
  \addlegendimage{blue,thick, dotted}  
  \addlegendimage{black,thick,densely dashed}   
  \addlegendimage{black,thick,densely dashed}  
  \end{customlegend}
\end{tikzpicture}
\caption{M/M/$1$ with $\lambda=0.8$ and $\mu=1$. The exponent function $\varphi(\alpha)$ (solid red line) and its Z-estimator $\hat{\varphi}_n(\alpha;\hat{\psi}_n)$ (dotted blue line) from a sample of $n=30$ observations with sampling rate $\xi=1$. The intermediate estimator $\hat{\psi}_n$ is the MLE of Section \ref{sec:MG1_MLE}. Also plotted are the upper and lower bounds of the confidence interval using the normal approximation and asymptotic variance.}\label{fig:phi_alpha_rho08_xi1}
\end{figure}

\begin{figure}[H]
\centering
\begin{subfigure}{0.9\linewidth}
\begin{tikzpicture}[xscale=0.5,yscale=0.7]
 \def\xmin{0}
 \def\xmax{20}
 \def\ymin{-2.5}
 \def\ymax{0.5}
  \draw[->] (\xmin,0) -- (\xmax,0) node[right] {$\alpha$} ;
  \draw[->] (\xmin,\ymin) -- (\xmin,\ymax) node[above] {$\hat{\varphi}_n(\alpha;\hat{\psi}_n)-\varphi(\alpha)$} ;
  \foreach \x in {10,20}
  \node at (\x,\ymin) [below] {\x};
  \foreach \y in {-2,-1,0}
  \node at (\xmin,\y) [left] {\y};
  
    \draw[red,thick,smooth] ( 0.1 , 0.01802 )-- ( 0.5 , 0.03888 )-- ( 1 , -0.00571 )-- ( 1.5 , -0.08084 )-- ( 2 , -0.16495 )-- ( 2.5 , -0.25051 )-- ( 3 , -0.33487 )-- ( 3.5 , -0.41714 )-- ( 4 , -0.49706 )-- ( 4.5 , -0.57465 )-- ( 5 , -0.65003 )-- ( 5.5 , -0.72334 )-- ( 6 , -0.79473 )-- ( 6.5 , -0.86435 )-- ( 7 , -0.93236 )-- ( 7.5 , -0.99889 )-- ( 8 , -1.06409 )-- ( 8.5 , -1.12807 )-- ( 9 , -1.19096 )-- ( 9.5 , -1.25287 )-- ( 10 , -1.31389 )-- ( 10.5 , -1.37412 )-- ( 11 , -1.43364 )-- ( 11.5 , -1.49252 )-- ( 12 , -1.55084 )-- ( 12.5 , -1.60864 )-- ( 13 , -1.66599 )-- ( 13.5 , -1.72293 )-- ( 14 , -1.77949 )-- ( 14.5 , -1.83572 )-- ( 15 , -1.89164 )-- ( 15.5 , -1.94727 )-- ( 16 , -2.00264 )-- ( 16.5 , -2.05777 )-- ( 17 , -2.11267 )-- ( 17.5 , -2.16735 )-- ( 18 , -2.22183 )-- ( 18.5 , -2.27611 )-- ( 19 , -2.33019 )-- ( 19.5 , -2.3841 )-- ( 20 , -2.43782 ) ;
    
    \draw[dotted,thick,blue] ( 0.1 , 0.01203 )-- ( 0.5 , 0.02606 )-- ( 1 , -0.01 )-- ( 1.5 , -0.07656 )-- ( 2 , -0.15432 )-- ( 2.5 , -0.23503 )-- ( 3 , -0.31539 )-- ( 3.5 , -0.39409 )-- ( 4 , -0.47065 )-- ( 4.5 , -0.54494 )-- ( 5 , -0.61694 )-- ( 5.5 , -0.68671 )-- ( 6 , -0.75431 )-- ( 6.5 , -0.81984 )-- ( 7 , -0.88337 )-- ( 7.5 , -0.945 )-- ( 8 , -1.00482 )-- ( 8.5 , -1.06291 )-- ( 9 , -1.11936 )-- ( 9.5 , -1.17427 )-- ( 10 , -1.22771 )-- ( 10.5 , -1.27978 )-- ( 11 , -1.33055 )-- ( 11.5 , -1.38009 )-- ( 12 , -1.42848 )-- ( 12.5 , -1.47579 )-- ( 13 , -1.52208 )-- ( 13.5 , -1.56741 )-- ( 14 , -1.61183 )-- ( 14.5 , -1.65541 )-- ( 15 , -1.69819 )-- ( 15.5 , -1.74022 )-- ( 16 , -1.78153 )-- ( 16.5 , -1.82218 )-- ( 17 , -1.8622 )-- ( 17.5 , -1.90163 )-- ( 18 , -1.94049 )-- ( 18.5 , -1.97883 )-- ( 19 , -2.01666 )-- ( 19.5 , -2.05402 )-- ( 20 , -2.09094 );
    
   \draw[purple,thick, dashed]  ( 0.1 , 0.01734 )-- ( 0.5 , 0.05521 )-- ( 1 , 0.0641 )-- ( 1.5 , 0.05018 )-- ( 2 , 0.02494 )-- ( 2.5 , -0.00558 )-- ( 3 , -0.03834 )-- ( 3.5 , -0.07182 )-- ( 4 , -0.10522 )-- ( 4.5 , -0.13814 )-- ( 5 , -0.17032 )-- ( 5.5 , -0.20165 )-- ( 6 , -0.23205 )-- ( 6.5 , -0.26147 )-- ( 7 , -0.28991 )-- ( 7.5 , -0.31739 )-- ( 8 , -0.34393 )-- ( 8.5 , -0.36958 )-- ( 9 , -0.39437 )-- ( 9.5 , -0.41836 )-- ( 10 , -0.4416 )-- ( 10.5 , -0.46414 )-- ( 11 , -0.48604 )-- ( 11.5 , -0.50734 )-- ( 12 , -0.52809 )-- ( 12.5 , -0.54833 )-- ( 13 , -0.56812 )-- ( 13.5 , -0.58748 )-- ( 14 , -0.60645 )-- ( 14.5 , -0.62508 )-- ( 15 , -0.64337 )-- ( 15.5 , -0.66138 )-- ( 16 , -0.67911 )-- ( 16.5 , -0.6966 )-- ( 17 , -0.71386 )-- ( 17.5 , -0.73092 )-- ( 18 , -0.7478 )-- ( 18.5 , -0.7645 )-- ( 19 , -0.78105 )-- ( 19.5 , -0.79747 )-- ( 20 , -0.81375 );
  
   \draw[green,thick,dashdotted] ( 0.1 , 0.00613 )-- ( 0.5 , 0.02719 )-- ( 1 , 0.03453 )-- ( 1.5 , 0.01886 )-- ( 2 , -0.00739 )-- ( 2.5 , -0.03695 )-- ( 3 , -0.06681 )-- ( 3.5 , -0.09588 )-- ( 4 , -0.12388 )-- ( 4.5 , -0.15078 )-- ( 5 , -0.17663 )-- ( 5.5 , -0.2015 )-- ( 6 , -0.22545 )-- ( 6.5 , -0.24853 )-- ( 7 , -0.2708 )-- ( 7.5 , -0.29228 )-- ( 8 , -0.31303 )-- ( 8.5 , -0.33309 )-- ( 9 , -0.35249 )-- ( 9.5 , -0.37128 )-- ( 10 , -0.38951 )-- ( 10.5 , -0.40721 )-- ( 11 , -0.42442 )-- ( 11.5 , -0.44119 )-- ( 12 , -0.45755 )-- ( 12.5 , -0.47355 )-- ( 13 , -0.4892 )-- ( 13.5 , -0.50456 )-- ( 14 , -0.51964 )-- ( 14.5 , -0.53449 )-- ( 15 , -0.54911 )-- ( 15.5 , -0.56355 )-- ( 16 , -0.57781 )-- ( 16.5 , -0.59194 )-- ( 17 , -0.60593 )-- ( 17.5 , -0.61982 )-- ( 18 , -0.63362 )-- ( 18.5 , -0.64734 )-- ( 19 , -0.66101 )-- ( 19.5 , -0.67462 )-- ( 20 , -0.6882 );
   
\end{tikzpicture}
\end{subfigure}

\begin{subfigure}{0.87\linewidth}
\begin{tikzpicture}[xscale=0.5,yscale=1.8]
 \def\xmin{0}
 \def\xmax{20}
 \def\ymin{-0.5}
 \def\ymax{1}
  \draw[->] (\xmin,0) -- (\xmax,0) node[right] {$\alpha$} ;
  \draw[->] (\xmin,\ymin) -- (\xmin,\ymax) node[above] {$\frac{\hat{\varphi}_n(\alpha;\hat{\psi}_n)-\varphi(\alpha)}{\varphi(\alpha)}$} ;
  \foreach \x in {10,20}
  \node at (\x,\ymin) [below] {\x};
  \foreach \y in {-0.5,0,0.5}
  \node at (\xmin,\y) [left] {\y};
  
    \draw[red,thick,smooth] ( 	0.1 , 	0.66080	)-- 	( 	0.5 , 	0.16663	)-- 	( 	1 , 	-0.00951666666666666	)-- 	( 	1.5 , 	-0.0792549019607843	)-- 	( 	2 , 	-0.11247	)-- 	( 	2.5 , 	-0.12989	)-- 	( 	3 , 	-0.139529166666667	)-- 	( 	3.5 , 	-0.14495	)-- 	( 	4 , 	-0.147934523809524	)-- 	( 	4.5 , 	-0.14944	)-- 	( 	5 , 	-0.15001	)-- 	( 	5.5 , 	-0.14997	)-- 	( 	6 , 	-0.14955	)-- 	( 	6.5 , 	-0.14885	)-- 	( 	7 , 	-0.147993650793651	)-- 	( 	7.5 , 	-0.14702	)-- 	( 	8 , 	-0.14599	)-- 	( 	8.5 , 	-0.14492	)-- 	( 	9 , 	-0.143835748792271	)-- 	( 	9.5 , 	-0.14276	)-- 	( 	10 , 	-0.14169	)-- 	( 	10.5 , 	-0.14065	)-- 	( 	11 , 	-0.13964	)-- 	( 	11.5 , 	-0.138658491267187	)-- 	( 	12 , 	-0.13771	)-- 	( 	12.5 , 	-0.13680	)-- 	( 	13 , 	-0.13592	)-- 	( 	13.5 , 	-0.13508	)-- 	( 	14 , 	-0.13427	)-- 	( 	14.5 , 	-0.13349	)-- 	( 	15 , 	-0.132746666666667	)-- 	( 	15.5 , 	-0.13203	)-- 	( 	16 , 	-0.13135	)-- 	( 	16.5 , 	-0.13069	)-- 	( 	17 , 	-0.13005	)-- 	( 	17.5 , 	-0.12945	)-- 	( 	18 , 	-0.12886	)-- 	( 	18.5 , 	-0.12830	)-- 	( 	19 , 	-0.127751644736842	)-- 	( 	19.5 , 	-0.127226503840649	)-- 	( 	20 , 	-0.126718334970709	);
    
    \draw[dotted,thick,blue] ( 	0.1 , 	0.44114	)-- 	( 	0.5 , 	0.11169	)-- 	( 	1 , 	-0.0166666666666667	)-- 	( 	1.5 , 	-0.0750588235294117	)-- 	( 	2 , 	-0.10522	)-- 	( 	2.5 , 	-0.12187	)-- 	( 	3 , 	-0.1314125	)-- 	( 	3.5 , 	-0.13694	)-- 	( 	4 , 	-0.140074404761905	)-- 	( 	4.5 , 	-0.14171	)-- 	( 	5 , 	-0.14237	)-- 	( 	5.5 , 	-0.14238	)-- 	( 	6 , 	-0.14194	)-- 	( 	6.5 , 	-0.14119	)-- 	( 	7 , 	-0.14021746031746	)-- 	( 	7.5 , 	-0.139090860920914	)-- 	( 	8 , 	-0.13786	)-- 	( 	8.5 , 	-0.13655	)-- 	( 	9 , 	-0.135188405797102	)-- 	( 	9.5 , 	-0.13380	)-- 	( 	10 , 	-0.13240	)-- 	( 	10.5 , 	-0.13100	)-- 	( 	11 , 	-0.12960	)-- 	( 	11.5 , 	-0.128213489409142	)-- 	( 	12 , 	-0.12685	)-- 	( 	12.5 , 	-0.12550	)-- 	( 	13 , 	-0.12418	)-- 	( 	13.5 , 	-0.12288	)-- 	( 	14 , 	-0.12162	)-- 	( 	14.5 , 	-0.12038	)-- 	( 	15 , 	-0.119171228070175	)-- 	( 	15.5 , 	-0.11799	)-- 	( 	16 , 	-0.11684	)-- 	( 	16.5 , 	-0.11573	)-- 	( 	17 , 	-0.114636146275279	)-- 	( 	17.5 , 	-0.11358	)-- 	( 	18 , 	-0.11254	)-- 	( 	18.5 , 	-0.11154	)-- 	( 	19 , 	-0.1105625	)-- 	( 	19.5 , 	-0.10961	)-- 	( 	20 , 	-0.108687448344691	);
    
   \draw[purple,thick, dashed] ( 	0.1 , 	0.22479	)-- 	( 	0.5 , 	0.11653	)-- 	( 	1 , 	0.0575499999999999	)-- 	( 	1.5 , 	0.0184901960784315	)-- 	( 	2 , 	-0.00504	)-- 	( 	2.5 , 	-0.01916	)-- 	( 	3 , 	-0.0278374999999999	)-- 	( 	3.5 , 	-0.03332	)-- 	( 	4 , 	-0.0368690476190477	)-- 	( 	4.5 , 	-0.03921	)-- 	( 	5 , 	-0.04076	)-- 	( 	5.5 , 	-0.0417782827570764	)-- 	( 	6 , 	-0.04242	)-- 	( 	6.5 , 	-0.04280	)-- 	( 	7 , 	-0.042984126984127	)-- 	( 	7.5 , 	-0.04302	)-- 	( 	8 , 	-0.04295	)-- 	( 	8.5 , 	-0.04279	)-- 	( 	9 , 	-0.0425712560386474	)-- 	( 	9.5 , 	-0.04231	)-- 	( 	10 , 	-0.04201	)-- 	( 	10.5 , 	-0.04168	)-- 	( 	11 , 	-0.04134	)-- 	( 	11.5 , 	-0.0409875510962468	)-- 	( 	12 , 	-0.04063	)-- 	( 	12.5 , 	-0.04027	)-- 	( 	13 , 	-0.0399114312147859	)-- 	( 	13.5 , 	-0.03956	)-- 	( 	14 , 	-0.03921	)-- 	( 	14.5 , 	-0.03887	)-- 	( 	15 , 	-0.0385340350877193	)-- 	( 	15.5 , 	-0.03821	)-- 	( 	16 , 	-0.03790	)-- 	( 	16.5 , 	-0.03759	)-- 	( 	17 , 	-0.03730	)-- 	( 	17.5 , 	-0.03702	)-- 	( 	18 , 	-0.03675	)-- 	( 	18.5 , 	-0.03649	)-- 	( 	19 , 	-0.0362395833333334	)-- 	( 	19.5 , 	-0.03600	)-- 	( 	20 , 	-0.0358	);

   \draw[green,thick,dashdotted] ( 	0.1 , 	0.22479	)-- 	( 	0.5 , 	0.11653	)-- 	( 	1 , 	0.0575499999999999	)-- 	( 	1.5 , 	0.0184901960784315	)-- 	( 	2 , 	-0.00504	)-- 	( 	2.5 , 	-0.01916	)-- 	( 	3 , 	-0.0278374999999999	)-- 	( 	3.5 , 	-0.03332	)-- 	( 	4 , 	-0.0368690476190477	)-- 	( 	4.5 , 	-0.03921	)-- 	( 	5 , 	-0.04076	)-- 	( 	5.5 , 	-0.0417782827570764	)-- 	( 	6 , 	-0.04242	)-- 	( 	6.5 , 	-0.04280	)-- 	( 	7 , 	-0.042984126984127	)-- 	( 	7.5 , 	-0.04302	)-- 	( 	8 , 	-0.04295	)-- 	( 	8.5 , 	-0.04279	)-- 	( 	9 , 	-0.0425712560386474	)-- 	( 	9.5 , 	-0.04231	)-- 	( 	10 , 	-0.04201	)-- 	( 	10.5 , 	-0.04168	)-- 	( 	11 , 	-0.04134	)-- 	( 	11.5 , 	-0.0409875510962468	)-- 	( 	12 , 	-0.04063	)-- 	( 	12.5 , 	-0.04027	)-- 	( 	13 , 	-0.0399114312147859	)-- 	( 	13.5 , 	-0.03956	)-- 	( 	14 , 	-0.03921	)-- 	( 	14.5 , 	-0.03887	)-- 	( 	15 , 	-0.0385340350877193	)-- 	( 	15.5 , 	-0.03821	)-- 	( 	16 , 	-0.03790	)-- 	( 	16.5 , 	-0.03759	)-- 	( 	17 , 	-0.03730	)-- 	( 	17.5 , 	-0.03702	)-- 	( 	18 , 	-0.03675	)-- 	( 	18.5 , 	-0.03649	)-- 	( 	19 , 	-0.0362395833333334	)-- 	( 	19.5 , 	-0.03600	)-- 	( 	20 , 	-0.0358	) ; 
   
\end{tikzpicture}
\end{subfigure}

\begin{tikzpicture}
  \begin{customlegend}
  [legend entries={$n=30$,$n=50$,$n=100$,$n=200$},legend columns=-1,legend style={/tikz/every even column/.append style={column sep=0.8cm}}]  
  \addlegendimage{red,thick,smooth}
  \addlegendimage{blue,thick, dotted}  
  \addlegendimage{purple,thick, dashed} 
  \addlegendimage{green,thick,dashdotted}  
  \end{customlegend}
\end{tikzpicture}
\caption{M/M/$1$ with $\lambda=0.8$ and $\mu=1$. The relative estimation error of the exponent function, $(\hat{\varphi}_n(\alpha;\hat{\psi}_n)-\varphi(\alpha))/\varphi(\alpha)$, based on sample sizes $n\in\{30,50,100,200\}$ with sampling rate $\xi=1$. The intermediate estimator $\hat{\psi}_n$ is the MLE of Section \ref{sec:MG1_MLE}. Common random numbers were used to simulate the full sample of $n=200$ and the other estimators use the first $n$ observations.}\label{fig:phi_alpha_rho08_xi1_n}
\end{figure}
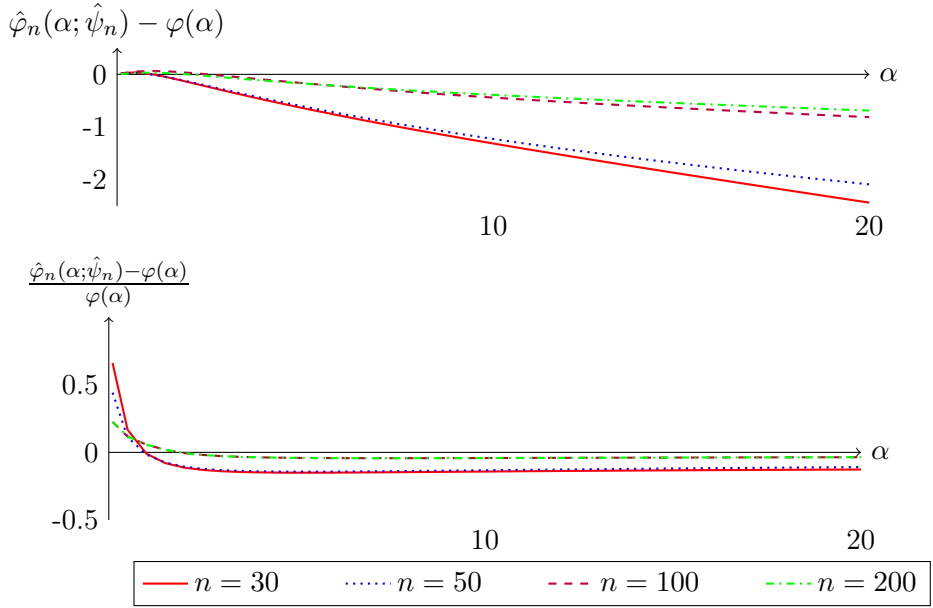

In Figure \ref{fig:sigma_alpha_xi} the asymptotic variance $\sigma_{\alpha,\xi}^2$ is plotted as a function of $\alpha$ for different sampling rates. These functions are all increasing and appear to be convex and unbounded. The asymptotic correlation, \begin{equation}r(\alpha,1):=\frac{\sigma_{\alpha,1}^2}{\sqrt{\sigma_{\alpha,\xi}^2 \sigma_{1,\xi}^2}}\end{equation} (see Theorem \ref{thm:phiZ_asymp_norm_multi}), between the estimated function $\hat{\varphi}_n(\alpha;\hat{\psi}_n)$ and $\hat{\varphi}_n(1;\hat{\psi}_n)$ is illustrated in Figure \ref{fig:corr_alpha_xi}. The correlation is a decreasing function of the distance from one, but interestingly does not disappear to zero but to a lower bound (about $0.2$ in this example). This should come as no surprise after some thought, as the function is estimated using the same sample of observations for all $\alpha>0$.

\begin{figure}[H]
\centering
\begin{tikzpicture}[xscale=0.2,yscale=0.015]
 \def\xmin{0}
 \def\xmax{50}
 \def\ymin{0}
 \def\ymax{200}
  \draw[->] (\xmin,\ymin) -- (\xmax,\ymin) node[right] {$\alpha$} ;
  \draw[->] (\xmin,\ymin) -- (\xmin,\ymax) node[above] {$\sigma_{\alpha,\xi}^2$} ;
  \foreach \x in {10,20,30,40,50}
  \node at (\x,\ymin) [below] {\x};
  \foreach \y in {0,50,100,150,200}
  \node at (\xmin,\y) [left] {\y};
  
   \draw[densely dashed,thick, black] ( 0.1 , 0.01604 )-- ( 0.5 , 0.35706 )-- ( 1 , 1.16124 )-- ( 1.5 , 2.12448 )-- ( 2 , 3.13597 )-- ( 2.5 , 4.15946 )-- ( 3 , 5.18443 )-- ( 3.5 , 6.20924 )-- ( 4 , 7.23522 )-- ( 4.5 , 8.2645 )-- ( 5 , 9.29926 )-- ( 5.5 , 10.34145 )-- ( 6 , 11.39273 )-- ( 6.5 , 12.45452 )-- ( 7 , 13.52797 )-- ( 7.5 , 14.61406 )-- ( 8 , 15.71358 )-- ( 8.5 , 16.82721 )-- ( 9 , 17.95549 )-- ( 9.5 , 19.0989 )-- ( 10 , 20.25784 )-- ( 10.5 , 21.43264 )-- ( 11 , 22.62358 )-- ( 11.5 , 23.83091 )-- ( 12 , 25.05484 )-- ( 12.5 , 26.29554 )-- ( 13 , 27.55319 )-- ( 13.5 , 28.8279 )-- ( 14 , 30.11981 )-- ( 14.5 , 31.42901 )-- ( 15 , 32.75559 )-- ( 15.5 , 34.09965 )-- ( 16 , 35.46124 )-- ( 16.5 , 36.84044 )-- ( 17 , 38.23729 )-- ( 17.5 , 39.65185 )-- ( 18 , 41.08417 )-- ( 18.5 , 42.53428 )-- ( 19 , 44.00221 )-- ( 19.5 , 45.48801 )-- ( 20 , 46.99171 )-- ( 20.5 , 48.51332 )-- ( 21 , 50.05287 )-- ( 21.5 , 51.61038 )-- ( 22 , 53.18588 )-- ( 22.5 , 54.77938 )-- ( 23 , 56.3909 )-- ( 23.5 , 58.02045 )-- ( 24 , 59.66805 )-- ( 24.5 , 61.33371 )-- ( 25 , 63.01744 )-- ( 25.5 , 64.71925 )-- ( 26 , 66.43915 )-- ( 26.5 , 68.17715 )-- ( 27 , 69.93326 )-- ( 27.5 , 71.70748 )-- ( 28 , 73.49983 )-- ( 28.5 , 75.31031 )-- ( 29 , 77.13892 )-- ( 29.5 , 78.98568 )-- ( 30 , 80.85058 )-- ( 30.5 , 82.73363 )-- ( 31 , 84.63484 )-- ( 31.5 , 86.55421 )-- ( 32 , 88.49174 )-- ( 32.5 , 90.44744 )-- ( 33 , 92.42131 )-- ( 33.5 , 94.41336 )-- ( 34 , 96.42358 )-- ( 34.5 , 98.45198 )-- ( 35 , 100.49857 )-- ( 35.5 , 102.56334 )-- ( 36 , 104.6463 )-- ( 36.5 , 106.74744 )-- ( 37 , 108.86678 )-- ( 37.5 , 111.00431 )-- ( 38 , 113.16004 )-- ( 38.5 , 115.33397 )-- ( 39 , 117.52609 )-- ( 39.5 , 119.73641 )-- ( 40 , 121.96494 )-- ( 40.5 , 124.21166 )-- ( 41 , 126.4766 )-- ( 41.5 , 128.75973 )-- ( 42 , 131.06107 )-- ( 42.5 , 133.38063 )-- ( 43 , 135.71838 )-- ( 43.5 , 138.07435 )-- ( 44 , 140.44853 )-- ( 44.5 , 142.84092 )-- ( 45 , 145.25152 )-- ( 45.5 , 147.68034 )-- ( 46 , 150.12737 )-- ( 46.5 , 152.59261 )-- ( 47 , 155.07607 )-- ( 47.5 , 157.57774 )-- ( 48 , 160.09763 )-- ( 48.5 , 162.63573 )-- ( 49 , 165.19206 )-- ( 49.5 , 167.7666 )-- ( 50 , 170.35936 );

  \draw[dotted,thick,blue] ( 0.1 , 0.01009 )-- ( 0.5 , 0.15326 )-- ( 1 , 0.40857 )-- ( 1.5 , 0.67992 )-- ( 2 , 0.95228 )-- ( 2.5 , 1.22438 )-- ( 3 , 1.49783 )-- ( 3.5 , 1.77455 )-- ( 4 , 2.05618 )-- ( 4.5 , 2.34404 )-- ( 5 , 2.63913 )-- ( 5.5 , 2.94223 )-- ( 6 , 3.25393 )-- ( 6.5 , 3.57469 )-- ( 7 , 3.90486 )-- ( 7.5 , 4.24474 )-- ( 8 , 4.59455 )-- ( 8.5 , 4.95446 )-- ( 9 , 5.32462 )-- ( 9.5 , 5.70514 )-- ( 10 , 6.09613 )-- ( 10.5 , 6.49765 )-- ( 11 , 6.90978 )-- ( 11.5 , 7.33257 )-- ( 12 , 7.76606 )-- ( 12.5 , 8.21029 )-- ( 13 , 8.66529 )-- ( 13.5 , 9.13109 )-- ( 14 , 9.60772 )-- ( 14.5 , 10.09518 )-- ( 15 , 10.5935 )-- ( 15.5 , 11.1027 )-- ( 16 , 11.62277 )-- ( 16.5 , 12.15374 )-- ( 17 , 12.69562 )-- ( 17.5 , 13.2484 )-- ( 18 , 13.81211 )-- ( 18.5 , 14.38673 )-- ( 19 , 14.97229 )-- ( 19.5 , 15.56878 )-- ( 20 , 16.1762 )-- ( 20.5 , 16.79457 )-- ( 21 , 17.42388 )-- ( 21.5 , 18.06414 )-- ( 22 , 18.71534 )-- ( 22.5 , 19.3775 )-- ( 23 , 20.0506 )-- ( 23.5 , 20.73466 )-- ( 24 , 21.42968 )-- ( 24.5 , 22.13565 )-- ( 25 , 22.85257 )-- ( 25.5 , 23.58046 )-- ( 26 , 24.3193 )-- ( 26.5 , 25.06911 )-- ( 27 , 25.82987 )-- ( 27.5 , 26.60159 )-- ( 28 , 27.38427 )-- ( 28.5 , 28.17792 )-- ( 29 , 28.98252 )-- ( 29.5 , 29.79809 )-- ( 30 , 30.62462 )-- ( 30.5 , 31.46211 )-- ( 31 , 32.31056 )-- ( 31.5 , 33.16998 )-- ( 32 , 34.04035 )-- ( 32.5 , 34.92169 )-- ( 33 , 35.81399 )-- ( 33.5 , 36.71726 )-- ( 34 , 37.63148 )-- ( 34.5 , 38.55667 )-- ( 35 , 39.49282 )-- ( 35.5 , 40.43993 )-- ( 36 , 41.39801 )-- ( 36.5 , 42.36705 )-- ( 37 , 43.34705 )-- ( 37.5 , 44.33801 )-- ( 38 , 45.33993 )-- ( 38.5 , 46.35281 )-- ( 39 , 47.37666 )-- ( 39.5 , 48.41147 )-- ( 40 , 49.45724 )-- ( 40.5 , 50.51397 )-- ( 41 , 51.58166 )-- ( 41.5 , 52.66032 )-- ( 42 , 53.74993 )-- ( 42.5 , 54.85051 )-- ( 43 , 55.96205 )-- ( 43.5 , 57.08455 )-- ( 44 , 58.21801 )-- ( 44.5 , 59.36243 )-- ( 45 , 60.51781 )-- ( 45.5 , 61.68415 )-- ( 46 , 62.86145 )-- ( 46.5 , 64.04972 )-- ( 47 , 65.24894 )-- ( 47.5 , 66.45912 )-- ( 48 , 67.68027 )-- ( 48.5 , 68.91237 )-- ( 49 , 70.15543 )-- ( 49.5 , 71.40946 )-- ( 50 , 72.67444 );
  
   \draw[red,thick,smooth] ( 0.1 , 0.00425 )-- ( 0.5 , 0.05309 )-- ( 1 , 0.1233 )-- ( 1.5 , 0.1908 )-- ( 2 , 0.25798 )-- ( 2.5 , 0.32784 )-- ( 3 , 0.40259 )-- ( 3.5 , 0.48372 )-- ( 4 , 0.57224 )-- ( 4.5 , 0.66888 )-- ( 5 , 0.77414 )-- ( 5.5 , 0.88839 )-- ( 6 , 1.01191 )-- ( 6.5 , 1.14488 )-- ( 7 , 1.28748 )-- ( 7.5 , 1.43981 )-- ( 8 , 1.60197 )-- ( 8.5 , 1.77403 )-- ( 9 , 1.95605 )-- ( 9.5 , 2.14808 )-- ( 10 , 2.35016 )-- ( 10.5 , 2.5623 )-- ( 11 , 2.78455 )-- ( 11.5 , 3.01693 )-- ( 12 , 3.25943 )-- ( 12.5 , 3.51209 )-- ( 13 , 3.77492 )-- ( 13.5 , 4.04791 )-- ( 14 , 4.33109 )-- ( 14.5 , 4.62446 )-- ( 15 , 4.92802 )-- ( 15.5 , 5.24178 )-- ( 16 , 5.56574 )-- ( 16.5 , 5.89991 )-- ( 17 , 6.24429 )-- ( 17.5 , 6.59889 )-- ( 18 , 6.9637 )-- ( 18.5 , 7.33872 )-- ( 19 , 7.72396 )-- ( 19.5 , 8.11943 )-- ( 20 , 8.52511 )-- ( 20.5 , 8.94101 )-- ( 21 , 9.36714 )-- ( 21.5 , 9.80349 )-- ( 22 , 10.25006 )-- ( 22.5 , 10.70686 )-- ( 23 , 11.17388 )-- ( 23.5 , 11.65113 )-- ( 24 , 12.1386 )-- ( 24.5 , 12.6363 )-- ( 25 , 13.14422 )-- ( 25.5 , 13.66237 )-- ( 26 , 14.19074 )-- ( 26.5 , 14.72933 )-- ( 27 , 15.27816 )-- ( 27.5 , 15.83721 )-- ( 28 , 16.40648 )-- ( 28.5 , 16.98598 )-- ( 29 , 17.5757 )-- ( 29.5 , 18.17565 )-- ( 30 , 18.78582 )-- ( 30.5 , 19.40622 )-- ( 31 , 20.03685 )-- ( 31.5 , 20.6777 )-- ( 32 , 21.32877 )-- ( 32.5 , 21.99007 )-- ( 33 , 22.66159 )-- ( 33.5 , 23.34333 )-- ( 34 , 24.0353 )-- ( 34.5 , 24.7375 )-- ( 35 , 25.44991 )-- ( 35.5 , 26.17255 )-- ( 36 , 26.90542 )-- ( 36.5 , 27.64851 )-- ( 37 , 28.40182 )-- ( 37.5 , 29.16535 )-- ( 38 , 29.93911 )-- ( 38.5 , 30.72309 )-- ( 39 , 31.51729 )-- ( 39.5 , 32.32172 )-- ( 40 , 33.13637 )-- ( 40.5 , 33.96124 )-- ( 41 , 34.79633 )-- ( 41.5 , 35.64164 )-- ( 42 , 36.49718 )-- ( 42.5 , 37.36294 )-- ( 43 , 38.23892 )-- ( 43.5 , 39.12512 )-- ( 44 , 40.02154 )-- ( 44.5 , 40.92819 )-- ( 45 , 41.84506 )-- ( 45.5 , 42.77214 )-- ( 46 , 43.70945 )-- ( 46.5 , 44.65698 )-- ( 47 , 45.61473 )-- ( 47.5 , 46.58271 )-- ( 48 , 47.5609 )-- ( 48.5 , 48.54931 )-- ( 49 , 49.54795 )-- ( 49.5 , 50.5568 )-- ( 50 , 51.57587 ) ;
   
\end{tikzpicture}

\begin{tikzpicture}
  \begin{customlegend}
  [legend entries={ $\xi=0.1$,$\xi=1$,$\xi=5$},legend columns=-1,legend style={/tikz/every even column/.append style={column sep=0.8cm}}]  
  \addlegendimage{red,thick,smooth}
  \addlegendimage{blue,thick, dotted}  
  \addlegendimage{black,thick,densely dashed}   
  \end{customlegend}
\end{tikzpicture}
\caption{M/M/$1$ with $\lambda=0.8$ and $\mu=1$. The asymptotic variance $\sigma_{\alpha,\xi}^2$ as a function of $\alpha$ for different sampling rates.}\label{fig:sigma_alpha_xi}
\end{figure}
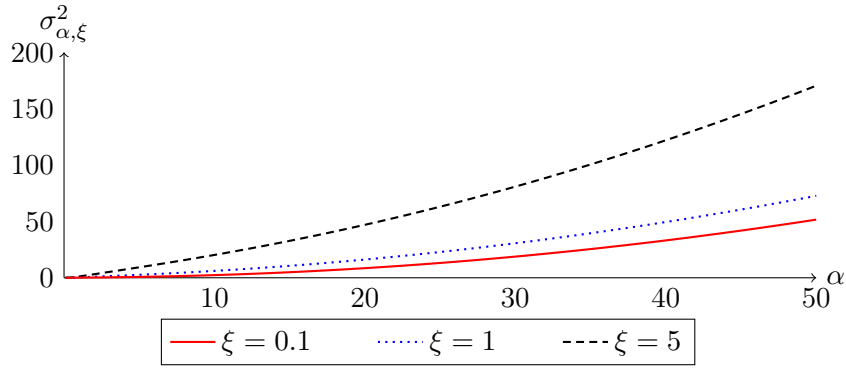

\begin{figure}[H]
\centering
\begin{tikzpicture}[xscale=0.2,yscale=4]
 \def\xmin{0}
 \def\xmax{50}
 \def\ymin{0}
 \def\ymax{1.1}
  \draw[->] (\xmin,\ymin) -- (\xmax,\ymin) node[right] {$\alpha$} ;
  \draw[->] (\xmin,\ymin) -- (\xmin,\ymax) node[above] {$r(\alpha,1)$} ;
  \foreach \x in {10,20,30,40,50}
  \node at (\x,\ymin) [below] {\x};
  \foreach \y in {0,0.2,0.4,0.6,0.8,1}
  \node at (\xmin,\y) [left] {\y};
  
   \draw[red,thick,smooth] ( 0.1 , 0.84704 )-- ( 0.5 , 0.97323 )-- ( 1 , 1 )-- ( 1.5 , 0.9879 )-- ( 2 , 0.96306 )-- ( 2.5 , 0.93372 )-- ( 3 , 0.90317 )-- ( 3.5 , 0.87286 )-- ( 4 , 0.84344 )-- ( 4.5 , 0.81524 )-- ( 5 , 0.78839 )-- ( 5.5 , 0.7629 )-- ( 6 , 0.73876 )-- ( 6.5 , 0.7159 )-- ( 7 , 0.69428 )-- ( 7.5 , 0.67381 )-- ( 8 , 0.65442 )-- ( 8.5 , 0.63605 )-- ( 9 , 0.61863 )-- ( 9.5 , 0.60209 )-- ( 10 , 0.58637 )-- ( 10.5 , 0.57142 )-- ( 11 , 0.55718 )-- ( 11.5 , 0.54362 )-- ( 12 , 0.53068 )-- ( 12.5 , 0.51832 )-- ( 13 , 0.50651 )-- ( 13.5 , 0.49522 )-- ( 14 , 0.4844 )-- ( 14.5 , 0.47404 )-- ( 15 , 0.4641 )-- ( 15.5 , 0.45457 )-- ( 16 , 0.4454 )-- ( 16.5 , 0.4366 )-- ( 17 , 0.42813 )-- ( 17.5 , 0.41997 )-- ( 18 , 0.41212 )-- ( 18.5 , 0.40454 )-- ( 19 , 0.39724 )-- ( 19.5 , 0.39019 )-- ( 20 , 0.38338 )-- ( 20.5 , 0.3768 )-- ( 21 , 0.37044 )-- ( 21.5 , 0.36429 )-- ( 22 , 0.35833 )-- ( 22.5 , 0.35257 )-- ( 23 , 0.34698 )-- ( 23.5 , 0.34156 )-- ( 24 , 0.33631 )-- ( 24.5 , 0.33121 )-- ( 25 , 0.32626 )-- ( 25.5 , 0.32146 )-- ( 26 , 0.31679 )-- ( 26.5 , 0.31226 )-- ( 27 , 0.30784 )-- ( 27.5 , 0.30355 )-- ( 28 , 0.29938 )-- ( 28.5 , 0.29531 )-- ( 29 , 0.29136 )-- ( 29.5 , 0.2875 )-- ( 30 , 0.28375 )-- ( 30.5 , 0.28008 )-- ( 31 , 0.27651 )-- ( 31.5 , 0.27303 )-- ( 32 , 0.26963 )-- ( 32.5 , 0.26632 )-- ( 33 , 0.26308 )-- ( 33.5 , 0.25992 )-- ( 34 , 0.25683 )-- ( 34.5 , 0.25382 )-- ( 35 , 0.25087 )-- ( 35.5 , 0.24799 )-- ( 36 , 0.24517 )-- ( 36.5 , 0.24241 )-- ( 37 , 0.23971 )-- ( 37.5 , 0.23707 )-- ( 38 , 0.23449 )-- ( 38.5 , 0.23196 )-- ( 39 , 0.22949 )-- ( 39.5 , 0.22706 )-- ( 40 , 0.22468 )-- ( 40.5 , 0.22236 )-- ( 41 , 0.22007 )-- ( 41.5 , 0.21784 )-- ( 42 , 0.21564 )-- ( 42.5 , 0.21349 )-- ( 43 , 0.21138 )-- ( 43.5 , 0.20932 )-- ( 44 , 0.20728 )-- ( 44.5 , 0.20529 )-- ( 45 , 0.20334 )-- ( 45.5 , 0.20142 )-- ( 46 , 0.19953 )-- ( 46.5 , 0.19768 )-- ( 47 , 0.19586 )-- ( 47.5 , 0.19408 )-- ( 48 , 0.19232 )-- ( 48.5 , 0.1906 )-- ( 49 , 0.1889 )-- ( 49.5 , 0.18724 )-- ( 50 , 0.1856 ) ;
   
\end{tikzpicture}
\caption{M/M/$1$ with $\lambda=0.8$ and $\mu=1$. The asymptotic correlation $r(\alpha,1)$, between the estimators $\hat{\varphi}_n(\alpha;\hat{\psi}_n)$ and $\hat{\varphi}_n(1;\hat{\psi}_n)$. The sampling rate is $\xi=1$.}\label{fig:corr_alpha_xi}
\end{figure}
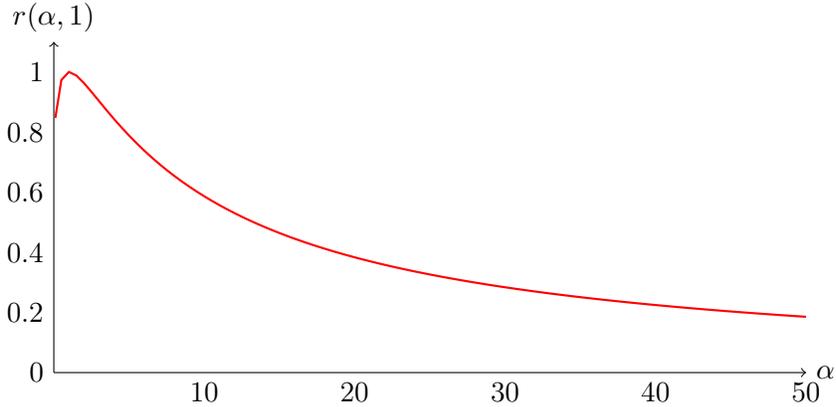

We used the two-step estimation method of $\psi(\xi)$ introduced in Section \ref{sec:levy} to test the performance of the estimator $\hat{\varphi}_n$ for L\'evy input. The data was simulated for an input process with parameter vector $(\lambda,\eta,\mu,d,\sigma,\beta,\gamma)=(0.2,1.2,0.5,-1,0.1,1,5)$ and initial workload $V(0)=0$. In Figure \ref{fig:phi_levy} the estimator is based on $m=200$ observations with a threshold $\tau=2$ for different sampling rates. As before, accuracy improves with a slow sampling rate and decreases as $\alpha$ grows. For a fixed sampling rate of $\xi=1$, Figure \ref{fig:phi_levy2} illustrates the estimated functions for different threshold levels for the biased intermediate estimation step. Both low and high thresholds give less accurate estimators than the intermediate level, $\tau=2$ in this example. This is because a low threshold has a large bias, whereas a high threshold has a small bias but a smaller sample size and thus a larger variance. 

\begin{figure}
\centering
\begin{tikzpicture}[xscale=1,yscale=0.9]
 \def\xmin{0}
 \def\xmax{10}
 \def\ymin{-2.5}
 \def\ymax{1.5}
  \draw[->] (\xmin,0) -- (\xmax,0) node[right] {$\alpha$} ;
  \draw[->] (\xmin,\ymin) -- (\xmin,\ymax) node[above] {$\hat{\varphi}_n(\alpha;\tilde{\psi}_n)-\varphi(\alpha)$} ;
  \foreach \x in {5,10}
  \node at (\x,\ymin) [below] {\x};
  \foreach \y in {-2,-1,0,1}
  \node at (\xmin,\y) [left] {\y};
 
  \draw[red,thick,smooth] ( 	0.1 , 	0.000949999999999999	)-- 	( 	0.5 , 	0.01767	)-- 	( 	1 , 	0.03726	)-- 	( 	1.5 , 	0.05640	)-- 	( 	2 , 	0.07786	)-- 	( 	2.5 , 	0.10092	)-- 	( 	3 , 	0.12431	)-- 	( 	3.5 , 	0.14697	)-- 	( 	4 , 	0.16812	)-- 	( 	4.5 , 	0.18717	)-- 	( 	5 , 	0.20383	)-- 	( 	5.5 , 	0.21786	)-- 	( 	6 , 	0.22918	)-- 	( 	6.5 , 	0.23774	)-- 	( 	7 , 	0.24357	)-- 	( 	7.5 , 	0.24671	)-- 	( 	8 , 	0.24722	)-- 	( 	8.5 , 	0.24515	)-- 	( 	9 , 	0.24060	)-- 	( 	9.5 , 	0.23360	)-- 	( 	10 , 	0.22423	);
  
  \draw[blue,thick, dotted] ( 	0.1 , 	0.01728	)-- 	( 	0.5 , 	0.07348	)-- 	( 	1 , 	0.14877	)-- 	( 	1.5 , 	0.22741	)-- 	( 	2 , 	0.30554	)-- 	( 	2.5 , 	0.38244	)-- 	( 	3 , 	0.45836	)-- 	( 	3.5 , 	0.53374	)-- 	( 	4 , 	0.60894	)-- 	( 	4.5 , 	0.68421	)-- 	( 	5 , 	0.75974	)-- 	( 	5.5 , 	0.83565	)-- 	( 	6 , 	0.91203	)-- 	( 	6.5 , 	0.98892	)-- 	( 	7 , 	1.06637	)-- 	( 	7.5 , 	1.14439	)-- 	( 	8 , 	1.22299	)-- 	( 	8.5 , 	1.30215	)-- 	( 	9 , 	1.38189	)-- 	( 	9.5 , 	1.46217	)-- 	( 	10 , 	1.54298	);						
  
   \draw[purple,thick,densely dashed] ( 	0.1 , 	-0.00665	)-- 	( 	0.5 , 	-0.05175	)-- 	( 	1 , 	-0.10388	)-- 	( 	1.5 , 	-0.14967	)-- 	( 	2 , 	-0.19059	)-- 	( 	2.5 , 	-0.22783	)-- 	( 	3 , 	-0.26224	)-- 	( 	3.5 , 	-0.29445	)-- 	( 	4 , 	-0.32501	)-- 	( 	4.5 , 	-0.35441	)-- 	( 	5 , 	-0.38306	)-- 	( 	5.5 , 	-0.41133	)-- 	( 	6 , 	-0.43952	)-- 	( 	6.5 , 	-0.46792	)-- 	( 	7 , 	-0.49672	)-- 	( 	7.5 , 	-0.52613	)-- 	( 	8 , 	-0.55627	)-- 	( 	8.5 , 	-0.58730	)-- 	( 	9 , 	-0.61931	)-- 	( 	9.5 , 	-0.65238	)-- 	( 	10 , 	-0.68660	);					

  \draw[green,thick,dashdotted] ( 	0.1 , 	0.03724	)-- 	( 	0.5 , 	0.12742	)-- 	( 	1 , 	0.21698	)-- 	( 	1.5 , 	0.30840	)-- 	( 	2 , 	0.40247	)-- 	( 	2.5 , 	0.49794	)-- 	( 	3 , 	0.59413	)-- 	( 	3.5 , 	0.69088	)-- 	( 	4 , 	0.78822	)-- 	( 	4.5 , 	0.88625	)-- 	( 	5 , 	0.98510	)-- 	( 	5.5 , 	1.08485	)-- 	( 	6 , 	1.18553	)-- 	( 	6.5 , 	1.28716	)-- 	( 	7 , 	1.38973	)-- 	( 	7.5 , 	1.49319	)-- 	( 	8 , 	1.59750	)-- 	( 	8.5 , 	1.70256	)-- 	( 	9 , 	1.80833	)-- 	( 	9.5 , 	1.91472	)-- 	( 	10 , 	2.02165	);																																	
     
\end{tikzpicture}

\begin{tikzpicture}
  \begin{customlegend}
  [legend entries={$\xi=0.5$,$\xi=1$,$\xi=5$,$\xi=10$},legend columns=-1,legend style={/tikz/every even column/.append style={column sep=0.8cm}}]  
  \addlegendimage{red,thick,smooth} 
  \addlegendimage{blue,thick, dotted}  
  \addlegendimage{purple,thick,densely dashed}   
  \addlegendimage{green,thick,dashdotted} 
  \end{customlegend}
\end{tikzpicture}
\caption{L\'evy input with $(\lambda,\eta,\mu,d,\sigma,\beta,\gamma)=(0.2,1.2,0.5,-1,0.1,1,5)$. The estimation error of the exponent function, $\hat{\varphi}_n(\alpha;\tilde{\psi}_n)-\varphi(\alpha)$, based on a sample of $m=200$ observations for different sampling rates $\xi\in\{0.5,1,5,10\}$ and threshold $\tau=2$.}\label{fig:phi_levy}
\end{figure}
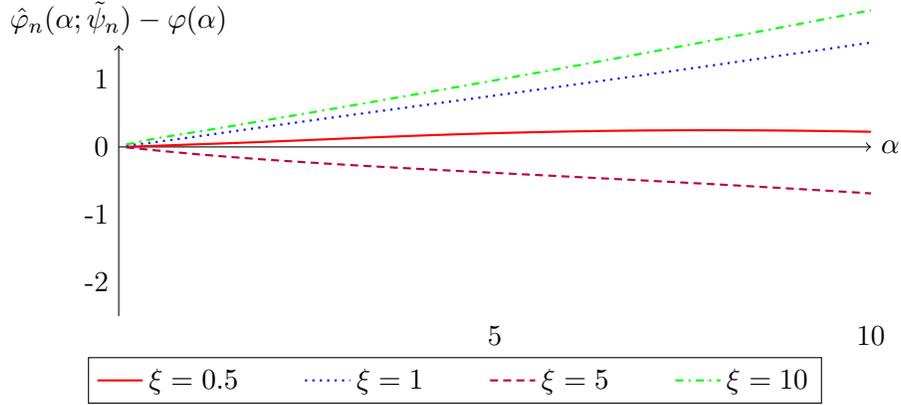

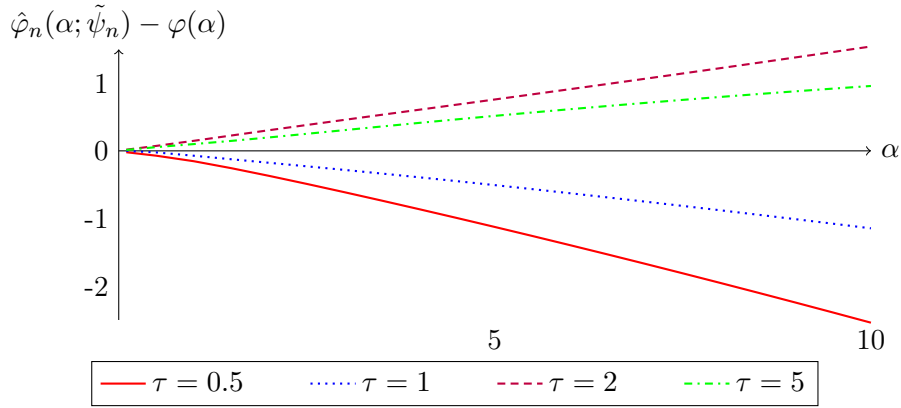
\begin{figure}
\centering
\begin{tikzpicture}[xscale=1,yscale=0.9]
 \def\xmin{0}
 \def\xmax{10}
 \def\ymin{-2.5}
 \def\ymax{1.5}
  \draw[->] (\xmin,0) -- (\xmax,0) node[right] {$\alpha$} ;
  \draw[->] (\xmin,\ymin) -- (\xmin,\ymax) node[above] {$\hat{\varphi}_n(\alpha;\tilde{\psi}_n)-\varphi(\alpha)$} ;
  \foreach \x in {5,10}
  \node at (\x,\ymin) [below] {\x};
  \foreach \y in {-2,-1,0,1}
  \node at (\xmin,\y) [left] {\y};

  \draw[red,thick,smooth]  ( 	0.1 , 	-0.02021	)-- 	( 	0.5 , 	-0.07175	)-- 	( 	1 , 	-0.15349	)-- 	( 	1.5 , 	-0.25681	)-- 	( 	2 , 	-0.36998	)-- 	( 	2.5 , 	-0.48866	)-- 	( 	3 , 	-0.61105	)-- 	( 	3.5 , 	-0.73626	)-- 	( 	4 , 	-0.86382	)-- 	( 	4.5 , 	-0.99352	)-- 	( 	5 , 	-1.12518	)-- 	( 	5.5 , 	-1.25871	)-- 	( 	6 , 	-1.39407	)-- 	( 	6.5 , 	-1.53122	)-- 	( 	7 , 	-1.67014	)-- 	( 	7.5 , 	-1.81080	)-- 	( 	8 , 	-1.95321	)-- 	( 	8.5 , 	-2.09735	)-- 	( 	9 , 	-2.24321	)-- 	( 	9.5 , 	-2.39080	)-- 	( 	10 , 	-2.54011	);		
  
  \draw[blue,thick, dotted]  ( 	0.1 , 	-0.00039	)-- 	( 	0.5 , 	-0.02325	)-- 	( 	1 , 	-0.07203	)-- 	( 	1.5 , 	-0.12441	)-- 	( 	2 , 	-0.17643	)-- 	( 	2.5 , 	-0.22859	)-- 	( 	3 , 	-0.28160	)-- 	( 	3.5 , 	-0.33583	)-- 	( 	4 , 	-0.39139	)-- 	( 	4.5 , 	-0.44832	)-- 	( 	5 , 	-0.50654	)-- 	( 	5.5 , 	-0.56597	)-- 	( 	6 , 	-0.62654	)-- 	( 	6.5 , 	-0.68818	)-- 	( 	7 , 	-0.75081	)-- 	( 	7.5 , 	-0.81438	)-- 	( 	8 , 	-0.87886	)-- 	( 	8.5 , 	-0.94421	)-- 	( 	9 , 	-1.01040	)-- 	( 	9.5 , 	-1.07742	)-- 	( 	10 , 	-1.14527	);	
  
   \draw[purple,thick,densely dashed]  ( 	0.1 , 	0.01728	)-- 	( 	0.5 , 	0.07348	)-- 	( 	1 , 	0.14877	)-- 	( 	1.5 , 	0.22741	)-- 	( 	2 , 	0.30554	)-- 	( 	2.5 , 	0.38244	)-- 	( 	3 , 	0.45836	)-- 	( 	3.5 , 	0.53374	)-- 	( 	4 , 	0.60894	)-- 	( 	4.5 , 	0.68421	)-- 	( 	5 , 	0.75974	)-- 	( 	5.5 , 	0.83565	)-- 	( 	6 , 	0.91203	)-- 	( 	6.5 , 	0.98892	)-- 	( 	7 , 	1.06637	)-- 	( 	7.5 , 	1.14439	)-- 	( 	8 , 	1.22299	)-- 	( 	8.5 , 	1.30215	)-- 	( 	9 , 	1.38189	)-- 	( 	9.5 , 	1.46217	)-- 	( 	10 , 	1.54298	);		 		

  \draw[green,thick,dashdotted]  ( 	0.1 , 	0.01277	)-- 	( 	0.5 , 	0.05551	)-- 	( 	1 , 	0.10051	)-- 	( 	1.5 , 	0.14767	)-- 	( 	2 , 	0.19870	)-- 	( 	2.5 , 	0.25211	)-- 	( 	3 , 	0.30642	)-- 	( 	3.5 , 	0.36061	)-- 	( 	4 , 	0.41406	)-- 	( 	4.5 , 	0.46642	)-- 	( 	5 , 	0.51756	)-- 	( 	5.5 , 	0.56740	)-- 	( 	6 , 	0.61594	)-- 	( 	6.5 , 	0.66316	)-- 	( 	7 , 	0.70912	)-- 	( 	7.5 , 	0.75383	)-- 	( 	8 , 	0.79731	)-- 	( 	8.5 , 	0.83957	)-- 	( 	9 , 	0.88065	)-- 	( 	9.5 , 	0.92055	)-- 	( 	10 , 	0.95930	);				
     
\end{tikzpicture}

\begin{tikzpicture}
  \begin{customlegend}
  [legend entries={ $\tau=0.5$,$\tau=1$,$\tau=2$,$\tau=5$},legend columns=-1,legend style={/tikz/every even column/.append style={column sep=0.8cm}}]  
  
  \addlegendimage{red,thick,smooth}  
  \addlegendimage{blue,thick, dotted}    
  \addlegendimage{purple,thick,densely dashed}					
  \addlegendimage{green,thick,dashdotted}  
  \end{customlegend}
\end{tikzpicture}
\caption{L\'evy input with $(\lambda,\eta,\mu,d,\sigma,\beta,\gamma)=(0.2,1.2,0.5,-1,0.1,1,5)$.The estimation error of the exponent function, $\hat{\varphi}_n(\alpha;\tilde{\psi}_n)-\varphi(\alpha)$, based on a sample of $m=200$ observations for different thresholds $\tau\in\{0.5,1,2,5\}$ and sampling rate $\xi=1$.}\label{fig:phi_levy2}
\end{figure}

\section{Discussion and concluding remarks}\label{sec:discussion}

In this paper we have studied the problem of statistical inference on the input process to a L\'evy-driven queue, where the workload is observed at Poisson times. An estimation method for the characteristic exponent function of the input process has been introduced. The setting is semi-parametric because although the goal is to estimate a function, the methods employed are parametric (namely, method-of-moments and ML-based). 
Under specific verifiable conditions, consistency and asymptotic normality are proven. 

\vspace{2mm}

\noindent {\it Relation with previous work.} In the introduction we already reflected on the papers \cite{BP1999, HP2006}; we now focus on the relation with the generalized method-of-moments (GMM) framework of \cite{DG2004}. 
Recall that our main estimation device is the Z-estimator $\hat{\varphi}_n(\alpha)$ which is derived by solving an estimation equation that equates the empirical LST of the workload to the conditional expected value after an exponential sampling time. In this respect our work is related to \cite{DG2004} that studies GMM estimation of parameters of Markov processes based on data sampled according to a random process which may be time and state dependent. In our setting, a GMM procedure is used to estimate the vector $\left(\varphi(\alpha_1),\ldots,\varphi(\alpha_p)\right)$ by solving the respective estimation equations \eqref{eq:phi_z_equation} for $i=1,\ldots,p$. 

Our estimation equations can be stated in terms of the generator functions used in \cite{DG2004} with a homogeneous Poisson sampling process. However, there is a need to externally estimate the term $\psi(\xi)$; importantly, this  cannot be achieved simultaneously using the same estimation equations. This is more than just a technical issue associated with the estimation method, and is inherent to the non-parametric setting of estimating the function $\varphi$, as opposed to finite-dimensional parameter estimation. Another issue is that we make inference on the input process without observing it directly but rather from observations of the reflected process. Finally, the asymptotic analysis of \cite{DG2004} only deals with asymptotic normality for the case of a one-dimensional parameter space, where we have derived the multi-dimensional asymptotic distribution of the estimation errors and provided explicit terms for all associated asymptotic (co-)variance terms.

\vspace{2mm}

\noindent {\it Identifiability for CP input.} For the CP case one can in principle identify the $\lambda$ and $G^*(\alpha)$ separately from the estimates $\hat{\varphi}_n(\alpha;\hat{\psi}_n)$. First one sets, for some large value of $\alpha_+$,
\begin{equation}\hat\lambda_n = \alpha_+- \hat{\varphi}_n(\alpha_+;\hat{\psi}_n),\end{equation}
and then, for any $\alpha$,
\begin{equation}
\hat G_n^*(\alpha) = 1 +\frac{1}{\hat\lambda} \left(\hat{\varphi}_n(\alpha;\hat{\psi}_n)-\alpha\right)\ .
\end{equation}

\vspace{2mm}

\noindent {\it Parametric special cases.}
Our estimation method is directly applicable to various parametric special cases using the transient moment equations. For example, when the input is Brownian motion one has $\varphi(\alpha)=-d \alpha +\frac{1}{2}\alpha^2\sigma^2$. In this case the first two conditional moment equations, \eqref{eq:EV_transient} and \eqref{eq:EV2_transient}, can be used to construct estimators for $d$ and $\sigma^2$. For the M/M/$1$ queue jump sizes follow an exponential distribution with rate $\mu$. Therefore, $G^*(\alpha)={\mu}/({\mu+\alpha})$, which yields \begin{equation}
\varphi(\alpha)=\frac{\alpha(\mu+\alpha-\lambda)}{\mu+\alpha},\:\:\:\:\psi(\xi)=\frac{1}{2}\left(\xi+\lambda-\mu+\sqrt{(\xi+\lambda-\mu)^2+4\xi\mu}\right).
\end{equation}
In Section \ref{sec:MG1_MLE} for the subordinator input case an MLE for $\psi(\xi)$ was derived. This can be used to estimate the two parameters of an M/M/$1$, or any other two-parameter formulation for that matter, by using only the first moment equation \eqref{eq:EV_transient}. This can be done similarly for other types of parametric subordinator input such as a Gamma process.

\vspace{2mm}
\noindent {\it Bootstrap estimation.}
The asymptotic variance of the estimation errors (in Theorems \ref{thm:phiZ_asymp_norm}, \ref{thm:phiZ_asymp_norm_multi}, \ref{thm:psi_MLE_asymp_norm}, and \ref{thm:psi_MLE_Z_asymp_norm}, as well as  Lemmas \ref{lemma:Jn_normal} and \ref{lemma:ell1_normal}) are all expressed in terms of the stationary distribution of workload $V$, which is unknown. For example, the variance of the estimation error of the MLE $\hat{\psi}_n$ in Theorem \ref{thm:psi_MLE_asymp_norm} is
\begin{equation}
I_\xi:=\E\left[\frac{\frac{\xi}{\psi(\xi)}e^{-\psi(\xi) V}\left(\frac{1}{\psi(\xi)}+V\right)^2}{1-\frac{\xi}{\psi(\xi)}e^{-\psi(\xi) V}}\right] \ ,
\end{equation}
which cannot be directly computed without knowing the distribution of $V$. A reasonable approach to overcome this issue is to apply a bootstrap procedure: given a sample of size $n$ randomly select a sample of size $m<n$ with replacement and compute the sample mean of the desired term, e.g.\ $I_\xi$,  and repeat the above a large  number of times and compute the mean of all results. Such a bootstrap procedure is known to be consistent. As a consequence, for large $n$ and $m$ the method has good accuracy. We refer to e.g.\ \cite[Ch. 23]{book_vdV1998} for general results, and \cite{H2005} for a specific implementation dealing with the stationary workload of an M/G/1 queue.

\vspace{2mm}

\noindent {\it Future research.}
There are several interesting questions that are left open in this work. First of all, the question whether a consistent estimator for $\psi(\xi)$ can be constructed for the general case. Although simulation analysis suggests that in practice the biased estimator performs well, it is still of interest to find a consistent estimator for theoretical and practical efficiency purposes. Another issue is that the object of estimation is the function $\varphi$ but all asymptotic results are in terms of the function at a finite collection of points. A natural extension would be to derive functional limit theorems, for example uniform-consistency and a FCLT approximation for the estimation error as a Gaussian process. In Section \ref{sec:simulation} we observed that the asymptotic variance appears to be increasing and unbounded with $\alpha$ which may lead to technical difficulties, and in particular it is likely that the estimator is not uniformly-consistent.

\vspace{2mm}

\noindent {\it Applications.}
We conclude this paper by mentioning a few applications for which the model and methods presented here may be useful.
\begin{enumerate}
\item[$\circ$] \textit{Optimal sampling rate}. As we saw, explicitly in Section \ref{sec:MG1_MLE} and numerically in Section \ref{sec:simulation}, the accuracy of the estimators improves as the sampling rate decreases. This, however, comes at a cost of a longer sampling duration. In applications where there is a cost associated with the duration of the learning period of the system, this leads to the problem of finding an optimal sampling rate that balances this tradeoff. Our results can be used to derive an approximate solution.
\item[$\circ$] \textit{Stability detection}. Throughout this work we assumed that $\E X(1)<0$ and consequently the workload process is stable. However when the input is unknown, we do not know {\it a priori} whether $\E X(1)<0$. For example if the processing speed can be controlled in an M/G/$1$ queue but the input is unknown then it is possible that for some processing speeds the queue is unstable. In this case a likelihood ratio test can be constructed using the setup of Section \ref{sec:MG1_MLE}. In the general case, the biased estimator $\hat{\vartheta}_n$ for $-\E X(1)$ (Section \ref{sec:levy}) can be used to set up a threshold rule for detecting instability.
%
\item[$\circ$] \textit{Dynamic pricing}. Suppose that the server can charge an admission price with the goal of maximizing social welfare or revenue. The optimal pricing scheme obviously depends on the input distribution \cite{book_HH2003}, which may not be known. This naturally leads to a model with an exploration-exploitation element, and using our framework the exploration step can be carried out for transient observations without relying on any approximations. The extension to the dynamic setting requires a state-dependent input model such as \cite{BBR2009}.
\item[$\circ$] \textit{Testing distributional assumptions}. An issue that is related to the two previous points is that often controllable system parameters are chosen with the goal of maximizing specific performance measures. This is typically done under some modeling assumptions. Our framework enables testing whether the input to the system is consistent with the modeling assumptions (as long as the input alternatives are within the class of spectrally positive Lévy processes).
\item[$\circ$] \textit{Service level monitoring}. Service systems are often obliged to provide performance guarantees. For example, when an Internet Service Provider commits to certain bandwidth, it is important for users and regulators to be able to estimate the performance of the system in order to monitor whether the contract is upheld or not \cite[Ch. IX]{book_H2016}. Our method allows to estimate performance measures, such as moments and tail probabilities, using a transient sample of observations.
\end{enumerate}

\section*{Acknowledgments}\label{sec:acknowledgments}
We wish to thank Offer Kella for his valuable advice and comments and two anonymous reviewers for their helpful comments. This research was carried out as part of the NWO Gravitation project {\sc Networks}, grant number 024.002.003.


\pagebreak

\pagenumbering{arabic}
\setcounter{page}{1}

\noindent \LARGE{Online Appendix: Estimating the input of a L\'evy-driven queue by Poisson sampling of the workload process}
\normalsize

\begin{appendices}
\section{Proofs for Section \ref{sec:phi_z_estimator}}\label{sec:appA}

\begin{proof}[Proof of Lemma \ref{lemma:Jn_limits}] We first prove \eqref{eq:Jn_dpsi}.
From \eqref{eq:Jn},
\begin{equation}
\left.\frac{\partial}{\partial \psi}J_n({\psi},\varphi)\right|_{\varphi=\varphi(\alpha),\psi=\psi(\xi)}=-\frac{\xi}{\xi-\varphi(\alpha)}\cdot \frac{1}{n}\sum_{i=1}^n\left(\frac{\alpha}{{\psi(\xi)}^2} e^{-{\psi(\xi)} V_{i-1}}+\frac{\alpha V_{i-1}}{{\psi(\xi)}}e^{-{\psi(\xi)} V_{i-1}}\right)\ .
\end{equation}
Applying the GPK formula \eqref{eq:GPK}, PASTA and the law of large numbers,
\begin{equation}
\begin{split}
\frac{\partial}{\partial \psi}J_n({\psi(\xi)},\varphi(\alpha)) & {\asarrow} -\frac{\xi}{\xi-\varphi(\alpha)}\left(
\frac{\alpha}{{\psi(\xi)}^2} \E e^{-{\psi(\xi)} V_{i-1}}
+
\frac{\alpha }{{\psi(\xi)}}
\xi\,\E[V e^{-{\psi(\xi)} V}]\right) \ \\&=
-\frac{\alpha\varphi'(0)}{(\xi-\varphi(\alpha))\psi'(\xi)}=\partial J_\psi\ ,
\end{split}
\end{equation}
as $n\to\infty$,
where it is used that
\[
\E [V e^{-\alpha V}] = \frac{\varphi'(0)}{(\varphi(\alpha))^2}\big(\alpha \varphi'(\alpha)-\varphi(\alpha)\big)\ .
\]
Now we prove \eqref{eq:Jn_dphi}. Following the same steps as above, 
\begin{equation}
\left.\frac{\partial}{\partial \varphi}J_n({\psi},\varphi)\right|_{\varphi=\varphi(\alpha),\psi=\psi(\xi)}\asarrow -\frac{\alpha\varphi\apost (0)}{\varphi(\alpha)(\xi-\varphi(\alpha))}=:
\partial J_\varphi \ ,
\end{equation}
as $n\to\infty$.
\end{proof}

\begin{proof} [Proof of Lemma \ref{lemma:Jn_normal}]
\begin{enumerate}
\item[a.] We first compute the asymptotic variance of $\sqrt{n}J_n({\psi(\xi)},\varphi(\alpha))$ and then apply the martingale CLT of Lemma \ref{lemma:MCLT}a to verify that the limiting distribution is normal. Let
\begin{equation}
Z_i:=e^{-\alpha V_i}-\frac{\xi}{\xi-\varphi(\alpha)}\left(e^{-\alpha V_{i-1}}-\frac{\alpha}{{\psi(\xi)}}e^{-{\psi(\xi)} V_{i-1}}\right)\ .
\end{equation} 
Then $J_n({\psi(\xi)},\varphi(\alpha))=\frac{1}{n}\sum_{i=1}^n Z_i$. By \eqref{eq:V_LST_transient}, $\E[Z_i|V_{i-1}]=0$. Therefore,
\begin{equation}
n\Var(J_n({\psi(\xi)},\varphi(\alpha)))=\frac{1}{n}\E\left(\sum_{i=1}^n Z_i\right)^2=\frac{1}{n}\left(\sum_{i=1}^n \E Z_i^2+2\sum_{i=2}^n\sum_{j<i} \E Z_i Z_j\right)\ ,
\end{equation}
and as
\begin{equation}
\E Z_i Z_j=\E[\E( Z_i Z_j|(V_0,\ldots,V_{i-1}))]=\E[Z_j\E( Z_i |V_{i-1})]=0\quad \forall j<i\ ,
\end{equation}
we have that $n\Var(J_n({\psi(\xi)},\varphi(\alpha)))=\frac{1}{n}\sum_{i=1}^n \E Z_i^2$. The second moment of $Z_i$ is
\begin{equation}
\begin{split}
\E Z_i^2 &= \E\Bigg[ e^{-2\alpha V_i}-\frac{2\xi}{\xi-\varphi(\alpha)}e^{-\alpha V_i}\left(e^{-\alpha V_{i-1}}-\frac{\alpha}{{\psi(\xi)}}e^{-{\psi(\xi)} V_{i-1}}\right) \\
&\quad \quad +\left(\frac{\xi}{\xi-\varphi(\alpha)}\right)^2\left(e^{-2\alpha V_{i-1}}-\frac{2\alpha}{{\psi(\xi)}}e^{-(\alpha+{\psi(\xi)}) V_{i-1}}+\left(\frac{\alpha}{{\psi(\xi)}}\right)^2 e^{-2{\psi(\xi)} V_{i-1}}\right) \Bigg] \\
&= \E \Bigg[ e^{-2\alpha V_i}+\left(\frac{\xi}{\xi-\varphi(\alpha)}\right)^2 e^{-2\alpha V_{i-1}} + \left(\frac{\xi\alpha}{{\psi(\xi)}(\xi-\varphi(\alpha))}\right)^2 e^{-2{\psi(\xi)} V_{i-1}} \\
&\quad \quad -\frac{2\xi}{\xi-\varphi(\alpha)}e^{-\alpha V_i-\alpha V_{i-1}} + \frac{2\xi\alpha}{{\psi(\xi)}(\xi-\varphi(\alpha))} e^{-\alpha V_i-{\psi(\xi)} V_{i-1}} -\frac{2\xi^2\alpha}{{\psi(\xi)}(\xi-\varphi(\alpha))^2} e^{-(\alpha+{\psi(\xi)})V_{i-1}} \Bigg] \ .
\end{split}
\end{equation}
The limit of $\E e^{-\alpha V_i}$ as $i\to\infty$ is given by PASTA and the GPK formula \eqref{eq:GPK}. The joint distribution of $(V_i,V_{i-1})$ is obtained by conditioning on $V_{i-1}$ and applying \eqref{eq:V_LST_transient}, for any $\alpha>0$ and $\beta>0$,
\begin{equation}
\begin{split}
\E e^{-\alpha V_i-\beta V_{i-1}} &= \E\left[e^{-\beta V_{i-1}}\frac{\xi}{\xi-\varphi(\alpha)}\left(e^{-\alpha V_{i-1}}-\frac{\alpha}{{\psi(\xi)}}e^{-{\psi(\xi)} V_{i-1}}\right)\right] \\
&= \frac{\xi}{\xi-\varphi(\alpha)}\E\left[e^{-(\alpha+\beta) V_{i-1}}-\frac{\alpha}{{\psi(\xi)}}e^{-({\psi(\xi)}+\beta) V_{i-1}}\right] \\
& \xrightarrow{i\to\infty} \frac{\xi\varphi\apost (0)}{\xi-\varphi(\alpha)}\left(\frac{\alpha+\beta}{\varphi(\alpha+\beta)}-\frac{\alpha({\psi(\xi)}+\beta)}{{\psi(\xi)}\varphi({\psi(\xi)}+\beta)}\right) \ .
\end{split}
\end{equation}
Therefore,
\begin{equation}
\begin{split}
\lim_{i\to\infty}\E Z_i^2 &= \left(1+\left(\frac{\xi}{\xi-\varphi(\alpha)}\right)^2\right)\frac{2\alpha\varphi\apost (0)}{\varphi(2\alpha)} + \left(\frac{\xi\alpha}{{\psi(\xi)}(\xi-\varphi(\alpha))}\right)^2\cdot\frac{2{\psi(\xi)}\varphi\apost (0)}{\varphi(2{\psi(\xi)})} \\
&\quad \quad -\frac{2\xi^2}{(\xi-\varphi(\alpha))^2}\left(\frac{2\alpha\varphi\apost (0)}{\varphi(2\alpha)}-\frac{\alpha(\alpha+{\psi(\xi)})\varphi\apost (0)}{{\psi(\xi)}\varphi(\alpha+{\psi(\xi)})}\right) \\
&\quad \quad + \frac{2\xi^2\alpha}{{\psi(\xi)}(\xi-\varphi(\alpha))^2}\left(\frac{(\alpha+{\psi(\xi)})\varphi\apost (0)}{\varphi(\alpha+{\psi(\xi)})}-\frac{2\alpha\varphi\apost (0)}{\varphi(2{\psi(\xi)})}\right)\\
&\quad \quad -\frac{2\xi^2\alpha}{{\psi(\xi)}(\xi-\varphi(\alpha))^2}\cdot \frac{(\alpha+{\psi(\xi)})\varphi\apost (0)}{\varphi(\alpha+{\psi(\xi)})}\ .
\end{split}
\end{equation}
Rearranging terms we obtain $\lim_{i\to\infty} \E Z_i^2 =\sigma_\alpha^2$, as defined in \eqref{eq:phi_sigma_alpha}. Therefore, we conclude that
\begin{equation}\label{eq:J_EZ2_lim}
n\Var\left(J_n({\psi(\xi)},\varphi(\alpha))\right)=\frac{1}{n}\sum_{i=1}^n\E Z_i^2 \xrightarrow{n\to\infty} \sigma_{\alpha}^2\ .
\end{equation}
We now verify the conditions of Lemma \ref{lemma:MCLT}a.
Let
\begin{equation}
Z_{ni}:=\frac{Z_i}{\sqrt{\Var(nJ_n({\psi(\xi)},\varphi(\alpha)))}}=\frac{nJ_{i}({\psi(\xi)},\varphi(\alpha))-nJ_{i-1}({\psi(\xi)},\varphi(\alpha))}{\sqrt{\Var(nJ_n({\psi(\xi)},\varphi(\alpha)))}}\ , 
\end{equation}
and observe that $nJ_n({\psi(\xi)},\varphi(\alpha))$ is a martingale as $\E[Z_{i}|V_{i-1}]=0$ for every $i\geq 1$. Furthermore, for any $n$, $\Var(nJ_n({\psi(\xi)},\varphi(\alpha)))>0$ and $\P(Z_i<\infty)=1$ as $\E Z_i=0$ for all $i=1,\ldots,n$. Hence, $\E\max_{1\leq i\leq n}Z_{ni}^2<\infty$ and condition \eqref{eq:M_mclt1_finite} holds. Above, we established that $n\Var(J_n({\psi(\xi)},\varphi(\alpha)))$ goes to a constant, and thus
\begin{equation}
\Var(nJ_n({\psi(\xi)},\varphi(\alpha)))=n^2\Var(J_n({\psi(\xi)},\varphi(\alpha)))\to\infty\ ,
\end{equation} 
and as $\P(Z_i<\infty)=1$, we conclude that condition \eqref{eq:M_mclt2_maxtozero} is satisfied. Therefore, the Martingale CLT holds and $\sum_{i=1}^n Z_{ni} {\darrow} \mathrm{N}(0,1)$. Finally, applying Slutsky's lemma yields
\begin{equation}
\sqrt{n}J_n({\psi(\xi)},\varphi(\alpha))=\sqrt{n}\sum_{i=1}^n Z_i=\sqrt{n\Var\left(J_n({\psi(\xi)},\varphi(\alpha))\right)}\sum_{i=1}^n Z_{ni} {\darrow} \sqrt{\sigma_{\alpha}^2}\mathrm{N}(0,1)\ .
\end{equation}

\item[b.] In the previous part of the lemma we established that $J_n({\psi(\xi)},\varphi(\alpha))$ satisfies the conditions of Lemma \ref{lemma:MCLT}a. If we further assume that $\psi_n-\psi(\xi) \approx \frac{1}{n}M_n+R_n$ such that $M_n$ satisfies the conditions of Lemma \ref{lemma:MCLT}a, then 
\begin{equation}
\sqrt{n}\left(\psi_n-{\psi(\xi)}\right)\approx \frac{M_n}{\sqrt{n}}\approx \sqrt{\frac{\Var(M_n)}{n}}\mathrm{N}(0,1)\ ,
\end{equation}
where $\sqrt{{\Var(M_n)}/{n}}\to\sigma_{\xi}$ by condition (ii). We can next apply Lemma \ref{lemma:MCLT}b to conclude joint asymptotic normality:
\begin{equation}
\sqrt{n}\begin{pmatrix}
J_n({\psi(\xi)},\varphi(\alpha)) \\
\psi_n-{\psi(\xi)}
\end{pmatrix} \approx \sqrt{n}\begin{pmatrix}
J_n({\psi(\xi)},\varphi(\alpha)) \\
\frac{M_n}{n}
\end{pmatrix} {\darrow}\mathrm{N}\left(0,S \right) \ ,
\end{equation}
if the covariance matrix $S$ has finite terms. The variance terms are $S_{11}=\sigma_\alpha^2 $ and
\begin{equation}
S_{22}=\lim_{n\to\infty}\E[M_n^2]=\lim_{n\to\infty}\E [(\psi_n-{\psi(\xi)})^2]=\sigma_{\xi}^2\ .
\end{equation}
The asymptotic covariance is finite by condition (iii):
\begin{equation}
\begin{split}
S_{12} &= S_{21}=\lim_{n\to\infty}\E\left[\sqrt{n}J_n({\psi(\xi)},\varphi(\alpha)) \frac{M_n}{\sqrt{n}}\right] \\
&= \lim_{n\to\infty}\frac{1}{n}\sum_{i=1}^n \E\left[Z_i M_n\right] = \sigma_{\alpha,\psi}^2<\infty\ .
\end{split}
\end{equation}
By the Cram\'er-Wold device this implies that any linear combination of $J_n({\psi(\xi)},\varphi(\alpha))$ and $\psi_n-{\psi(\xi)}$ is also normal, and thus \eqref{eq:Jpsi_sum_norm} follows.
\end{enumerate}
\end{proof}

\begin{proof}[Proof of Lemma \ref{lemma:Jn_remainder}]
As the proof is identical for all three second order terms, we restrict ourselves to the first of them: 
\begin{equation}
\sqrt{n}(\varphi_n-\varphi(\alpha))^2
\partial^2_\varphi J
=(\varphi_n-\varphi(\alpha))\sqrt{n}(\varphi_n-\varphi(\alpha))\partial^2_\varphi J,\:\:\:\:\:
\partial^2_\varphi J:=\frac{\partial^2}{\partial \varphi^2}J_n({\psi(\xi)},\varphi(\alpha)).
\end{equation}
By explicit evaluation it is straightforward to verify that $\partial^2_\varphi J<\infty$, almost surely if $\E X(1)<0$, i.e., the stationary workload $V$ exists and is bounded almost surely. The consistency assumption $\psi_n{\parrow}{\psi(\xi)}$ implies, by Theorem \ref{thm:phi_consistency}, that $\varphi_n-\varphi(\alpha){\parrow} 0$. The scaled error term $\sqrt{n}(\varphi_n-\varphi(\alpha))$ is almost surely bounded as the first two moments of the error have finite limits due to Lemma \ref{lemma:Jn_normal}. Therefore the error term indeed diminishes to zero, as $n\to\infty$.
\end{proof}

\begin{proof}[Proof of Theorem \ref{thm:phiZ_asymp_norm_multi}]
In Lemma \ref{lemma:Jn_normal} we verified the martingale CLT conditions of Lemma~\ref{lemma:MCLT}a for every element in ${\boldsymbol J}_n:=(J_n({\psi(\xi)},\varphi(\alpha_1)),\ldots J_n({\psi(\xi)},\varphi(\alpha_p)))$.
Along the same lines, it is shown that any linear combination of the entries of ${\boldsymbol J}_n$ converges to a normal random variable, so that by the Cram\'er-Wold device we conclude joint asymptotic normality.
All that remains is to compute for every pair $\alpha\neq\beta$ the covariance \begin{equation}
\lim_{n\to\infty}
\Cov\Big[\sqrt{n}(\hat{\varphi}_n(\alpha;\psi_n)-\varphi(\alpha)), \sqrt{n}(\hat{\varphi}_n(\beta;\psi_n)-\varphi(\beta))\Big].
\end{equation}
 By \eqref{eq:Jn_approx2}, this covariance has the same limit as
\begin{equation}
\frac{n}{\partial J_{\varphi,\alpha}\partial J_{\varphi,\beta}}\Cov\Big[J_n({\psi(\xi)},\varphi(\alpha))+\partial J_{\psi,\alpha}(\psi_n-{\psi(\xi)}), J_n({\psi(\xi)},\varphi(\beta))+\partial J_{\psi,\beta}(\psi_n-{\psi(\xi)})\Big] \ .
\end{equation}
The covariance can be computed by considering the covariance of the different pairs,
\begin{equation}
\begin{split}
& n\Cov\Big[ J_n({\psi(\xi)},\varphi(\alpha))+\partial J_{\psi,\alpha}(\psi_n-{\psi(\xi)}), J_n({\psi(\xi)},\varphi(\beta))+\partial J_{\psi,\beta}(\psi_n-{\psi(\xi)})\Big] \\
& \quad = n\Cov\left[ J_n({\psi(\xi)},\varphi(\alpha)), J_n({\psi(\xi)},\varphi(\beta))\right] \\
& \quad \quad + n\partial J_{\psi,\beta}\Cov\left[ J_n({\psi(\xi)},\varphi(\alpha)), \psi_n-{\psi(\xi)}\right] \\
& \quad \quad + n\partial J_{\psi,\alpha}\Cov\left[ J_n({\psi(\xi)},\varphi(\beta)), \psi_n-{\psi(\xi)}\right] \\
& \quad \quad + \partial J_{\psi,\alpha}\partial J_{\psi,\beta}\Var\left( \sqrt{n}(\psi_n-{\psi(\xi)})\right) \ .
\end{split}
\end{equation}
The limit of the latter three terms are $\sigma_{\alpha,\xi}^2$, $\sigma_{\beta,\xi}^2$ and $\sigma_{\xi}^2$, respectively. As $\E J_n({\psi(\xi)},\varphi(\alpha))= \E J_n({\psi(\xi)},\varphi(\beta))=0$,
\begin{equation}
\Cov\left( J_n({\psi(\xi)},\varphi(\alpha)), J_n({\psi(\xi)},\varphi(\beta))\right) = \E\left[J_n({\psi(\xi)},\varphi(\alpha)) J_n({\psi(\xi)},\varphi(\beta))\right] \ .
\end{equation}
Similar to the arguments made in the proof of Lemma \ref{lemma:Jn_normal}, by conditioning on $V_{i-1}$ and applying \eqref{eq:V_LST_transient} we obtain that, for every $i>j$,
\begin{align}
& \E\Bigg[\left(e^{-\alpha V_i}-\frac{\xi}{\xi-\varphi(\alpha)}\left(e^{-\alpha V_{i-1}}-\frac{\alpha}{{\psi(\xi)}}e^{-{\psi(\xi)} V_{i-1}}\right)\right) \\
& \quad\:\:\cdot \left(e^{-\beta V_j}-\frac{\xi}{\xi-\varphi(\beta)}\left(e^{-\beta V_{j-1}}-\frac{\beta}{{\psi(\xi)}}e^{-{\psi(\xi)} V_{j-1}}\right)\right)\Bigg]=0 \ ,
\end{align}
and similarly by conditioning on $V_{j-1}$ for every $i<j$ we get a zero expectation for the respective terms. Hence, by \eqref{eq:Jn},
\begin{align}
\nonumber
\lefteqn{\hspace{-12mm}
n\E\left[J_n({\psi(\xi)},\varphi(\alpha)) J_n({\psi(\xi)},\varphi(\beta))\right]}\\&  \nonumber= \frac{1}{n}\sum_{i=1}^n \E\Bigg[\left(e^{-\alpha V_i}-\frac{\xi}{\xi-\varphi(\alpha)}\left(e^{-\alpha V_{i-1}}-\frac{\alpha}{{\psi(\xi)}}e^{-{\psi(\xi)} V_{i-1}}\right)\right) \\
& \quad \quad\quad \quad \quad\quad \:\:\cdot \left(e^{-\beta V_i}-\frac{\xi}{\xi-\varphi(\beta)}\left(e^{-\beta V_{i-1}}-\frac{\beta}{{\psi(\xi)}}e^{-{\psi(\xi)} V_{i-1}}\right)\right)\Bigg] \ .
\end{align}
We compute the expectation for every $i\geq 1$,
\begin{equation}
\begin{split}
& \E\Bigg[\left(e^{-\alpha V_i}-\frac{\xi}{\xi-\varphi(\alpha)}\left(e^{-\alpha V_{i-1}}-\frac{\alpha}{{\psi(\xi)}}e^{-{\psi(\xi)} V_{i-1}}\right)\right) \left(e^{-\beta V_i}-\frac{\xi}{\xi-\varphi(\beta)}\left(e^{-\beta V_{i-1}}-\frac{\beta}{{\psi(\xi)}}e^{-{\psi(\xi)} V_{i-1}}\right)\right)\Bigg] \\
& = \E\Bigg[e^{-(\alpha+\beta)V_i}+\frac{\xi^2}{(\xi-\varphi(\alpha))(\xi-\varphi(\beta))}e^{-(\alpha+\beta)V_{i-1}}+\frac{\alpha\beta\xi^2}{{\psi(\xi)}^2(\xi-\varphi(\alpha))(\xi-\varphi(\beta))}e^{-2{\psi(\xi)} V_{i-1}} \\
& \quad \quad \quad -\frac{\beta\xi^2}{{\psi(\xi)}(\xi-\varphi(\alpha))(\xi-\varphi(\beta))}e^{-(\alpha+{\psi(\xi)})V_{i-1}} -\frac{\xi}{\xi-\varphi(\beta)}e^{-\alpha V_i-\beta V_{i-1}} +\frac{\beta\xi}{{\psi(\xi)}(\xi-\varphi(\beta))}e^{-\alpha V_i-{\psi(\xi)} V_{i-1}}\\
& \quad \quad \quad -\frac{\alpha\xi^2}{{\psi(\xi)}(\xi-\varphi(\alpha))(\xi-\varphi(\beta))}e^{-(\beta+{\psi(\xi)})V_{i-1}} -\frac{\xi}{\xi-\varphi(\alpha)}e^{-\beta V_i-\alpha V_{i-1}} +\frac{\alpha\xi}{{\psi(\xi)}(\xi-\varphi(\alpha))}e^{-\beta V_i-{\psi(\xi)} V_{i-1}} \Bigg] \ ,
\end{split}
\end{equation}
and by applying the GPK formula \eqref{eq:GPK}, with the additional conditioning step used before for the joint limiting distribution of $(V_{i-1},V_i)$, we obtain the limit of the expectation,
\begin{align}
\nonumber
 &\frac{(\alpha+\beta)\varphi\apost (0)}{\varphi(\alpha+\beta)}+\frac{\xi^2\varphi\apost (0)}{(\xi-\varphi(\alpha))(\xi-\varphi(\beta))}\Bigg(\frac{\alpha+\beta}{\varphi(\alpha+\beta)}+\frac{2\alpha\beta}{{\psi(\xi)}\varphi(2{\psi(\xi)})}-\frac{\beta(\alpha+{\psi(\xi)})}{{\psi(\xi)}\varphi(\alpha+{\psi(\xi)})}-\frac{\alpha(\beta+{\psi(\xi)})}{{\psi(\xi)}\varphi(\beta+{\psi(\xi)})} \\
& \hspace{2cm}\ -2\left(\frac{\alpha+\beta}{\varphi(\alpha+\beta)}-\frac{\alpha(\beta+{\psi(\xi)})}{{\psi(\xi)}\varphi(\beta+{\psi(\xi)})}-\frac{\beta(\alpha+{\psi(\xi)})}{{\psi(\xi)}\varphi(\alpha+{\psi(\xi)})}+\frac{2\alpha\beta}{{\psi(\xi)}\varphi(2{\psi(\xi)})} \right)\Bigg).
\end{align}
Rearranging the terms yields the desired expression for $\sigma_{J,\alpha,\beta}^2$.
\end{proof}

\section{Proofs for Section \ref{sec:MG1_MLE}}\label{sec:appB}

\begin{proof}[Proof of Lemma \ref{lemma:smooth_log_likelihood}]
The smoothness is due to the fact that $\ell_n(\psi)$, $\ell_n\apost (\psi)$ and $\ell_n''(\psi)$ are all continuous functions of $\psi$. Furthermore, for every $\xi<\psi<\infty$ the functions are well defined and finite and therefore bounded on any compact interval.
\end{proof}

\begin{proof}[Proof of Lemma \ref{lemma:dlog_likelihood_limits}]
The first derivative of the empirical log-likelihood function is
\begin{equation}
\ell_n\apost(\psi) := \frac{1}{n} L_n'(\psi) = \frac{1}{n}\sum_{i=1}^n \frac{\frac{1}{\psi}+V_{i-1}}{1-\frac{\xi}{\psi}e^{-\psi V_{i-1}}}\left(\frac{\xi}{\psi}e^{-\psi V_{i-1}}-Y_i\right)\ ,
\end{equation}
for any $\psi>\xi$ (where we note that it is otherwise not always well defined).
The limiting (time-average) distribution of the pair $(Y_n,V_{n-1})$ is given by the distribution of a stationary workload observation and the probability to find the system empty in an exponentially distributed time after the observation. Thus, the joint distribution is characterized by $V=V(\infty)$, combined with conditional probability $\P(Y=1\,|\,V)=\frac{\xi}{\psi}e^{-\psi V}$. Applying PASTA for the sampling process $N(t)$,
\begin{equation}
\ell_n\apost (\psi) =\frac{1}{N(t_n)}\ell\apost _{N(t_n)}(\psi){\asarrow} \E\ell_1\apost (\psi)= \E\left[\frac{\frac{1}{\psi}+V}{1-\frac{\xi}{\psi}e^{-\psi V}}\left(\frac{\xi}{\psi}e^{-\psi V}-Y\right)\right]\ ,
\end{equation} 
and by iterating the expectation we obtain \eqref{eq:d_loglikelihood_limit},
\begin{equation}
\ell_n\apost (\psi) {\asarrow} \E(\E(\ell_1\apost (\psi)|V))=\E\left[\frac{\frac{1}{\psi}+V}{1-\frac{\xi}{\psi}e^{-\psi V}}\left(\frac{\xi}{\psi}e^{-\psi V}-\frac{\xi}{{\psi(\xi)}}e^{-{\psi(\xi)} V}\right)\right]\ .
\end{equation}
The same argument yields \eqref{eq:d2_loglikelihood_limit}.
\end{proof}

\begin{proof}[Proof of Lemma \ref{lemma:ell1_normal}] From
\begin{equation}
\begin{split}
\E\ell_n\apost (\psi(\xi) ) &= \frac{1}{n}\sum_{i=1}^n\E\left[\frac{\frac{1}{\psi(\xi) }+V_{i-1}}{1-\frac{\xi}{\psi(\xi) }e^{-\psi(\xi) V_{i-1}}}\left(\frac{\xi}{\psi(\xi) }e^{-\psi(\xi) V_{i-1}}-Y_i\right)\right] \\
&= \frac{1}{n}\sum_{i=1}^n\E\left[\frac{\frac{1}{\psi(\xi) }+V_{i-1}}{1-\frac{\xi}{\psi(\xi) }e^{-\psi(\xi) V_{i-1}}}\left(\frac{\xi}{\psi(\xi) }e^{-\psi(\xi) V_{i-1}}-\frac{\xi}{\psi(\xi) }e^{-\psi(\xi) V_{i-1}}\right)\right] =0\ ,
\end{split}
\end{equation}
we find
\begin{equation}
\Var(\ell_n\apost (\psi(\xi) )) = \frac{1}{n^2}\sum_{i=1}^n\E\left[\frac{\frac{1}{\psi(\xi) }+V_{i-1}}{1-\frac{\xi}{\psi(\xi) }e^{-\psi(\xi) V_{i-1}}}\left(\frac{\xi}{\psi(\xi) }e^{-\psi(\xi) V_{i-1}}-Y_i\right)\right]^2 +\frac{1}{n^2}\sum_{i=1}^n\sum_{j\neq i}s_{ij} \ ,
\end{equation}
where $s_{ij}:=$
\begin{equation}
\E\left[\frac{\frac{1}{\psi(\xi) }+V_{i-1}}{1-\frac{\xi}{\psi(\xi) }e^{-\psi(\xi) V_{i-1}}}\left(\frac{\xi}{\psi(\xi) }e^{-\psi(\xi) V_{i-1}}-Y_i\right)\right]\left[\frac{\frac{1}{\psi(\xi) }+V_{j-1}}{1-\frac{\xi}{\psi(\xi) }e^{-\psi(\xi) V_{j-1}}}\left(\frac{\xi}{\psi(\xi) }e^{-\psi(\xi) V_{j-1}}-Y_j\right)\right]\ .
\end{equation}
By conditioning on $V_{i-1}$ and using the fact that $\P(Y_i=1|V_{i-1})=\frac{\xi}{\psi(\xi) }e^{-\psi(\xi) V_{j-1}}$, where $\{Y_i|V_{i-1}\}$ is independent of $(V_{j-1},Y_j)$ for any $j< i$, we conclude that $s_{ij}=0$ for all $i< j$. Similarly, conditioning on $V_{j-1}$ yields that $s_{ij}=0$ for all $j> i$. We find that
\begin{equation}
\Var(\ell_n\apost (\psi(\xi) ))= \frac{1}{n^2}\sum_{i=1}^n\E\left[\frac{\frac{1}{\psi(\xi) }+V_{i-1}}{1-\frac{\xi}{\psi(\xi) }e^{-\psi(\xi) V_{i-1}}}\left(\frac{\xi}{\psi(\xi) }e^{-\psi(\xi) V_{i-1}}-Y_i\right)\right]^2 \ .
\end{equation}
Observe that, for $i=1,\ldots,n$,
\begin{align}
\nonumber m_i(v)&:=\E\left[\left(\frac{\frac{1}{\psi(\xi) }+V_{i-1}}{1-\frac{\xi}{\psi(\xi) }e^{-\psi(\xi) V_{i-1}}}\right)^2\left(\frac{\xi}{\psi(\xi) }e^{-\psi(\xi) V_{i-1}}-Y_i\right)^2\Big\rvert V_{i-1}=v\right]\\
\nonumber &= \left(\frac{\frac{1}{\psi(\xi) }+v}{1-\frac{\xi}{\psi(\xi) }e^{-\psi(\xi) v}}\right)^2\left(\frac{\xi^2}{\psi(\xi) ^2}e^{-2\psi(\xi) v}-2\frac{\xi}{\psi(\xi) }e^{-\psi(\xi) v}\E[Y_i|V_{i-1}=v]+\E[Y_i^2|V_{i-1}=v]\right) \\
\nonumber &= \left(\frac{\frac{1}{\psi(\xi) }+v}{1-\frac{\xi}{\psi(\xi) }e^{-\psi(\xi) v}}\right)^2\left(\frac{\xi^2}{\psi(\xi) ^2}e^{-2\psi(\xi) v}-2\frac{\xi^2}{\psi(\xi) ^2}e^{-2\psi(\xi) v}+\frac{\xi}{\psi(\xi) }e^{-\psi(\xi) v}\right) \\
&= \frac{\frac{\xi}{\psi(\xi) }e^{-\psi(\xi) v}\left(\frac{1}{\psi(\xi) }+v\right)^2}{1-\frac{\xi}{\psi(\xi) }e^{-\psi(\xi) v}} \ .\end{align}
Applying PASTA once more yields 
\begin{equation}\label{eq:asymp_score_variance}
n\Var(\ell_n\apost (\psi(\xi) ))=\frac{1}{N(t_n)}\sum_{i=1}^{N(t_n)}\E \,m_i(V(t_{i-1})){\asarrow}\E \,m_1(V)=I_\xi\ .
\end{equation}
Next we apply the martingale CLT of Lemma \ref{lemma:MCLT}, using the same notation as before for the corresponding martingale difference terms. For $i\in\{1,\ldots,n\}$, we denote
\begin{equation}
Z_i:=\frac{\left(\frac{1}{\psi(\xi) }+V_{i-1}\right)\left(\frac{\xi}{\psi(\xi) }e^{-\psi(\xi) V_{i-1}}-Y_i\right)}{1-\frac{\xi}{\psi(\xi) }e^{-\psi(\xi) V_{i-1}}}\ ,\:\:\:Z_{ni}:=\frac{Z_i}{\sqrt{\Var(L_n\apost (\psi(\xi) ))}}.
\end{equation} As $\E Z_i=0$, $L_n\apost (\psi(\xi) )=\sum_{i=1}^n Z_i$ is a martingale with respect to $\{V_1,\ldots,V_n\}$. For any $n$, $\Var(L_n\apost (\psi(\xi) ))>0$ and $\P(Z_i<\infty)=1$ as $\E Z_i=0$ for all $i\in\{1,\ldots,n\}$. Hence $\E\max_{1\leq i\leq n}Z_{ni}^2<\infty$ so that the condition \eqref{eq:M_mclt1_finite} holds. We have established that
\begin{equation}
\frac{1}{n}\Var(L_n\apost (\psi(\xi) ))=n\Var(\ell_n\apost (\psi(\xi) )){\asarrow} I_\xi<\infty\ ,
\end{equation}
and therefore $\sqrt{\Var(L_n\apost (\psi(\xi) ))}{\asarrow} \infty$. Furthermore, as $\P(Z_i<\infty)=1$,
\begin{equation}
\max_{1\leq i\leq n}|Z_{ni}|=\max_{1\leq i\leq n}\frac{|Z_i|}{\sqrt{\Var(L_n\apost (\psi(\xi) ))}} {\asarrow} 0
\end{equation}
and thus condition \eqref{eq:M_mclt2_maxtozero} is satisfied as well. Therefore, by Lemma \ref{lemma:MCLT},
\begin{equation}
\frac{L_n\apost (\psi(\xi) )}{\sqrt{\Var(L_n\apost (\psi(\xi) ))}} {\darrow} \mathrm{N}(0,1)\ .
\end{equation}
Recall that $\ell_n\apost (\psi(\xi) )={n}^{-1}\,L_n\apost (\psi(\xi) )$, so that
\begin{equation}
\frac{L_n\apost (\psi(\xi) )}{\sqrt{\Var(L_n\apost (\psi(\xi) ))}}=\frac{n\ell_n\apost (\psi(\xi) )}{\sqrt{\Var(n\ell_n\apost (\psi(\xi) ))}}=\frac{\sqrt{n}\ell_n\apost (\psi(\xi) )}{\sqrt{n\Var(\ell_n\apost (\psi(\xi) ))}}\ .
\end{equation}
Because $\sqrt{n\Var(\ell_n\apost (\psi(\xi) ))}{\asarrow} \sqrt{I_\xi}$, and applying Slutsky's lemma, we obtain the claimed result: $\sqrt{n}\ell_n\apost (\psi(\xi) ) {\darrow} \sqrt{I_\xi}\mathrm{N}(0,1).$
\end{proof}

\begin{proof}[Proof of Lemma \ref{lemma:Dn_limit}]
{Taking derivative of $\ell_n\apost(\psi)$, as given in \eqref{eq:loglikelihood_psi_deriv}, yields
\begin{equation}
\begin{split}
\ell_1''(\psi) &= -\frac{1}{\left(1-\frac{\xi}{\psi(\xi) }e^{-\psi(\xi) V_0}\right)^2}\Bigg[\frac{\xi}{\psi(\xi) }e^{-\psi(\xi) V_0}\left(\frac{1}{\psi(\xi) }+V_0\right)^2(1-Y_1 ) \\
&\quad +\frac{1}{\psi(\xi) ^2}\left(1-\frac{\xi}{\psi(\xi) }e^{-\psi(\xi) V_0}\right)\left(\frac{\xi}{\psi(\xi) }e^{-\psi(\xi) V_0}-Y_1 \right) \Bigg] \\
&= -\frac{\frac{\xi}{\psi(\xi) }e^{-\psi(\xi) V_0}\left(\frac{1}{\psi(\xi) }+v\right)^2}{1-\frac{\xi}{\psi(\xi) }e^{-\psi(\xi) V_0}} \ .
\end{split}
\end{equation}
The conditional expectation for any initial workload $V_0=v$ is
\begin{equation}
\begin{split}
\E[\ell''_1(\psi(\xi) )|V_0=v] &= -\frac{1}{\left(1-\frac{\xi}{\psi(\xi) }e^{-\psi(\xi) v}\right)^2}\Bigg[\frac{\xi}{\psi(\xi) }e^{-\psi(\xi) v}\left(\frac{1}{\psi(\xi) }+v\right)^2(1-\P_{\psi(\xi) }(Y_1|V_0=v) ) \\
&\quad +\frac{1}{\psi(\xi) ^2}\left(1-\frac{\xi}{\psi(\xi) }e^{-\psi(\xi) v}\right)\left(\frac{\xi}{\psi(\xi) }e^{-\psi(\xi) v}-\P_{\psi(\xi) }(Y_1|V_0=v) \right) \Bigg] \\
&= -\frac{\frac{\xi}{\psi(\xi) }e^{-\psi(\xi) v}\left(\frac{1}{\psi(\xi) }+v\right)^2}{1-\frac{\xi}{\psi(\xi) }e^{-\psi(\xi) v}} \ .
\end{split}
\end{equation}
Denote the limit of the second derivative, that is given by \eqref{eq:d2_loglikelihood_limit} in Lemma \ref{lemma:dlog_likelihood_limits}, at value $\psi$ by $k_2(\psi)$. Recall that by \eqref{eq:P_idle_transient}, $\P_{\psi(\xi) }(Y_1|V_0=v)=\frac{\xi}{\psi(\xi) }e^{-\psi(\xi) v}$, and thus by taking expectation with respect to the stationary workload we obtain
\begin{equation}
\ell_n''(\psi) {\asarrow} \E\left[\E[\ell_1''(\psi(\xi) )|V]\right]=k_2(\psi(\xi))=-I_{\xi}\ .
\end{equation}
Now consider the difference between $D_n$ and $-I_\xi$:
\begin{equation}
\begin{split}
|\,D_n-(-I_\xi)\,| &= \Big\rvert \,\int_{0}^{1}\left[\ell_n''(t\hat{\psi}_n+(1-t)\psi(\xi) )+I_\xi )\right] \,{\rm d}t\,\Big\rvert \\
&\leq \int_0^1\Big\rvert\ell_n{''}(t\hat{\psi}_n+(1-t)\psi(\xi) )+I_\xi\Big\rvert \ {\rm d}t \\
& \leq \sup_{t\in[0,1]}\Big\rvert\ell_n{''}(t\hat{\psi}_n+(1-t)\psi(\xi) )+I_\xi\Big\rvert \ .
\end{split}
\end{equation} 
By the consistency of the MLE established in Theorem \ref{thm:psi_MLE_consistency}, $t\hat{\psi}_n+(1-t)\psi(\xi) {\asarrow}\psi(\xi) $ for any $t\in[0,1]$. Hence, for any $\epsilon>0$ there exists some $N$ such that $|t\hat{\psi}_n+(1-t)\psi(\xi) -\psi(\xi) |<\epsilon$, for all $n\geq N$ almost surely, and therefore
\begin{equation}
\begin{split}
|D_n+I_\xi| & \leq \sup_{\psi:\ |\psi-\psi(\xi) |<\epsilon}\Big\rvert\ell_n{''}(\psi)+I_\xi\Big\rvert \\
&\leq \sup_{\psi:\ |\psi-\psi(\xi) |<\epsilon}\Big\rvert\ell_n{''}(\psi)-k_2(\psi)\Big\rvert+\sup_{\psi:\ |\psi-\psi(\xi) |<\epsilon}\Big\rvert k_2(\psi)+I_\xi\Big\rvert \ .
\end{split}
\end{equation} 
By Lemma \ref{lemma:dlog_likelihood_limits}, $\ell_n''(\psi){\asarrow}k_2(\psi)$. Consequently, the first term goes to zero almost surely. Let $\delta(\epsilon):=\sup_{\psi:\ |\psi-\psi(\xi) |<\epsilon}|k_2(\psi)+I_\xi|$, then (as $k_2$ is continuous and bounded on finite intervals) we have that $\delta(\epsilon)\to 0$ as $\epsilon\to 0$. We conclude that
\begin{equation}
|D_n+I_\xi|{\asarrow}0 \ ,
\end{equation}
and the result follows. 
}
\end{proof}

\begin{proof}[Proof of Theorem \ref{thm:psi_MLE_Z_asymp_norm}]
Let $\Omega_\xi(\alpha):=
{\xi}/({\xi-\varphi(\alpha)})$
and
\begin{equation}
\Xi_{\xi,i}(\alpha):=e^{-\alpha V_{i-1}}-\frac{\alpha}{{\psi(\xi)}}e^{-{\psi(\xi)} V_{i-1}}.
\end{equation}
As $\E\left[J_n({\psi(\xi)},\varphi(\alpha))\right]=0$,
\begin{equation}
\Cov(\sqrt{n}J_n({\psi(\xi)},\varphi(\alpha),\sqrt{n}(\hat{\psi}_n-{\psi(\xi)}))=n\E\left[J_n({\psi(\xi)},\varphi(\alpha))(\hat{\psi}_n-{\psi(\xi)})\right]\ .
\end{equation}
Recall that as $n\to\infty$, we almost surely have 
$
\sqrt{n}(\hat{\psi}_n-{\psi(\xi)})\approx {\sqrt{n}\ell_n\apost ({\psi(\xi)})}/{I_\xi}$,
and therefore
\begin{equation}
\lim_{n\to\infty}\Cov(\sqrt{n}J_n({\psi(\xi)},\varphi(\alpha)),\sqrt{n}(\hat{\psi}_n-{\psi(\xi)}))=\frac{1}{I_\xi}\lim_{n\to\infty}n\E\left[J_n({\psi(\xi)},\varphi(\alpha))\ell_n\apost ({\psi(\xi)})\right]\ .
\end{equation}
Direct computations yield that for any finite $n$ the expectation $\E\left[J_n({\psi(\xi)},\varphi(\alpha))\,\ell_n\apost ({\psi(\xi)})\right]$ equals
\begin{equation}
 \frac{1}{n^2}\sum_{i=1}^n\sum_{j=1}^n\E\left[ \frac{\left(e^{-\alpha V_i}-\Omega_\xi(\alpha) \Xi_{\xi,i}(\alpha)\right)\left(\frac{1}{{\psi(\xi)}}+V_{j-1}\right)\left(\frac{\xi}{{\psi(\xi)}}e^{-{\psi(\xi)} V_{j-1}}-Y_j\right)}{1-\frac{\xi}{{\psi(\xi)}}e^{-{\psi(\xi)} V_{j-1}}}\right] \ .
\end{equation}
Conditioning on $(V_0,\ldots,V_{i-1})$,
all terms in the above display with $i> j$ can be written as
\begin{equation}
 \frac{\left(\frac{1}{{\psi(\xi)}}+V_{j-1}\right)\left(\frac{\xi}{{\psi(\xi)}}e^{-{\psi(\xi)} V_{j-1}}-Y_j\right)}{1-\frac{\xi}{{\psi(\xi)}}e^{-{\psi(\xi)} V_{j-1}}}\E\left[e^{-\alpha V_i}-\Omega_\xi(\alpha) \Xi_{\xi,i}(\alpha) \Big\rvert V_{i-1}\right]=0\ ,
\end{equation}
and for the terms such that $i<j$, conditioning on $(V_0,\ldots,V_{j-1})$ yields
\begin{equation}
\left(e^{-\alpha V_i}-\Omega_\xi(\alpha) \Xi_{\xi,i}(\alpha)\right) \E\left[\frac{\left(\frac{1}{{\psi(\xi)}}+V_{j-1}\right)\left(\frac{\xi}{{\psi(\xi)}}e^{-{\psi(\xi)} V_{j-1}}-Y_j\right)}{1-\frac{\xi}{{\psi(\xi)}}e^{-{\psi(\xi)} V_{j-1}}} \Big\rvert V_{j-1}\right]=0\ .
\end{equation}
For $j=i$, we compute the expectation conditioned on $V_{i-1}$,
\begin{align}
\nonumber & \E\left[ \frac{\left(e^{-\alpha V_i}-\Omega_\xi(\alpha) \Xi_{\xi,i}(\alpha)\right)\left(\frac{1}{{\psi(\xi)}}+V_{i-1}\right)\left(\frac{\xi}{{\psi(\xi)}}e^{-{\psi(\xi)} V_{i-1}}-Y_i\right)}{1-\frac{\xi}{{\psi(\xi)}}e^{-{\psi(\xi)} V_{i-1}}}\Big\rvert V_{i-1}\right] \\
\nonumber &\ = \E\Bigg[ \frac{\frac{1}{{\psi(\xi)}}+V_{i-1}}{1-\frac{\xi}{{\psi(\xi)}}e^{-{\psi(\xi)} V_{i-1}}} \left(\frac{\xi}{{\psi(\xi)}}e^{-{\psi(\xi)} V_{i-1}}\E[e^{-\alpha V_i}|V_{i-1}]-\E[Y_i e^{-\alpha V_i}]\right) \\
\nonumber &\ \quad -\Omega_\xi(\alpha) \left( \frac{\Xi_{\xi,i}(\alpha)\left(\frac{1}{{\psi(\xi)}}+V_{i-1}\right)\E\left[\frac{\xi}{{\psi(\xi)}}e^{-{\psi(\xi)} V_{i-1}}-Y_i\Big\rvert V_{i-1}\right]}{1-\frac{\xi}{{\psi(\xi)}}e^{-{\psi(\xi)} V_{i-1}}}\right) \Bigg] \\
\nonumber &\ =\E\Bigg[\frac{\frac{1}{{\psi(\xi)}}+V_{i-1}}{1-\frac{\xi}{{\psi(\xi)}}e^{-{\psi(\xi)} V_{i-1}}} \left(\frac{\xi}{{\psi(\xi)}}e^{-{\psi(\xi)} V_{i-1}}\E[e^{-\alpha V_i}|V_{i-1}]-\P(Y_i=1,V_{i}=0)\right) \Bigg] \\
&\ = \E\Bigg[\frac{\xi^2\left(\frac{1}{{\psi(\xi)}}+V_{i-1}\right)e^{-{\psi(\xi)} V_{i-1}}}{{\psi(\xi)}(\xi-\varphi(\alpha))\left(1-\frac{\xi}{{\psi(\xi)}}e^{-{\psi(\xi)} V_{i-1}}\right)}\left(\Xi_{\xi,i}(\alpha)-\frac{\xi-\varphi(\alpha)}{\xi}\right) \Bigg] \ ,
 \end{align}
where the second equality follows from the fact that given $V_{i-1}$, $Y_i$ is a Bernoulli random variable with probability $\frac{\xi}{{\psi(\xi)}}e^{-{\psi(\xi)} V_{i-1}}$ and equals one only when $V_{i-1}=0$. Finally, by applying PASTA we conclude that
\begin{align}\nonumber
& \lim_{n\to\infty}\Cov(\sqrt{n}J_n({\psi(\xi)},\varphi(\alpha),\sqrt{n}(\hat{\psi}_n-{\psi(\xi)}))\ = \frac{1}{I_\xi}\lim_{n\to\infty}n\E\left[J_n({\psi(\xi)},\varphi(\alpha))\ell_n\apost ({\psi(\xi)})\right] \\
\nonumber &\ = \frac{\xi^2}{{\psi(\xi)}(\xi-\varphi(\alpha))I_\xi}\lim_{n\to\infty}\frac{1}{n}\sum_{i=1}^n\E\left[\frac{\left(\frac{1}{{\psi(\xi)}}+V_{i-1}\right)e^{-{\psi(\xi)} V_{i-1}}}{1-\frac{\xi}{{\psi(\xi)}}e^{-{\psi(\xi)} V_{i-1}}}\left(\Xi_{\xi,i}(\alpha)-\frac{\xi-\varphi(\alpha)}{\xi}\right) \right] \\
&\ = \frac{\xi^2}{{\psi(\xi)}(\xi-\varphi(\alpha))I_\xi}\E\left[\frac{\left(\frac{1}{{\psi(\xi)}}+V\right)e^{-{\psi(\xi)} V}}{1-\frac{\xi}{{\psi(\xi)}}e^{-{\psi(\xi)} V}}\left(e^{-\alpha V}-\frac{\alpha}{{\psi(\xi)}}e^{-{\psi(\xi)} V}-\frac{\xi-\varphi(\alpha)}{\xi}\right) \right] \ .
\end{align}
This completes the proof.
\end{proof}

\section{Proofs for Section \ref{sec:levy}}\label{sec:appC}

\begin{proof}[Proof of Proposition \ref{prop:tau_bias_bound}]
Consider the expected workload after an exponential time $T\sim\exp(\xi)$ given that the initial workload was $v$,
\[
\begin{split}
\E(V(T)|V(0)=v) &= \E\left[V(T)\mathbf{1}\left(\inf_{0\leq s\leq T}X(s)\geq -v\right)|V(0)=v\right] \\
& \ \ \ +\E\left[V(T)\mathbf{1}\left(\inf_{0\leq s\leq T}X(s)< -v\right)|V(0)=v\right] \\
&= \E\left[(v+X(T))\mathbf{1}\left(\inf_{0\leq s\leq T}X(s)\geq -v\right)|V(0)=v\right] \\
& \ \ \ +\P\left(\inf_{0\leq s\leq T}X(s)< -v\right)\E[V(T)|V(0)=0] \\
& \leq v+\E X(T)+\P\left(\inf_{0\leq s\leq T}X(s)< -v\right)\E[V(T)|V(0)=0]\ ,
\end{split}
\]
yielding, with the last step being due to \cite[Lemma 6.2]{book_DM2015}
\begin{align}
\nonumber \E(V(T)-v\,|\,V(0)=v)-\E X(T)&\:\leq\P\left(\inf_{0\leq s\leq T}X(s)< -v\right)\E[V(T)\,|\,V(0)=0]\\
&\:= e^{-\psi(\xi)v}\E[V(T)|V(0)=0]\ .
\end{align}
Consequently, for any threshold $\tau>0$ and all $v\ge \tau$,
\begin{equation}\label{eq:EV_tau_bound}
\E(V(T)-v\,|\,V(0)=v)-\E X(T)\leq \E[V(T)\,|\,V(0)=0]e^{-\psi(\xi)\tau }\ .
\end{equation}
The expected bias of $\hat{\vartheta}_m(\tau)$ is
\begin{align}
\nonumber b_m(\vartheta;\tau) &= \E\left[\hat{\vartheta}_m(\tau)-\vartheta\right]=\E\left[\frac{\xi}{m}\sum_{j=1}^m (V_{i(j)}-V_{i(j)-1})-\left(-\varphi\apost (0)\right)\right] \\
&= \frac{\xi}{m}\sum_{j=1}^m\E\left[ (V_{i(j)}-V_{i(j)-1})-\E X(T)\right] \ .
\end{align}
For every $j=1\ldots,m$, conditioning on $V_{i(j)-1}=v$ (with $v\geq \tau$), applying \eqref{eq:EV_tau_bound} yields
\begin{equation}
\E\left[ (V_{i(j)}-V_{i(j)-1})-\E X(T)\Big\rvert V_{i(j)-1}=v \right] \leq \E[V(T)|V(0)=0]e^{-\psi(\xi)\tau }\ ,
\end{equation}
and the bound \eqref{eq:bias_varphi_prime} follows by total expectation.
\end{proof}

\begin{proof}[Proof of Proposition \ref{prop:tau_psi_bias_bound}]
For any sample ${\boldsymbol V}=(V_0,\ldots,V_n)$ such that $n\geq m$, let
\begin{equation}
f(x,n):=\frac{\xi(V_{n}-V_0)}{n}-\frac{\xi}{n}\sum_{i=1}^n\frac{e^{-x V_{i-1}}}{x}\ ,
\end{equation}
then the estimator $\tilde{\psi}_m$ is given by solving \eqref{eq:levy_psi_root}, i.e., 
$\hat{\vartheta}_m(\tau)=f (\tilde{\psi}_m,M(m,\tau)).$ By \eqref{eq:EV_transient}, for any $\{M(m,\tau)=n\}$, 
\begin{equation}
\E[f(\psi(\xi),n)-f(\tilde{\psi}_m,n)] = \E\left[\hat{\vartheta}_m(\tau)+\varphi\apost (0)\right]\ ,
\end{equation}
thus, the upper bound \eqref{eq:bias_varphi_prime} for the bias of $\hat{\varphi}_m(\tau)$ we obtain yields 
\begin{equation}
\E[f(\psi(\xi),n)-f(\tilde{\psi}_m,n)] \leq \xi e^{-\psi(\xi)\tau}\E[V(T)|V(0)=0]\ .
\end{equation}
As $f(x,n)$ is continuous and monotone decreasing with $x$, this also means that the bias of $f(\tilde{\psi}_n,n)$ from $f(\psi(\xi),n)$ is bounded. The bound further holds when taking expectation over the possible values of $M(m,\tau)$. Moreover, if the threshold $\tau$ is increased, then the bias vanishes, and by continuity so does the bias of the estimator for $\tilde{\psi}$.
\end{proof}

\begin{proof}[Proof of Proposition \ref{prop:tau_asymp}]
We define \begin{equation}K(t):=\sum_{i=1}^{N(t)}\mathbf{1}(V(t_i)\geq \tau),\end{equation} 
where $N(t)$ is the Poisson sampling process and $t_i$ are the event times. By PASTA, as $t\to\infty$,
\begin{equation}
\frac{K(t)}{N(t)}=\frac{1}{N(t)}\sum_{i=1}^{N(t)}\mathbf{1}(V(t_i)\geq \tau){\parrow} \P(V>\tau)\ .
\end{equation}
Observe that $m=K(t_{M(m,\tau)})$, so that for $M(m,\tau)=n$,
\begin{equation}
\begin{split}
\hat{\vartheta}_m(\tau) &= \frac{\xi}{K(t_n)}\sum_{i=1}^{N(t_n)}(V(t_i)-V(t_{i-1}))\mathbf{1}(V(t_{i-1})\geq \tau) \\
&= \xi\frac{N(t_n)}{K(t_n)}\cdot \frac{1}{N(t_n)}\sum_{i=1}^{N(t_n)}(V(t_i)-V(t_{i-1}))\mathbf{1}(V(t_{i-1})\geq \tau) \ .
\end{split}
\end{equation}
The term ${N(t_n)}/{K(t_n)}$ converges to $({\P(V>\tau)})^{-1}$ as $n\to\infty$, and this also holds for $m\to\infty$ as $M(m,\tau)\to\infty$ almost surely in this case. By PASTA, the limiting joint distribution of the indicator and the increment is given by $(V(T)-V,\mathbf{1}(V\geq \tau))$, where $T\sim\exp(\xi)$. Therefore, we can apply Slutsky's lemma to conclude that the limiting distribution of the product is given by the product of the limiting distributions,
\begin{equation}
\hat{\vartheta}_m(\tau){\parrow}\xi\frac{\E\left[\E[(V(T)-V)\mathbf{1}(V\geq \tau)\,\rvert \,V]\right]}{\P(V>\tau)}\ .
\end{equation} 
By first conditioning on $V=v$ and then plugging in the explicit expectation of $\E[V(T)-v]$ given in \eqref{eq:EV_transient} we obtain that for $v\ge \tau$
\begin{equation}
\E\left[(V(T)-V)\mathbf{1}(V\geq \tau)\Big\rvert V=v\right] =
\frac{e^{-\psi(\xi)v}}{\psi(\xi)}-\frac{\varphi\apost (0)}{\xi}\ , 
\end{equation}
and $0$ else. 
Then the asymptotic estimator is
\begin{equation}
\vartheta(\tau) = \frac{\xi}{\P(V>\tau)}\int_{\tau}^{\infty}\left[\frac{e^{-\psi(\xi)v}}{\psi(\xi)}-\frac{\varphi\apost (0)}{\xi}\right]\ {\rm d}F_V(v) = \frac{\xi\E\left[e^{-\psi(\xi)V}\mathbf{1}(V\geq \tau)\right]}{\psi(\xi)\P(V>\tau)}-\varphi\apost (0)\ .
\end{equation}
\end{proof}

\end{appendices}


\begin{thebibliography}{}

\end{thebibliography}


\begin{thebibliography}{10}
\bibliographystyle{abbrv}


\bibitem{AA2013}
P.~Af{\`e}che and B.~Ata.
\newblock Bayesian dynamic pricing in queueing systems with unknown delay cost
  characteristics.
\newblock {\em Manufacturing \& Service Operations Management}, 15(2):292--304,
  2013.

\bibitem{AJP2013}
N.~Antunes, G.~Jacinto, and A.~Pacheco.
\newblock {Probing a M/G/1 queue with general Input and service times}.
\newblock {\em SIGMETRICS Perform. Eval. Rev.}, 41(3):34--36, Jan. 2014.

\bibitem{AJPW2016}
N.~Antunes, G.~Jacinto, A.~Pacheco, and C.~Wichelhaus.
\newblock {Estimation of the traffic intensity in a piecewise-stationary
  M$_t$/G$_t$/1 queue with probing}.
\newblock {\em SIGMETRICS Perform. Eval. Rev.}, 44(2):3--5, Sept. 2016.

\bibitem{AP2005}
M.~Armony and E.~L. Plambeck.
\newblock The impact of duplicate orders on demand estimation and capacity
  investment.
\newblock {\em Management Science}, 51(10):1505--1518, 2005.

\bibitem{ANP2017}
A.~Asanjarani, Y.~Nazarathy, and P.~K. Pollett.
\newblock Parameter and state estimation in queues and related stochastic
  models: A bibliography.
\newblock {\em ArXiv:1701.08338}, 2017.

\bibitem{BBR2009}
R.~Bekker, O.~J. Boxma, and J.~A.~C. Resing.
\newblock L{\'e}vy processes with adaptable exponent.
\newblock {\em Advances in Applied Probability}, 41(1):177--205, 2009.

\bibitem{B2011}
D.~Belomestny.
\newblock {Statistical inference for time-changed L\'{e}vy processes via
  composite characteristic function estimation}.
\newblock {\em The Annals of Statistics}, 39(4):2205--2242, 2011.

\bibitem{BP1999}
N.~H. Bingham and S.~M. Pitts.
\newblock {Nonparametric inference from M/G/l busy periods}.
\newblock {\em Communications in Statistics. Stochastic Models},
  15(2):247--272, 1999.

\bibitem{BNW2013}
N.~Blanghaps, Y.~Nov, and G.~Weiss.
\newblock {Sojourn time estimation in an M/G/$\infty$ queue with partial
  information}.
\newblock {\em Journal of Applied Probability}, 50(4):1044--1056, 2013.

\bibitem{BG2003}
B.~Buchmann and R.~Gr{\"u}bel.
\newblock {Decompounding: an estimation problem for Poisson random sums}.
\newblock {\em The Annals of Statistics}, 31(4):1054--1074, 2003.

\bibitem{book_DM2015}
K.~D\k{e}bicki and M.~Mandjes.
\newblock {\em Queues and L{\'e}vy Fluctuation Theory}.
\newblock Springer, 2015.

\bibitem{DG2004}
D.~Duffie and P.~Glynn.
\newblock {Estimation of continuous-time Markov processes sampled at random
  time intervals}.
\newblock {\em Econometrica}, 72(6):1773--1808, 2004.

\bibitem{book_F1996}
T.~S. Ferguson.
\newblock {\em A Course in Large Sample Theory}.
\newblock Chapman \& Hall, 1996.

\bibitem{GMW1993}
P.~W. Glynn, B.~Melamed, and W.~Whitt.
\newblock Estimating customer and time averages.
\newblock {\em Operations Research}, 41(2):400--408, 1993.

\bibitem{G2016}
A.~Goldenshluger.
\newblock {Nonparametric estimation of the service time distribution in the
  M/G/$\infty$ queue}.
\newblock {\em Advances in Applied Probability}, 48(4):1117--1138, 2016.

\bibitem{G2018}
A.~Goldenshluger.
\newblock {The M/G/$\infty$ estimation problem revisited}.
\newblock {\em Bernoulli}, 24(4A):2531--2568, 2018.

\bibitem{G2009}
S.~Gugushvili.
\newblock {Nonparametric estimation of the characteristic triplet of a
  discretely observed L{\'e}vy process}.
\newblock {\em Journal of Nonparametric Statistics}, 21(3):321--343, 2009.

\bibitem{book_HH1980}
P.~Hall and C.~C. Heyde.
\newblock {\em Martingale Limit Theory and its Application}.
\newblock Academic Press, 1980.

\bibitem{H2005}
M.~B. Hansen.
\newblock {Nonparametric estimation of the stationary M/G/1 workload distribution function}.
\newblock In {\em Proceedings of the 37th conference on Winter simulation},
  pages 869--877. Winter Simulation Conference, 2005.

\bibitem{HP2006}
M.~B. Hansen and S.~M. Pitts.
\newblock {Nonparametric inference from the M/G/1 workload}.
\newblock {\em Bernoulli}, 12:737--759, 2006.

\bibitem{book_H2016}
R.~Hassin.
\newblock {\em Rational Queueing}.
\newblock CRC Press, 2016.

\bibitem{book_HH2003}
R.~Hassin and M.~Haviv.
\newblock {\em {To Queue or Not to Queue: Equilibrium Behavior in Queueing
  Systems}}.
\newblock Springer, 2003.

\bibitem{HM1986}
R.~D.~H. Heijmans and J.~R. Magnus.
\newblock Consistent maximum-likelihood estimation with dependent observations:
  the general (non-normal) case and the normal case.
\newblock {\em Journal of Econometrics}, 32(2):253--285, 1986.

\bibitem{HM1986b}
R.~D.~H. Heijmans and J.~R. Magnus.
\newblock On the first--order efficiency and asymptotic normality of maximum
  likelihood estimators obtained from dependent observations.
\newblock {\em Statistica Neerlandica}, 40(3):169--188, 1986.

\bibitem{book_H1997}
C.~C. Heyde.
\newblock {\em Quasi-Likelihood and its Application: A General Approach to
  Optimal Parameter Estimation}.
\newblock Springer Science \& Business Media, 1997.

\bibitem{KBM2006}
O.~Kella, O.~Boxma, and M.~Mandjes.
\newblock {A L{\'e}vy process reflected at a Poisson age process}.
\newblock {\em Journal of Applied Probability}, 43(1):221--230, 2006.

\bibitem{book_K2006}
A.~Kyprianou.
\newblock {\em Introductory Lectures on Fluctuations of L{\'e}vy Processes with
  Applications}.
\newblock Springer, 2006.

\bibitem{NR2009}
M.~H. Neumann and M.~Rei{\ss}.
\newblock {Nonparametric estimation for L{\'e}vy processes from low-frequency
  observations}.
\newblock {\em Bernoulli}, 15(1):223--248, 2009.

\bibitem{NTV2006}
A.~Novak, P.~Taylor, and D.~Veitch.
\newblock The distribution of the number of arrivals in a subinterval of a busy
  period of a single server queue.
\newblock {\em Queueing Systems}, 53(3):105--114, 2006.

\bibitem{P1994}
S.~M. Pitts.
\newblock {Nonparametric estimation of the stationary waiting time distribution
  function for the GI/G/1 queue}.
\newblock {\em The Annals of Statistics}, 22(3):1428--1446, 1994.

\bibitem{RTP2007}
J.~V. Ross, T.~Taimre, and P.~K. Pollett.
\newblock Estimation for queues from queue length data.
\newblock {\em Queueing Systems}, 55(2):131--138, 2007.

\bibitem{SW2015}
S.~Schweer and C.~Wichelhaus.
\newblock {Nonparametric estimation of the service time distribution in the
  discrete-time GI/G/$\infty$ queue with partial information}.
\newblock {\em Stochastic Processes and their Applications}, 125(1):233--253,
  2015.

\bibitem{patent_SM2002}
D.~E. Smith and M.~Li.
\newblock {Estimating data delays from Poisson probe delays}, Aug.~6 2002.
\newblock US Patent 6,430,160.

\bibitem{book_vdV1998}
A.~W. van~der Vaart.
\newblock {\em Asymptotic Statistics}, volume~3.
\newblock Cambridge University Press, 1998.

\bibitem{Y1992}
D.~D. Yao.
\newblock {On Wolff's PASTA martingale}.
\newblock {\em Operations Research}, 40(3-supplement-2):S352--S355, 1992.

\end{thebibliography}
\end{document}